\documentclass[10pt, reqno]{amsart}
\usepackage{amsmath, esint}
\usepackage[shortlabels]{enumitem}
\usepackage{graphicx}
\usepackage{amssymb,amsthm}
\usepackage[colorlinks=false]{hyperref}
\usepackage{esint}
\usepackage{epstopdf}
\usepackage{color}
\usepackage{url}
\usepackage[top=.9in, left=.9in, right=.9in, bottom=.8in]{geometry}

\numberwithin{equation}{section}

\newtheorem{theorem}{Theorem}[section]
\newtheorem{lemma}[theorem]{Lemma}
\newtheorem{proposition}[theorem]{Proposition}
\newtheorem{corollary}[theorem]{Corollary}

\newtheorem{remark}[theorem]{Remark}

\newtheorem{example}[theorem]{Example}

\def\R{{\mathbb R}}

\def\e{\varepsilon}

\def\w{\omega}
\def\W{\Omega}

\def\om{\omega}

\def\1{\left(}
\def\2{\right)}
\def\3{\left\{}
\def\4{\right\}}
\def\8{\infty}

\usepackage[scr=boondox]{mathalfa} 
\usepackage{cleveref}

\def\grad{\nabla}

\def\tor{\mathrm{tor}}

\def\vpar{\eta}

\def\fv{\mathscr{V}_{\vpar}}
\def\fvbar{\mathscr{V}}

\def\vo{\mathscr{V}}

\def\err{\tau}
\def\Ep{\mathcal{F}_{\err}}

\def\pa{{\partial}}
 \def\nl{\mathscr{h}_f}

\def\na{\nabla}
\def\baryEps{\bar{\e}}

\def\TBF{q}
\def\ccdel{\bar{\delta}}

\begin{document}

\setcounter{secnumdepth}{4}
\setcounter{tocdepth}{1}
\begin{abstract}
For a bounded open set $\Omega \subset \mathbb{R}^n$ with the same volume as the unit ball, the classical Faber-Krahn inequality says that the first Dirichlet eigenvalue $\lambda_1(\Omega)$ of the Laplacian is at least that of the unit ball $B$. We prove that the deficit $\lambda_1(\W)- \lambda_1(B)$ in the Faber-Krahn inequality controls the square of the distance between the resolvent operator $(-\Delta_\Omega)^{-1}$ for the Dirichlet Laplacian on $\Omega$ and the resolvent operator on the nearest unit ball $B(x_\Omega)$. The distance is measured by the operator norm from $L^{\infty}$ to $L^2$. As a main application, we show that the Faber-Krahn deficit  $\lambda_1(\Omega)- \lambda_1(B)$ controls the squared $L^2$ norm between $k$th eigenfunctions on $\W$ and $B(x_\Omega)$ for every $k \in \mathbb{N}.$ In both of these main theorems, the quadratic power is optimal.
\end{abstract}
	\title[Resolvent and Eigenfunction Stability for the Faber-Krahn Inequality]{Quantitative Resolvent  and Eigenfunction Stability for the Faber-Krahn Inequality}
    \author{Mark Allen, Dennis Kriventsov, and Robin Neumayer}
\maketitle

\section{Introduction}
Let  $\Omega \subset \R^n$ be an open bounded set. The Laplace equation with Dirichlet boundary data has  discrete spectrum $0 < \lambda_1(\W) <  \lambda_2(\W) \leq \dots$ of eigenvalues, whose corresponding eigenfunctions $\{ u_{\W, k} \}_{k \in \mathbb{N}}$ solve the eigenvalue problem
\[
\begin{cases}
    -\Delta u_{\W,k} = \lambda_k(\W) \, u_{\W,k} & \text {in } \W\\
  \quad \ \   u_{\W,k} = 0 & \text{on } \partial \W
\end{cases}
\]
and, after normalizing so that $\| u_{\W,k}\|_{L^2(\W)} = 1$ which we shall always do,  form an orthonormal basis for $L^2(\W)$. 

A classical shape optimization problem is the minimization of the first eigenvalue $\lambda_1(\W)$ among domains $\W$ of a fixed volume. Following Lord Rayleigh's 1894 conjecture, in the 1920s Faber \cite{Faber23} and Krahn \cite{Krahn25} independently showed that balls are the only solutions to this variational problem. This leads to the Faber-Krahn inequality:
\begin{equation}\label{eqn: FK intro}
    \lambda_1(\W ) \geq \lambda_1(B)  \quad \text{ if } \quad|\W| = \omega_n\,,
\end{equation}
with equality if and only if $\W = B(x)$ for some $x\in \R^n$ (up to capacity zero sets).
Here $B(x) = \{|y-x|<1\}$ is the unit ball centered at $x$,  $B= B(0)$, and $\om_n = |B|$. The proof of the Faber-Krahn inequality is based on rearrangement and the isoperimetric inequality and  is nowadays classical.

A fundamental question related to the Faber-Krahn inequality is that of stability: if a domain $\W$ with $|\W| = \omega_n$ {\it almost} achieves equality in \eqref{eqn: FK intro}, then is $\W$ {\it almost} a unit ball? The first ``almost'' is naturally measured by the deficit in \eqref{eqn: FK intro}: $\lambda_1(\W)  - \lambda_1(B)$. As for the second,  there are various ways in which one can quantify the proximity of $\W$ to the space of unit balls,  with the most suitable  notion dependent on context.

One classical geometric notion is the Fraenkel asymmetry, $\alpha(\W) = \inf_{x \in \R^n} |\W\Delta B(x)|$, measuring the volume of the symmetric difference between $\W$ and the nearest unit ball. Sharp quantitative stability for the Faber-Krahn inequality with respect to the Fraenkel asymmetry was shown  in the paper \cite{BDV15} of Brasco, De Philippis, and Velichkov;  see also \cite{Bhat01, FMP09} for earlier results.

A different way to quantify how close $\W$ is to the space of unit balls
is {\it spectrally}: Are the eigenvalues and eigenfunctions of $\W$ close to the corresponding eigenvalues and eigenfunctions of the nearest unit ball? 
This notion of distance and corresponding quantitative stability estimates  turn out to be particularly important in applications to harmonic analysis and free boundary problems.

These applications motivated the authors' previous work in \cite{AKN1, AKN2}, where it was  shown that for any open bounded set $\W \subset \R^n$ with $|\W|  =\omega_n$,
\begin{equation}\label{eqn: first efn stability}
\lambda_1(\W ) - \lambda_1(B) \geq c_n \inf_{x \in\R^n} \int |u_{\W,1 } - u_{B(x), 1}|^2 \,.
\end{equation}
Here the eigenfunctions are extended by zero to be defined on $\R^n$.
In the same work, an analogue of \eqref{eqn: first efn stability} was shown on the round sphere $\mathbb{S}^n$. This was used in \cite{AKN3} to establish a rectifiability and uniqueness-of-blowups theorem in terms of the Alt-Caffarelli-Friedman monotonicity formula, with applications to free boundary regularity for certain (e.g. vectorial) free boundary problems. In \cite{FTV1, FTV2}, Fleschler, Tolsa, and Villa applied the analogue of \eqref{eqn: first efn stability} on $\mathbb{S}^n$ as a main tool, alongside  some deep covering arguments, to establish a higher-dimensional version of Carleson's famed $\e^2$-conjecture in harmonic analysis \cite{JTV21}. In \cite{FTV1}, the authors also apply \eqref{eqn: first efn stability} to prove a different form of quantitative stability for the Faber-Krahn inequality with respect to a geometric distance involving the capacity.

The first main result of this paper substantially extends \eqref{eqn: first efn stability}: the Faber-Krahn deficit quantitatively controls not only the distance between first eigenfunctions,  but between  {\it all eigenfunctions}, of $\W$ and the nearest unit ball:

\begin{theorem}\label{thm: eigenfunction estimates} Fix $n \geq 2$ and $k \in \mathbb{N}$. There is a positive constant $c_{n,k}$ depending only on $n$ and $k$ such that 
for any open bounded set $\Omega \subset \R^n$ with $|\Omega| = \omega_n$, there is a point $x_\W \in \R^n$ such that
\begin{equation}\label{eqn: simple efn stability}
\lambda_1(\W ) - \lambda_1(B)  \geq c_{k, n} \int_{\mathbb{R}^n}\left|u_{\Omega, k}-u_{B({x_\W}),k}\right|^2\,.
 \end{equation}
\end{theorem}
As in \eqref{eqn: first efn stability} above, in the statement of Theorem~\ref{thm: eigenfunction estimates}, we extend each $u_{\W,k} \in H^1_0(\W)$ by zero to be in $H^1_0(\R^n)$, and likewise for eigenfunctions on the ball.  When $\lambda_k(B)$ has multiplicity greater than $1$, $u_{B({x_\W}),k}$ denotes {\it some} eigenfunction in the $\lambda_k(B)$-eigenspace with $L^2$ norm equal to $1$, since of course any choice of ordered basis for this eigenspace is not canonical. The quadratic power on the $L^2$-norm on the right-hand side of \eqref{eqn: simple efn stability} is optimal; see section~\ref{ssec: examples}.

One can view Theorem~\ref{thm: eigenfunction estimates} as the counterpart for eigenfunctions of a recent result of Bucur, Lamboley, Nahon, and Prunier \cite{blnp23} for eigenvalues. There, the authors show  that
for each $\Omega \subset \R^n$ with $|\Omega| = \omega_n$ and  $k \in \mathbb{N}$, the following sharp quantitative estimates hold: if $\lambda_k(B)$ is simple, then 
\begin{align}\label{eqn: Bucur simple}
 \lambda_1(\W ) - \lambda_1(B)   & \geq c_{n,k}  |\lambda_k(\W)-\lambda_k(B)|,   \\
 \intertext{
while if $\lambda_k(B)$ has multiplicity---say $\lambda_{k-1}(B) <  \lambda_k(B)=\ldots=\lambda_{\ell}(B) < \lambda_{\ell+1}(B)$---then 
 }
 \label{eqn: Bucur multiplicity}
 \lambda_1(\W ) - \lambda_1(B)    &\geq c_{n,k}    \Big| \sum_{i=k}^\ell ( \lambda_i (\W) - \lambda_i(B) \Big| \,.
\end{align}
The linear powers in \eqref{eqn: Bucur simple} and \eqref{eqn: Bucur multiplicity} are stronger than one typically expects in stability estimates (and would be false for Theorem~\ref{thm: eigenfunction estimates}), and have to do with the fact that the ball is critical for the functionals on the right-hand sides, as well as the left-hand sides, of \eqref{eqn: Bucur simple} and \eqref{eqn: Bucur multiplicity}. 

Theorem~\ref{thm: eigenfunction estimates} is proven as an application of the other main theorem of the paper: sharp  quantitative estimates for the resolvent operator. To state the result we first introduce a bit of terminology. For an open bounded set $\W \subset \R^n$ and a function $f \in H^{-1}(\R^n)$, let $u_{\W}^f  \in H^{1}_0(\W)$ be the unique (weak) solution to the Poisson equation 
\begin{equation}\label{eqn: PDE with f RHS}
	\begin{cases}
	\hfill	-\Delta u^f_{\W}  = f & \text{ in } \W,\\
	\hfill	u_{\W}^f =0 & \text{ on } \partial \W.
	\end{cases}
\end{equation}
As above for eigenfunctions, we extend $u_\W^f$ by zero to be defined as a function in $H^1_0(\R^n)$. The {\it resolvent operator}
$$(-\Delta)_\W^{-1} : H^{-1}(\R^n) \to H^1_0(\R^n)$$ associates to $f \in H^{-1}(\R^n)$ the unique solution to \eqref{eqn: PDE with f RHS}. As an operator from $L^{\infty}$ to $L^2$, the resolvent operator of $\W$ is close to the resolvent operator of the nearest unit ball, quantitatively in terms of the Faber-Krahn deficit:
\begin{theorem}\label{thm: resolvent}\label{thm: main}
Fix $n\geq 2$. There is a constant $c_{n}>0$ depending on $n$ such that  for any bounded open set $\W \subset \R^n$ with $|\W| = \omega_n$,
\begin{align}
    \label{eqn: resolvent estimate}
  \lambda_1(\W ) - \lambda_1(B)  & \geq c_{n} \left\| (-\Delta)_{\W}^{-1} -(-\Delta)_{B(x_\W)}^{-1}\right\|^2_{L^{\infty}\to L^2}\,\\
\intertext{for a point $x_\W \in \R^n$. In other words, for any $f:\R^n \to \R$ with $\|f\|_{L^{\infty}(\R^n)}\leq 1,$ we have}
\label{eqn: quant stability main}
 \lambda_1(\W ) - \lambda_1(B)  & \geq c_{n}\int_{\R^n} \big|u_{\W}^f - u_{B(x_\W)}^f\big|^2 \,.   
\end{align}
\end{theorem}
To establish \eqref{eqn: resolvent estimate} (and in turn Theorem~\ref{thm: eigenfunction estimates}) from \eqref{eqn: quant stability main}, it is essential that the constant $c_{n}$ and the ball center  $x_\W$ in \eqref{eqn: quant stability main} do not depend on the function $f$. The point $x_\W$  is a truncated version of the barycenter of $\W$ defined  in section~\ref{sec: truncated barycenter}.

A {\it  qualitative} version of \eqref{eqn: quant stability main}, with the modulus of continuity depending on $f$, can be deduced by combining a theorem of \v{S}ver\'{a}k  \cite[Theorem 3.2.5]{hp18}, the Kohler-Jobin inequality (see Lemma~\ref{lem: KJ consequence}), and a standard concentration compactness argument. To our knowledge, Theorem~\ref{thm: main} is the first {\it quantitative} estimate for the  resolvent operator.

The quadratic powers in \eqref{eqn: resolvent estimate} and \eqref{eqn: quant stability main} are optimal and cannot be replaced by a smaller exponent, as shown in section~\ref{ssec: examples}. In this sense Theorem~\ref{thm: main} is sharp. A related question is whether in \eqref{eqn: resolvent estimate} the operator norm from $L^{\infty}$ to $L^2$ can be replaced with  the operator norm from $X$ to $Y$ for other function spaces $X\supset L^{\infty}$ and $Y \subset L^2$. We construct examples in section~\ref{ssec: examples} exploring the extent to which one might expect to prove a ``strong form stability'' improvement in this direction. For instance, we show $L^2$ cannot be replaced by $H^1$, and $L^{\infty}$ cannot be replaced by $L^p_{loc}$ for $p<n/2.$ However, we have no examples ruling out the $X= L^p$ with $p>n$, which leaves the following natural open question:
\smallskip

\noindent{\bf Open Question.} {\it Does Theorem~\ref{thm: main} hold with the norm $L^p$ in place of the norm  $L^{\infty}$?
}

\smallskip
 We expect the answer to this open question to be affirmative. Many of the arguments work with $f \in L^p$ with $p>n$.

As mentioned above, we also showed \eqref{eqn: first efn stability} on the round sphere and hyperbolic space in \cite{AKN1, AKN2}. Similarly, Theorems~\ref{thm: eigenfunction estimates} and \ref{thm: main} can be shown to hold on these spaces. Nearly  all of the arguments in this paper can be adapted to these settings, with the only nontrivial modifications being (a) the spectral gap of Lemma~\ref{thm: spectral gap} below should be replaced by a spectral gap argument by adapting the proof of \cite[Theorem 3.11]{AKN1} as well as the arguments in \cite{p24}, and (b) a minimizer of the penalized variational problem in Theorem~\ref{thm: summary for minimizers} below should be shown to satisfy the volume constraint by using the Euler-Lagrange equation as in \cite[Proposition 9.5]{AKN2} instead of the simpler scaling argument of Lemma~\ref{lem: volume constraint lemma}. To aid in the clarity of presentation, we have opted to present all the proofs here for domains on Euclidean space.

\label{thm: spectral gap}

Beyond the works mentioned above, there is a rich literature on quantitative stability for geometric inequalities involving PDE-driven shape functionals like $\lambda_1$, with a number of exciting developments in the past decade or two. 
Following the work of Brasco-De Philippis-Velichkov \cite{BDV15}, similar methods were used to establish sharp quantitative stability for the $p$-Faber-Krahn \cite{FuscoZhang}, isocapacitary  \cite{DPMMCapacity}, and $p$-isocapacitary \cite{Mukoseeva+2021} inequalities.
Through quite different methods, quantitative stability has been established for spectral shape optimization problems with Neumann \cite{BrascoPratelli} and Robin \cite{BucurRobinLaplacian} boundary conditions, and on surfaces \cite{MR4859580} (both globally and within a fixed conformal class). The recent papers \cite{acampora2025sharpquantitativetalentisinequality, amato2024talenticomparisonresultquantitative} address stability for Talenti's rearrangement inequality. This sampling of results is by no means exhaustive, and we refer the reader to \cite{BrascoDePSurvey} for an overview of spectral inequalities in quantitative form.

\subsection*{Discussion of the Proofs}
Theorem~\ref{thm: eigenfunction estimates} will be established as a consequence of Theorem~\ref{thm: main}. The basic idea in the case when $k\in \mathbb{N}$ corresponds to a simple eigenvalue $\lambda_k(B)$ of the ball (and $x_\W = 0$ for simplicity) is as follows. We apply Theorem~\ref{thm: main} with a normalization of $f = \lambda_k(B) u_{B,k}$ to find $u_B^f = u_{B, k}$ is quantitatively close to $u_\W^f$. Then, expanding both functions in the basis of eigenfunctions on $\W$, we are able to quantitatively estimate their Fourier coefficients in order to show $u_{B,k}$ is quantitatively close to $u_{\W ,k}$.

Although the Faber-Krahn inequality is a consequence of the isoperimetric inequality, it is thus far unclear how to deduce sharp quantitative stability for the Faber-Krahn inequality (even with respect to the Fraenkel asymmetry) from quantitative stability for the isoperimetric inequality \cite{FuMaPr, FiMaPr, CiLe12}; see \cite{Melas, HansenNadir94, FMP09} for non-sharp or conditional results and the survey \cite{BrascoDePSurvey} for a discussion of the limitations of this approach.

Instead, we use the selection principle framework introduced by Cicalese and Leonardi in \cite{CiLe12}. This robust tool, which has found application in myriad settings, e.g. \cite{FuJu, BDF, BDS15, N16, DPMMCapacity, Mukoseeva+2021, FuscoZhang}, allows one to prove quantitative stability for a given geometric inequality from {\it linear stability} together with {\it regularity estimates} for a suitable penalized functional.

Brasco-De Philippis-Velichkov's quantitative Faber-Krahn inequality \cite{BDV15} also used a selection principle scheme. While the linear stability portion of our proof follows in a straightforward manner from theirs, the penalized functional and corresponding regularity estimates that are needed to obtain the resolvent estimates \eqref{eqn: resolvent estimate}  differ fundamentally  from those needed to control the Fraenkel asymmetry in \cite{BDV15}. 
In this respect, a major new feature of this paper is the development of techniques introduced in our earlier paper \cite{AKN2} to study the regularity of minimizers of Alt-Caffarelli-type functionals involving a type of a-priori same-order perturbation described below.

First, several initial reductions simplify the proof of Theorem~\ref{thm: main} before running the selection principle. It suffices to show the estimate \eqref{eqn: quant stability main} in the case when
 \begin{equation}\label{eqn: f hyp after reductions}
     f \in C_0^\infty(\R^n), \qquad f\geq 0 , \qquad \|f\|_{L^{\infty}(\R^n)}\leq 1,
 \end{equation}
 as shown in section~\ref{sec: simplify}. Moreover, by the Kohler-Jobin inequality, it suffices to show Theorem~\ref{thm: main} with  the Faber-Krahn deficit $\lambda_1(\W) -\lambda_1(B)$ replaced by the Saint-Venant deficit $\tor(\W)-\tor(B)$. Here $\tor(\W)$ is
the torsional rigidity of $\W$,  defined as
	\begin{equation}\label{e: torsion def}
 \tor(\W) =\inf  \left\{ \int \frac{1}{2} |\grad w|^2 - w \, dx \ : \ w \in H^1_0(\W)\right\}.
	\end{equation}
The sign convention adopted here differs from the one used in some other references (and is minus the physical quantity). The torsional rigidity is  a non-positive quantity, and among sets of a fixed volume, it is uniquely (up to sets of capacity zero) minimized by balls. The resulting inequality is the {\it  Saint-Venant inequality}:
	\begin{equation}\label{e: saint venant}
    \tor(\W)\geq \tor(B) \qquad \text{ if } \quad|\W| = \omega_n\,.
    \end{equation}
The trick of replacing the Faber-Krahn deficit with the Saint-Venant deficit (and more generally, replacing $\lambda_1$ with the more-regularizing shape functional $\tor$) is well known and was used, for instance, in \cite{BDV15, blnp23}.

In view of these reductions, to prove Theorem~\ref{thm: main} it suffices to show there is a constant $c_{n}>0$ such that 
\begin{align}
\label{eqn: quant stability main 3}
\tor(\W ) - \tor_1(B)  & \geq c_{n}\int_{\R^n} \big|u_{\W}^f - u_{B(x_\W)}^f\big|^2 \,
\end{align}
for any open bounded set $\W \subset \R^n$ with $|\W| = \om_n$ and for any  $f:\R^n \to \R$ satisfying \eqref{eqn: f hyp after reductions}. Here, as in Theorems~\ref{thm: eigenfunction estimates} and \ref{thm: main} above, $x_\W$  denotes the truncated barycenter, defined in section~\ref{sec: truncated barycenter}. For sets of diameter at most $2$, like the minimizer in Theorem~\ref{thm: summary for minimizers} below, the truncated barycenter agrees with classical barycenter $\fint_\W x\,dx$. Of course, it is essential that  the constant $c_{n}$ in \eqref{eqn: quant stability main 3} does not depend on the smoothness of $f$ beyond its $L^{\infty}$ norm or on its compact support.

By using a selection principle reduction (carried out in section~\ref{sec: proofs of main theorems}), the main analytical challenge to prove \eqref{eqn: quant stability main 3} is to establish the existence and regularity of minimizers of the penalized functional 
\begin{equation}\label{eqn: energy functional intro}
		 \mathcal{E}(\W) = \tor(\W) + \tau \, 
        h\Big(\int |u_\W^f - u_{B(x_\W)}^f|^2  \Big)  
	\end{equation} 
among sets of volume $\om_n$. Here $h(t) =\sqrt{a^2 +(a - t )^2}$ for a given number $a\in(0,1]$ is (up to a constant) a regularization of the function $|t-a|$, and $\tau>0$ is a parameter that will be chosen to be small.
We show the following result for the volume-constrained minimization of the functional $\mathcal{E}$ in \eqref{eqn: energy functional intro}.

\begin{theorem}\label{thm: summary for minimizers} \label{summary ex and reg theorem}
Fix $n\geq 2$, $\alpha\in (0,1)$, and $\delta_0\in (0,1]$. There exists $\bar{\err}_0= \bar{\err}_0(n,\alpha, \delta_0)>0$
such that the following holds. Fix $a\in (0,1]$, $\tau \in (0,\bar{\tau}]$, and $f$ as in \eqref{eqn: f hyp after reductions}. The variational problem 
\begin{equation}
	\label{eqn: main min prob}
\inf \{\mathcal{E}(\W) : \W \subset \R^n \text{ open, bounded, } |\W|= \omega_n\}
\end{equation}
admits a minimizer $U$. There is a $C^{2, \alpha}$ function $\xi: \partial B(x_\W) \to \R$, with  $\| \xi \|_{C^{1,\alpha}(\partial B(x_\W))} \leq \delta_0$ such that 
 \[
 \partial U = \{ (1+\xi(x))x : x \in \partial B(x_\W) \}\,. 
 \]
\end{theorem}

{Since we assume in \eqref{eqn: f hyp after reductions} that $f$ is smooth, we obtain $\xi \in C^{2,\alpha}$. However, since we do not assume bounds on the derivative of $f$ we only obtain the first order bound $\|\xi\|_{C^{1,\alpha}}\leq \delta_0$, but utilizing the recent ideas in \cite{p24} to obtain Lemma \ref{thm: spectral gap} this will be sufficient for our methods.} 

Minimization of the first term in the energy $\mathcal{E}$, the torsional rigidity, behaves much like the minimization of the Alt-Caffarelli (also called one-phase Bernoulli) functional. Regularity theory for minimizers of this functional was first developed by Alt and Caffarelli in \cite{AC81} and its variants have since been a central topic in free boundary problems. 

What introduces new challenges to the analysis of the functional $\mathcal{E}$ is the term $\nl(\W):= h(\int |u_\W^f - u_{B(x_\W)}^f|^2)$, which is the same order as $\tor(\W)$, has poor lower semicontinuity properties, and in the Euler-Lagrange equation contributes a same-order term to the one coming from $\tor(\W)$ without a favorable sign.

Standard methods of proving existence fall short for the minimization problem \eqref{eqn: main min prob}; for instance 
\cite{BuDalMaso} cannot be applied as the functional $\mathcal{E}$ is not decreasing under set inclusions. Instead we follow the approach of our earlier paper \cite{AKN2}: we prove a priori estimates for inward and outward minimizers and construct a ``good'' minimizing sequence consisting of outward minimizers along which the functional actually is lower semicontinuous. This approach has the additional advantage that we work entirely in the framework of open sets and  there is no  need to pass to quasi-open sets.

In terms of regularity, since the term $\nl(\W)$ scales at the same order as $\tor(\W)$, one cannot apply the theory of almost-minimizers  of Alt-Caffarelli-type functionals \cite{DET}. A formal  computation  assuming $\partial \W$ is sufficiently regular shows that the Euler Lagrange equation is the vectorial Alt-Caffarelli problem
\begin{equation}
    \label{eqn: intro formal EL}
    \begin{cases}
\ \        -\Delta w_\W = 1, \ \  -\Delta u_\W^f = f, \ \ -\Delta p_\W = u_\W^f -u_{B(x_\W)}^f & \text{ on } \Omega\\
    \ \  \ \ \ \  \  w_\W = u_\W^f = p_\W = 0 & \text{ on } \partial \W,\\
       \ \  -|\na w_\W|^2 + C_\W \tau ( |\nabla p_\W| \, |\nabla u_\W^f| - g_\W)  = \text{const} & \text{ on }\partial \W\,.
    \end{cases}
\end{equation}
Here $g_\W:\R^n \to \R$ is a smooth bounded function depending on $\W$ (through the solution of another auxiliary PDE) and  $C_\W = h'(\int |u_\W^f - u_{B(x_\W)}^f|^2)>0$.
Significant difficulties arise in rigorously deriving the stationarity condition \eqref{eqn: intro formal EL}. Stemming from the genuinely nonlocal nature of the functional $\nl(\W)$---which leads, in particular, to the presence of the auxiliary potential $p_\W$ in \eqref{eqn: intro formal EL} that solves a Poisson equation with right-hand side depending on $\W$---it is not clear how to run the type of blowup argument used in \cite{blnp23} to show $\W$ is a viscosity solution of \eqref{eqn: intro formal EL}. 

Instead, we use careful regularization arguments to establish a distributional form of \eqref{eqn: intro formal EL} and then derive a pointwise form by plugging in suitable test functions (with the final line of \eqref{eqn: intro formal EL} understood in the sense of nontangential limits). On a broad level, this approach follows the one of our earlier paper \cite{AKN2}, but we are able to refine a number of our earlier arguments, leading to a significantly simpler proof.

Once we are able to rigorously establish that \eqref{eqn: intro formal EL} holds (and along the way, prove the $\W$ is an NTA domain), we can apply the inhomogeneous boundary Harnack theorem \cite{aks23} shown by the first two authors and Shahgholian. This allows us to recast \eqref{eqn: intro formal EL} as a standard one phase problem and apply the classical theory of Alt and Caffarelli \cite{AC81}.

\subsection*{Organization of the Paper}

The paper  essentially splits into two disjoint parts. After recalling notation and classical facts from the literature in section~\ref{sec: prelim},  we dedicate the first part of the paper, sections~\ref{sec: more preliminaries}--\ref{sec: proofs of main theorems}, to proving Theorems~\ref{thm: eigenfunction estimates} and \ref{thm: main} assuming Theorem~\ref{thm: summary for minimizers}. Specifically, we introduce the truncated barycenter and establish several reductions for our main theorems in section~\ref{sec: more preliminaries}, show  Theorem~\ref{thm: main} in the class of nearly spherical sets in section~\ref{sec: linear stability}, and in section~\ref{sec: proofs of main theorems} we conclude the proofs of Theorems~\ref{thm: eigenfunction estimates} and \ref{thm: main} and present examples demonstrating the sharpness of the theorems. 

The second part of the paper, sections~\ref{sec: a priori and existence} and \ref{sec: FB2},  is dedicated to proving Theorem~\ref{thm: summary for minimizers} using techniques from  free boundary theory. In section~\ref{sec: a priori and existence} we establish the existence of minimizers of a relaxed version of the minimization problem \eqref{eqn: main min prob}, as well as basic properties of minimizers like Lipschitz regularity and nondegeneracy of the torsion function. In section~\ref{sec: FB2}, we compute the Euler-Lagrange equation satisfied by minimizers and use this to show boundary regularity and that the minimizer of the relaxed problem satisfies the volume constraint $|U| =\om_n$, making it a minimizer of the original minimization problem \eqref{eqn: main min prob}.  
\medskip

\noindent{\it Acknowledgments.} MA is supported by Simons award 637757.
DK is supported by NSF grant DMS-2247096.
RN is supported by NSF grants DMS-2340195, DMS-2155054, and DMS-2342349. The authors are very grateful to Jimmy Lamboley for showing us Prunier's paper \cite{p24}; adapting an argument therein allowed us to substantially weaken the hypotheses of our main theorems.

\section{Preliminaries}\label{sec: prelim}
We introduce some known preliminary facts and notation that will be used throughout the paper. For the remainder of the paper, we fix $n$ to be an integer with $n\geq 2.$
 We write $B_r(x)$ to denote a ball of radius $r$ centered at $x$, and when $r=1$ or $x=0$ we often omit the dependence in the notation, writing, e.g. $B_r = B_r(0)$ and $B(x) = B_1(x).$

\subsection{Basic facts about the torsional rigidity}
Let $\W\subset \R^n$ be an open bounded domain. The infimum in the variational problem \eqref{e: torsion def} is uniquely achieved by the {\it torsion function} $w_\W \in H^1_0(\W)$ solving
\begin{equation}
	\label{e: torsion eqn}
\begin{cases}
\hfill	-\Delta w_\W = 1 & \text{ in }\W\\
\hfill	w_\W = 0 & \text{ on } \partial \W.
\end{cases}
\end{equation}
Of course, \eqref{e: torsion eqn} is a special case of  \eqref{eqn: PDE with f RHS} with $f \equiv 1$, so that $w_\W = u_\W^1 $ in the notation introduced there. As for eigenfunctions and solutions to \eqref{eqn: PDE with f RHS}, we extend $w_\W$ by zero to be defined in $H^1_0(\R^n)$, using the same notation $w_\W$ to denote the extended function. The distributional Laplacian on $w_\W$ (defined by acting on $\phi \in C^\infty_c(\R^n)$ by $\Delta w_\W(\phi) = \int w_\W \Delta \phi$) satisfies
\begin{equation*}
    -\Delta w_\W \leq 1 \quad\text{ on } \R^n\, .
\end{equation*}
The same is true for any solution of \eqref{eqn: PDE with f RHS} with $\| f\|_{L^\infty(\R^n)}\leq 1.$
Multiplying \eqref{e: torsion eqn} by $w_\W$ and integrating over $\R^n$, we obtain the following expressions for the torsional rigidity:
\begin{equation}\label{e: torsion and dir energy}
\tor(\W) = -\frac{1}{2} \int |\na w_\W|^2 = -\frac{1}{2} \int w_\W.
\end{equation}
If $\W'\subset \W,$ then any competitor $w $ for $\tor(\W')$ is also a competitor for $\tor (\W)$ and thus $\tor(\W) \leq \tor(\W')$. Direct computation shows $\tor\left(r \W \right) = r^{n+2} \tor (\Omega)$ and that 
the torsion function on a ball is 
\begin{equation}
 	\label{eqn: tor function on a ball}
 w_{B_r(x_0)} = \frac{1}{2n}(r^2 - |x-x_0|^2 ) \chi_{B_r(x_0)}. 
  \end{equation}
In particular, $-\tor(B_1)  = \frac{\omega_n}{2n(n+2)}$ and $\| w_{B_r(x_0)}\|_{L^\infty(\R^n)} =\frac{r^2}{2n}$.

By the maximum principle, $\| u_\W^f\|_{L^\infty}  \leq \| w_\W\|_{L^\infty}$, while
 Talenti's theorem\footnote{if $u=u_\W^f$ solves \eqref{eqn: PDE with f RHS} 
and $v= u_{B_r}^{f^*}$ solves \eqref{eqn: PDE with f RHS} on $B_r$ 
 where $r>0$ is again chosen to that $|\W| = |B_r|$ and $f^*$ is the symmetric decreasing rearrangement of $f$, then $u^* \leq v$ pointwise on $B_r$. } \cite{Talenti1976} guarantees that $\| w_\W\|_{L^\infty} \leq \|w_{B_r}\|_{L^\infty}$, where $r$ is such that $|B_r|=|\W|$. 
Thus, for any bounded open set $\W \subset \R^n$ with $|\W| \leq 2 \omega_n$, 
\begin{equation}
    \label{eqn: bound for tor fn}
\| w_\W\|_{L^\infty(\R^n)} \leq \frac{2}{n}, \qquad \| w_\W\|_{L^2(\R^n)}^2 \leq \frac{4\om_n}{n^2} \leq 1
\end{equation}

The Kohler-Jobin inequality \cite{KJ1, KJ2, KJ3}
says that among sets $\Omega \subset \mathbb{R}^n$ of a fixed torsional rigidity $\operatorname{tor}(\Omega)$, balls minimize the first Dirichlet eigenvalue $\lambda_1(\Omega)$. 
A well-known consequence of the Kohler Jobin inequality (see \cite[Cor. 3.18]{AKN1} or \cite[Prop. 2.1]{BDV15}) is that the deficit in the Faber-Krahn inequality linearly controls the definit in the Saint-Venant inequality:
\begin{lemma}\label{lem: KJ consequence}
 There exists a constant $C_n$ so that for any open set $\W\subset \R^n$ with $|\W|=\om_n$, 
\[
C_n\left(\lambda_1(\Omega)-\lambda_1(B)\right) \geq \operatorname{tor}(\Omega)-\tor(B)
\]
\end{lemma}

\subsection{Qualitative stability and relaxing the volume constraint}\label{ssec: qual stab}
As is typical in geometric variational problems, it is easier to replace a volume constraint with a volume penalization term. 
Following \cite{aac86} and \cite{BDV15}, we remove the volume constraint in \eqref{eqn: main min prob} in favor of adding to the energy a volume penalization term
\begin{equation} \label{eqn: def vol penalization fixed param}
	\mathscr{V}(t)=\fv(t) = \begin{cases}
		\vpar (t - \om_n) & t \leq \om_n\\
		\frac{1}{\vpar}(t - \om_n) & t > \om_n\,,
	\end{cases}
\end{equation}
where $\vpar>0$ is a parameter to be fixed later on.
By choosing  $\vpar$ sufficiently small, a set $\W$ whose volume exceeds the volume $\om_n$ of a unit ball will have a large energy contribution coming from the term $\fv(|\W|)$, and minimizing $\tor(\cdot)$ among sets of volume $\w_n$ is equivalent to the unconstrained problem of minimizing $\tor(\cdot)  + \fv(|\cdot|)$ among open bounded sets of volume at most $2\omega_n$:
\begin{lemma}
    \label{prop: base summary}
 There exists $\vpar_0 = \vpar_0(n)$ such that for $\vpar \leq \vpar_0$,  unit balls are the unique minimizers of the energy $\tor(\W)  + \fv(|\W|)$ among all open bounded sets $ \W \subset \R^n$ with $|\W | \leq 2 \omega_n$.
 Moreover, if $\W$ is an open bounded set with $|\W | \leq 2 \omega_n$ and  $\tor(\W)  + \fv(|\W|) \leq \frac{5}{4}\tor(B_1)$, then $|\W| \geq \omega_n/2$.
\end{lemma} 
The proof of Lemma~\ref{prop: base summary} is elementary and nearly identical to the proof of \cite[Lemma 4.5]{BDV15}, so we omit it.
The assumed hard volume constraint $|\W| \leq2 \om_n$ in Lemma~\ref{prop: base summary} is necessary, since for any fixed $\eta >0$,  a scaling computation shows that  $\tor(B_R) + \fv(|\W|) \to -\infty$ as $R \to \infty$ .  An alternative remedy, adopted in \cite{BDV15}, would be to minimize among sets contained in $B_R$, but this has the drawback of introducing a dependence of $\eta$ on $R$.
 
Next, a standard concentration compactness argument establishes the following qualitative stability result for the energy $\tor(\cdot) +  \mathscr{V}_\eta(|\cdot|)$: as $\tor(\W) +  \mathscr{V}_\eta(|\W|)$ tends to its infimum $\tor(B_1)$, $\Omega$ converges in $L^1$ to a ball:
 \begin{proposition}[Qualitative stability]
 	\label{lemma: qual stability}
 	Fix $\e>0$. There exist $\eta_1={\eta}_1(n)\in (0,\eta_0]$ and $\ccdel = \ccdel(n,\e)>0$ such that the following holds. Suppose  $\eta \in (0, \eta_1]$ and $\W\subset\R^n$ is an open bounded set with $|\W|\leq 2\omega_n$ and 
 	\begin{equation}
 		\label{eqn: small base energy}
 	\tor(\W) +\fv(|\W|) \leq \tor(B_1) + \ccdel.
 	 	\end{equation}
 	There exists $x_1$ such that
 	\begin{equation}
 		\label{eqn: qual stab}
 	|\W \Delta B(x_1) | \leq \frac{\omega_n\e}{2}.
 		\end{equation}
 \end{proposition}

\subsection{Geometric measure theoretic preliminaries} Let us recall some geometric measure theoretic definitions and facts that will be used to establish initial regularity and to compute the Euler-Lagrange equation of minimizers of the (relaxed version of the) minimization problem \eqref{eqn: main min prob}.

For $K>0$,
A bounded open set $\W\subset \R^n$ is called a {\it non-tangentially accessible (NTA) domain}  with NTA constant $K$ if the following properties hold: 
\begin{enumerate}
    \item  For each $x \in \partial \W$ and $r \in (0,1)$ there are points $y_{\textrm{i}}$ and $y_{\textrm{o}}$ such that 
    \[
    B_{r/K}(y_{\textrm{i}}) \subset B_r(x) \cap \W  \qquad \text{ and } \qquad B_{r/K}(y_{\textrm{o}}) \subset B_r(x) \setminus \W .
    \]
    \item  For each  $x,y \in  \W$, there is a curve $\gamma : [0,1] \to \W$ with $\gamma(0)=x$ and $\gamma(1)=y$ such that $\ell(\gamma([0,1]) \leq K |x-y|$ and 
    $$\text{dist}(\gamma(t),\partial \W) \geq \frac{1}{K} \min\{ \ell(\gamma([0,t])), \ell(\gamma([t,1]).
    $$
\end{enumerate}
\begin{remark}
    {\rm 
    Jerison-Kenig's original definition of NTA domain  \cite{JKNTA} takes a slightly different form but is equivalent to the one given here; see \cite[Theorem 2.15]{ahmnt17}.
    Some references, e.g., the classical texts \cite{k94, CaffSalsa}, define NTA domains with the exterior corkscrew condition $B_{r/K}(y_{\textrm{o}}) \subset B_r(x) \setminus \W$ replaced by the weaker condition that $\Omega$ satisfies uniform upper volume density estimates at each $x\in \partial \W$. Since this definition is strictly weaker, the results cited below from \cite{k94}  hold for NTA domains as defined above. }
\end{remark}

For   a open bounded set $\W \subset \R^n$,  $x \in \partial \W$, and $\beta>0$, define the non-tangential approach region
\[
\Gamma_\beta(x)  = \{ y \in \W : |y-x| \leq (1+\beta) \text{dist}(y, \partial \W)\}.
\]
Given $u \in C(\Omega)$, the associated non-tangential maximal function $u^*:\pa \W \to \R$ (which depends on $\beta$) is defined by $u^*(y) = \sup_{x \in \Gamma_\beta(y)}|u(x)|$. The function $u$  is said to {\it converge non-tangentially to $a \in \R$} at $x \in \partial \Omega$ if, for every $\beta >0$, 
\begin{equation}\label{def: NTlimit 1}
\lim_{\substack{y \to x ,\\ 
y \in \Gamma_\beta(x)  }} u(y) = a\,.
\end{equation}

The next lemma gives us a Poission kernel representation for certain solutions of the Dirichlet problem when $\W$ is sufficiently regular. The statement is not the most general one possible, but will be sufficient for our needs.
\begin{lemma}\label{lem: poisson kernel}
    Let $\W \subset \R^n$ be an NTA domain and assume there are  constants $c,r_0 >0$ such that for each $x \in \partial \W$ and $r\in (0,r_0]$,
\begin{equation}\label{eqn: prelim ahlfors}
   c \, r^{n-1} \leq   \mathcal{H}^{n-1} (\partial \W \cap B_r(x) ) \leq \frac{1}{c} r^{n-1}\,.
\end{equation}
Then for each $x \in \W$, there is a function $K(x,y)$ defined  for $\mathcal{H}^{n-1}$-a.e. $y \in \partial \W$ such that the following holds. 
For $g\in L^\infty(\partial \W,  \mathcal{H}^{n-1})$, there is a unique harmonic function $u \in C^\infty(\W)$ with $u^* \in L^\infty(\partial \W, d\mathcal{H}^{n-1})$ for some $\beta>0$ such that $u$ converges non-tangentially to $g$ for $\mathcal{H}^{n-1}$-a.e. $y \in \partial \W$. For each $x \in \W$, $u(x)$ is given by
$$
u(x) = \int_{\partial \W} g(y){K}(x,y) \, d \mathcal{H}^{n-1}(y)\,.
$$
\end{lemma}
\begin{proof}
Suppose $\W$ is an NTA domain. Let $\{\w^x\}_{x\in \W}$ be the harmonic measure for $\W$, and for a fixed point $x_0 \in \W$, set $\w= \w^{x_0}$. By [Lemma 1.2.7]\cite{k94}, the measures $\w^x$ and $\w$ are mutually absolutely continuous for any $x \in \W$, so 
the Radon-Nykodym  derivative $\tilde{K}(x,y)$ of $\w^x$ with respect to $\w$
 exists for $\w$-a.e. $y$. By \cite[Corollary 1.4.3, Theorem 1.4.4]{k94}, for any $g \in  L^\infty(\pa \W, d \om)$, the function 
$u(x) = \int_{\partial \W} g(y)\tilde{K}(x,y) \, d\om(y)$
is the unique harmonic function on $\W$ with $u^* \in L^\infty(\partial \W, d\om)$ such that  $u(x)$ converges non-tangentially to $g$ for $\om$-a.e. $y \in \partial \W$. 

Now additionally assume that $\W$ satisfies \eqref{eqn: prelim ahlfors}. By \cite[Theorem 2]{DavidJerison}, the harmonic measure $\w$ and surface measure $\sigma:= \mathcal{H}^{n-1}\llcorner \partial \W$ are mutually absolutely continuous (in fact, quantitatively so). Thus, in each statement above, we may replace $\om$ by $\sigma$. This proves the lemma.
\end{proof}

A Borel set $\W\subset \R^n$ is a {\it set of finite perimeter} in $\R^n$ if 
\[\sup\Big\{ \int_\W \text{div}\,  T  : T\in C^1_c(\R^n; \R^n) , |T|\leq 1\Big\}<\infty.\]
If the topological boundary $\partial \W$ of measurable set $\W$ has  $\mathcal{H}^{n-1}(\partial \Omega) <\infty$, then $\W$ is a set of finite perimeter.
The reduced boundary $\partial^*\W \subset \partial \W$ of a set of finite perimeter $\W$ is the set of points such that 
\[
\nu_\W(x) := \lim_{r\to 0^+} \frac{D\chi_\W(B_r(x))}{|D\chi_\W| (B_r(x))} \qquad \text{ exists and has modulus }1.
\]
For every $x\in \partial^*\W$, the rescaled sets $\W_{x,r}:=(\W-x)/r$ converge in $L^1_{loc}$ to the half-space $H_x:= \{y \in \R^n : y \cdot \nu_\W(x)\} \leq 0$ as $r \to 0^+$; see \cite[Theorem 15.5]{MaggiBook}.
 If additionally there are positive constants $c, r_0>0$ such that
\[
  c\, \om_nr^n  \leq |B_r(x) \cap \W| \leq (1-c)\om_n r^n
\]
  for all $x \in \partial \W$ and $r\in (0,r_0]$, then $\mathcal{H}^{n-1}(\partial \W \setminus \partial^*\W) =0$ by \cite[Theorem 16.2]{MaggiBook}, and for any $x \in \partial^*\W$ and $\e>0$, we have $d_H(\partial \W \cap B_r(x) , \partial H_{x} \cap B_r(x)) < \e r$ for $r$ sufficiently small, and thus for $r$ small enough,
\begin{equation}\label{eqn: cone}
    \partial\W \cap B_r(x)\subset \{ y \in B_r(x): |y\cdot \nu_\W(x)| <\e|y|\} 
\end{equation}

The divergence theorem holds for vector fields $Z \in C^1_c(\R^n ; \R^n)$ on a set of finite perimeter $\W$. In section~\ref{sec: FB2}, we will need to apply a different form of the divergence theorem, which holds assuming more regularity on $\W$ but less regularity for the vector fields. The following is a special case of \cite[Theorem 2.3.1]{HMT}. 
\begin{lemma}\label{l:fancy div thm}
Let $\Omega \subset \mathbb{R}^{n}$ be a bounded open set of finite perimeter with $\mathcal{H}^{n-1}(\partial \W \setminus \partial^*\W)=0$ and assume there are constants $c, r_0>0 $ such that
\eqref{eqn: prelim ahlfors} holds
for all  $r \in (0,r_0]$ and $x \in \partial \W$.
Then
\[
\int_{\Omega} \operatorname{div} Z d x =\int_{\partial \Omega} Z \cdot \nu_\W \,  d\mathcal{H}^{n-1}
\]
for any vector field $Z \in C(\Omega)$ with $\operatorname{div} v \in L^1(\Omega)$ and $Z^* \in L^\infty( \partial \W, \mathcal{H}^{n-1})$ whose pointwise nontangential limit exists $\mathcal{H}^{n-1}\llcorner \partial \W$-a.e.
\end{lemma}

\section{The truncated barycenter and initial simplifications}\label{sec: more preliminaries}
We define the truncated barycenter and prove a key Lipschitz property in Lemma~\ref{lem: lip bary}, then prove three lemmas that reduce the proof of  Theorem~\ref{thm: main} to $f$ that are smooth, nonnegative, and compactly supported.

 \subsection{Truncated barycenters and a key Lipschitz property}\label{sec: truncated barycenter}
\label{ssec: lip bary}
Simple examples, for instance, the sequence of sets $\Omega_j = B_{1-1/j}(0) \cup B_{r_j}(2^j\,e_1)$ with $r_j$  chosen so $|\Omega_j|= \om_n$,  show that even a qualitative form of Theorem~\ref{thm: main} is false if $x_\W$ is chosen to be the barycenter $ \tfrac{1}{|\W|}\int_\W x \,dx$ of $\W$.
This is a familiar challenge in stability problems, and the usual solution is to take the infimum over all $x \in \R^n$ of the given distance $d(\W, B_1(x))$, rather than choosing an explicit ball center. This solution is unsuitable in the present context for two reasons: first, it leads to serious challenges in establishing the necessary regularity for minimizers. 
Second, to prove Theorem~\ref{thm: eigenfunction estimates}, it is essential that the ball center in \eqref{eqn: quant stability main} is independent of the function $f$.\footnote{The first of these two issues already arises in \cite{BDV15}. 
Their solution is to prove a quantitative stability for sets contained in  $B_R$ using the barycenter $\text{bar}(\W)$, then pass $R$ to infinity using a previously-known non-sharp quantitative stability estimate \cite{FMP09} for the Faber-Krahn inequality with respect to the asymmetry. No analogous non-sharp result exists in the present context.}

In view of this, we define a suitable {truncated barycenter} $x_\W$ that agrees with $\text{bar}(\W)$ for sufficiently well-behaved sets but de-emphasizes small pieces of $\W$ at infinity, so that, for instance, $x_{\Omega_j} \to 0$ for the sequence $\W_{j}$ above. In view of the characterization  of the classical barycenter as $ \text{argmin}\{ \int_\Omega |x-y|^2 \,dy : x \in \R^n\}$, we define the {\it truncated barycenter} $x_\W$ of a measurable set $\W$ with $|\W|>0$ by 
\begin{equation}
	\label{eqn: barycenter definition}
x_\W = \text{argmin}\Big\{ \int_\W \TBF(|x-y|)\,dy : x \in\R^n \Big\},
\end{equation}
where  $\TBF: \R_+\to \R_+$ is a fixed $C^2$ function satisfying 
\begin{equation}
    \label{eqn: conditions on q}
\TBF(t) = t^2\text{ for }t\leq 100, 
\qquad \TBF'(t)<250 \text{ for }t\in \R, \quad \text{ and } \quad  0 < \TBF'' (t)<3 \text{ for }t\in \R.
\end{equation}
A simple way to construct such a function is to let $\ddot{q}(\rho)=2$ for $\rho\leq 100$ and $\ddot{q}(\rho) =\tfrac{1}{\rho^2 -100^2+1/2}$ for $\rho \geq 100$ and take $q(t) = \int_0^t\int_0^s \ddot{q}(\rho)\,d\rho\,ds$.  
It is easy to check from \eqref{eqn: conditions on q} that $x\mapsto \int_\W \TBF(|x-y|)\,dy$ is a strictly convex function and thus $x_\W$ is well-defined, and that $x_\W$ agrees with the classical barycenter for those sets $\W$ with diameter at most $100$.
An important feature of the truncated barycenter is the following Lipschitz continuity property:
 \begin{lemma}\label{lem: lip bary}
  There are dimensional constants $\baryEps >0$ and $C>0$ such that the following holds. Let $\W, \W'\subset \R^n$ be bounded open sets with $|\W \Delta B| \leq \baryEps \omega_n$ and $|\W' \Delta B| \leq \baryEps \w_n$.  Then $x_\W, x_{\W'} \in B_3$ and  
  \begin{equation}
      \label{eqn: bary lip}
  |x_\W -x_{\W'}|\leq C |\W\Delta \W'|.
   \end{equation}
  Moreover, for any $f \in L^\infty(\R^n)$  with $\|f\|_{L^\infty(\R^n)} \leq1$,
\begin{equation}\label{eqn: beta lip}
\Big| \int   |u_\W^f - u_{B(x_\W)}^f|^2 -  \int  |u_{\W'}^f-u_{B(x_{\W'})}^f|^2   \Big| \leq C\,\Big( |\W\Delta \W'| +\int |u_{\W}^f - u_{\W'}^f|\Big)\,. 
\end{equation}

\end{lemma}

\begin{proof}
{\it Step 1:}	Let $\baryEps=\baryEps\in(0,1/5)$ be a fixed constant to be specified later in the proof. We first show $x_\W, x_\W' \subset B_3$. Taking the origin as a competitor in the minimization problem \eqref{eqn: barycenter definition} defining $x_\W$ shows 
	\[
	\int_{\W \cap B} \TBF(|x_\W-y|)\,dy \leq \int_{\W \cap B} \TBF(|y|)\,dy +  \int_{\W \setminus B} [ \TBF(|y|) - \TBF(|x_\W-y|) ] \,dy\leq \w_n + 250 \baryEps \omega_n |x_\W|.
	\]
Since $\TBF(t) \geq \TBF(1)+ \TBF'(1)(t-1) = 2t-1$ by convexity, the left-hand side has the lower bound
\[
\int_{\W \cap B} \TBF(|x_\W-y|)\,dy \geq 2|x_\W|(1-\baryEps) \omega_n - \int_{B} 2|y|+1 \, dy \geq 2|x_\W|(1-\baryEps) \omega_n - 3\omega_n\,.
\]
Choosing $\baryEps$ small enough, these two inequalities together show that $|x_\W| \leq 4/(2(1-\baryEps) -250\baryEps) \leq 3.$\\

{\it Step 2: } Next we show \eqref{eqn: bary lip}. The Euler Lagrange equation for $x_\W$ is 
$0 = \int_\W \TBF'(|x_\W-y|) \frac{x_\W - y}{|x_\W-y|} \,dy,$ so splitting the domain of integration into $\W \cap B_{50}$ and its complement yields the following expression for $x_\W$:
\[
x_\W = \frac{1}{|\W \cap B_{50}|}\Big( \int_{\W\cap B_{50}} y\,dy + \int_{\W \setminus B_{50}} \TBF'(|x_\W -y|) \tfrac{ x_\W -y}{|x_\W -y|} \,dy \Big)
\]
So, $
|x_\W -x_{\W'}| \leq I^{(1)} + I^{(2)} + I^{(3)}$
where 
\begin{align*}
	I^{(1)} &:=  \left| \frac{1}{|\W \cap B_{50}|} -  \frac{1}{|\W' \cap B_{50}|}\right| \, 
	 \bigg| \int_{\W\cap B_{50}} y\,dy + \int_{\W \setminus B_{50}} \TBF'(|x_\W -y|) \tfrac{ x_\W -y}{|x_\W -y|} \,dy \bigg|\\
	 I^{(2)} & : = \frac{1}{|\W' \cap B_{50}|}
	 \left|  \int_{\W' \cap B_{50}} \Big( y + \TBF'(|x_\W - y|) \tfrac{x_\W - y}{|x_{\W }-y|} \Big) \,dy -   \int_{\W \cap B_{50}} \Big( y + \TBF'(|x_\W - y|) \tfrac{x_\W - y}{|x_{\W }-y|} \Big) \,dy\right|\\
	 I^{(3)} &: =\frac{1}{|\W' \cap B_{50} |} \left| \int_{\W' \setminus B_{50}} \Big(\TBF'(|x_\W - y|) \tfrac{x_\W - y}{|x_{\W }-y|}  -\TBF'(|x_{\W'} - y|) \tfrac{x_{\W'} - y}{|x_{\W' }-y|} \Big)\,dy  \right|  \,.
\end{align*}
We have 
\begin{align*}
	I^{(1)} 
	 &\leq  \left( 51  \omega_n +250 \baryEps \omega_n\right) \left| \frac{1}{|\W \cap B_{50}|} -  \frac{1}{|\W' \cap B_{50}|}\right| 
	  \leq 200\omega_n^{-1} \big| |\W \cap B_{50}| - |\W' \cap B_{50}|\big| \leq 200\omega_n^{-1}|\W \Delta \W'|.
\end{align*}
Moreover, the modulus of the integrand in $I^{(2)}$ is bounded by $300$ and thus  
\[
I^{(2)}  \leq \frac{300 |\W\Delta \W'| }{|\W' \cap B_{50}|} \leq 600 \omega_n  |\W\Delta \W'|.
\]
 For the final term $I^{(3)}$, direct computation shows that  the function ${g}_y(x):= \TBF'(|x-y|) \frac{x-y}{|x-y|}$ has $|\nabla_x {g}_y(x)|\in C_n$ uniformly bounded for all $x,y$ with $|x-y| \geq 10$. In particular, since $x_\W, x_{\W'}$ lie in $B_3$, by the fundamental theorem of calculus,
\begin{align*}
		I^{(3)} \leq \frac{1}{|\W' \cap B_{50}|} 
		 \int_{\W' \setminus B_{50}} 
		 |{g}_y(x_\W) - {g}_y(x_\W')| \leq C_n \baryEps  |x_\W - x_\W'| \,. 
\end{align*}
Putting these estimates together we have $|x_\W -x_{\W'}| \leq C_n ( |\W\Delta \W'| + \baryEps |x_\W -x_{\W'}|)$. We conclude the proof of \eqref{eqn: bary lip} by choosing $\baryEps$ small enough to absorb the second term into the left-hand side.\\

{\it Step 3:} Now we prove \eqref{eqn: beta lip}. Fix $f \in L^\infty$ with $\|f \|_{L^\infty(\R^n)}\leq 1$. For brevity we write $u_\W, x,$ and $x'$ in place of $u_{\W}^f, x_{\W},$ and $ x_{\W'}$ respectively. 
 Since $u_\W, u_{\W'}, u_{B(x)}$ and $u_{B(x')}$  are each pointwise bounded above by $C_n$, we use the identity $a^2-b^2 =(a+b)(a-b)$ to find
	\begin{equation}\label{eqn: beta first}
    \Big|  \int \left( |u_\W - u_{B(x)}|^2 - |u_{\W'} -u_{B(x')}|^2   \right) \Big|
		 \leq C_n \Big(  \int\left|u_\W - u_{\W'} \right| + \int \left| u_{B(x)} - u_{B(x')}\right| \Big).
	\end{equation}
 By the triangle inequality we have 
	\begin{equation}
		\label{e: step 1}
		\begin{split}
	\int |u_{B(x)} - u_{B(x')}| 
	&= \int_{B(x) \setminus B(x')} 	|u_{B(x)}|  + \int_{B(x') \setminus B(x)} 	|u_{B(x')} | \\
	&+ \int_{B(x)\cap B(x')} |u_{B(x)} - u_{B(x')}|  
	 =: I^{(1)} + I^{(2)} +I^{(3)} 
	 		\end{split}
		\end{equation}
By the maximum principle and the expression \eqref{eqn: tor function on a ball} for the torsion function on a ball,  \begin{equation}\label{eqn: I1 and I2}
I^{(1)} + I^{(2)} \leq C_n |B(x) \Delta B(x')| \leq C_n  |x-x'|.
\end{equation}
Next, $v= u_{B(x)} - u_{B(x')}$ is harmonic on $E:=B(x)\cap B(x')$ so by the maximum principle and \eqref{eqn: tor function on a ball},
\[
\sup_{E} |v|= \sup_{ \partial E}|v| \leq \sup_{ \partial B(x)} |u_{B(x')}| + \sup_{  \partial B(x')} |u_{B(x)}| \leq 2 \sup_{\partial B(x)} w_{B(x')} \leq C_n |x-x'|.
\]
Integrating this bound over $E$ shows $I^{(3)} \leq C_n  |x-x'|.$
Combining this estimate with \eqref{eqn: beta first}, \eqref{e: step 1}, and \eqref{eqn: I1 and I2}, we see that the left-hand side of \eqref{eqn: beta first} is bounded above by $ C( \int |u_\W - u_{\W'}| + |x - x'|).$
Finally,  choosing $\baryEps$  according to Lemma~\ref{lem: lip bary} so that $
|x-x'| \leq C_n |\W \Delta \W'| $ completes the proof.
 \end{proof}

\subsection{Simplifying Assumptions}\label{sec: simplify}
The next three lemmas show that to prove Theorem \ref{thm: main} it will be sufficient to assume that $f\geq 0$ and $f \in C^{\infty}$, and $f$ is compactly supported. The only quantitative assumption on $f$ will be that {$\| f \|_{L^{\infty}(\mathbb{R}^n)}\leq 1$.}

 \begin{lemma} \label{l:fnonneg}
   It is sufficient to prove the estimate \eqref{eqn: quant stability main} of  Theorem \ref{thm: main} in the case when  $f \geq 0$. 
 \end{lemma}

 \begin{proof}
 Choose any $f$ with $\| f\|_{L^{\infty}} \leq 1$ and  write $f = f^+ - f^-$. Note that $\| f_+\|_{L^{\infty}}, \| f_-\|_{L^{\infty}} \leq 1$
   \[
   u_{\W}^f = u_{\W}^{f^+} - u_{\W}^{f^-} \quad \text{ and } \quad 
   u_{B(x_{\W})}^f = u_{B(x_{\W})}^{f^+} - u_{B(x_{\W})}^{f^-}. 
   \]
   So, if the estimate \eqref{eqn: quant stability main}  is true whenever the right hand side is nonnegative with $L^{\infty}$ norm at most $1$, then applying it to $\W$ with $f_-$ and $f_+$,
   \[
   \begin{aligned}
   \int \big|u_{\W}^f- u_{B(x_{\W})}^f \big|^2
   &= \int \big|u_{\W}^{f^+} - u_{B(x_{\W})}^{f^+} -u_{\W}^{f^-} + u_{B(x_{\W})}^{f^-}  \big|^2\\
   &\leq 2\int \big|u_{\W}^{f^+} - u_{B(x_{\W})}^{f^+} \big|^2
   + 2\int \big|u_{\W}^{f^-} - u_{B(x_{\W})}^{f^-} \big|^2\leq 4 C_{n}\, (\lambda_1(\W) - \lambda_1(B)). 
   \end{aligned}
   \]
 \end{proof}

\begin{lemma} \label{l:fsmooth}
 It is sufficient to prove  the estimate \eqref{eqn: quant stability main} of  Theorem \ref{thm: main} in the case when  $f \in C^{\infty}(\R^n)$. 
\end{lemma}

 \begin{proof}
 Choose any $f$ with $\| f\|_{L^{\infty}} \leq 1$.
  Utilizing the standard mollification, if $f_{\epsilon} := f \ast \eta_{\epsilon}$, then 
  $\| f_{\epsilon} \|_{L^{\infty}(\mathbb{R}^n)} \leq \| f \|_{L^{\infty}(\mathbb{R}^n)} \leq 1$. Furthermore, $f_{\epsilon} \to f$ in $L^2(\mathbb{R}^n)$. Now 
  \begin{equation} \label{e:fepsilon}
  \begin{split}
    \int \big|u_{\W}^f- u_{B(x_{\W})}^f \big|^2
    &\leq \int \big|u_{\W}^f- u_{\W}^{f_{\epsilon}} \big|^2 + \int \big|u_{\W}^{f_{\epsilon}}- u_{B(x_{\W})}^{f_{\epsilon}} \big|^2  
    + \int \big|u_{B(x_{\W})}^{f_{\epsilon}}- u_{B(x_{\W})}^f \big|^2 
    \end{split}
  \end{equation}
To estimate the first term on the right-hand side of \eqref{e:fepsilon}, note that
  \[
   \begin{aligned}
       \| \nabla (u_{\W}^f- u_{\W}^{f_{\epsilon}}) \|^2_{L^2(\W)}
       &= \int_{\W} (f-f^{\e})(u_{\W}^f- u_{\W}^{f_{\epsilon}}) \\
       &\leq \| f - f^{\e} \|_{L^2(\W)} \| u_{\W}^f- u_{\W}^{f_{\epsilon}}\|_{L^2(\W)}. 
   \end{aligned}
  \]
On the other hand, from the Faber-Krahn inequality we have 
  \[
  \| u_{\W}^f - u_{\W}^{f_{\e}} \|^2_{L^2(\W)} \leq \frac{1}{\lambda_1(B)} \|\nabla (u_{\W}^f - u_{\W}^{f_{\e}}) \|^2_{L^2(\W)}.
  \]
Putting these together gives 
  \[
  \| u_{\W}^f - u_{\W}^{f_{\e}} \|_{L^2(\W)} \leq \frac{1}{\lambda_1(B)} \| f - f^{\e} \|_{L^2(\W)}. 
  \]
  The same estimate will also apply to the third term in \eqref{e:fepsilon}. 
  So, if the estimate \eqref{eqn: quant stability main} of Theorem~\ref{thm: main} holds for smooth right-hand sides, then applying this estimate to $f^\e$ shows the second term in \eqref{e:fepsilon} is bounded above by $C_{n}(\lambda_1(\W) - \lambda_1(B))$.
  Letting $\e \to 0$ in \eqref{e:fepsilon}, we conclude the result.  
 \end{proof}

\begin{lemma}\label{lem:fcptspt}
    It is sufficient to prove  the estimate \eqref{eqn: quant stability main} of  Theorem \ref{thm: main} in the case when  $f$ has compact support.
\end{lemma}

 \begin{proof}
 Fix a smooth cutoff function $\psi:\R_+\to \R$ that is identically equal to $1$ in $(0,1)$ and vanishes on $(2,\infty)$ with $|\psi'|\leq 5$. 
Let $\W$ be an open bounded set with $|\W| = 1$.
  Choose any $f$ with $\| f\|_{L^{\infty}} \leq 1$.   
  
  Now, choose $R\geq 10$ such that $\W \subset B_{R/2}(0)$. Now, let 
 \[
{g(x)  := \psi\Big(\tfrac{|x|}{R}\Big) \,f(x)}
 \]
 By construction, $g$ is compactly supported and agrees with {$f$} in $B_R(0)$.
 
 Finally, by the choice of $R$ and definition of $x_\W$ (to be precise, by Proposition~\ref{lemma: qual stability} and \eqref{eqn: bary lip}), we see that $x_\W \subset B_{R/2}(0)$ and thus $B(x_{{\W}}) \subset B(0, R)$. 
 In particular, $g=f$ on $B({x_\W})\cup \W$, and thus by linearity of the Poisson equation \eqref{eqn: PDE with f RHS}, 
 \[
 \big\| u_\W^f - u_{B(x_\W)}^f \big\|_{L^2} =  \big\| u_\W^g - u_{B(x_\W)}^g \big\|_{L^2}\,.
 \]
 So, if the estimate \eqref{eqn: quant stability main} of Theorem~\ref{thm: main} holds for compactly supported right-hand sides, then the estimate applied with right-hand side $g$ immediately yields the estimate with right-hand side $f$. 
 \end{proof}

\section{Resolvent Estimates for Nearly Spherical Sets}\label{sec: linear stability}
A set $\Omega\subset \R^n$ is called a {\it nearly spherical set} parametrized by $\phi \in C^{2,\gamma}(\partial B, \mathbb{R})$
 if $|\W|=\om_n$,  $x_\W =0$, and 
\[
\partial \Omega :=  \{ (1+\phi(x))x :  x \in \partial B\}\,
\]
for a function $\phi : \partial B\to \R$ with $\|\phi\|_{C^0}< 1.$ The truncated barycenter $x_\W$ of a nearly spherical set $\W$ agrees with the classical barycenter. We show Theorem~\ref{thm: main} holds if $\W$ is a nearly spherical set, provided $\|\phi\|_{C^{1,\alpha}}$ is small with $1/2< \alpha < 1$.
\begin{theorem}\label{cor: Fuglede}
	Fix {$\alpha \in (1/2,1)$}. There exist $\delta = \delta(n,\alpha)$  and $c=c(n,\alpha)$ such that if $f \in L^\infty(\R^n)$ with $\|f \|_{L^\infty}\leq 1$ and  $\W$ is a nearly spherical set of class $C^{2,\gamma}$ parametrized by $\phi$ with $\| \phi\|_{C^{1,\alpha}} \leq \delta$, then
\[
\tor(\W) - \tor(B) \geq c \int \big| u_\W^f - u_{B({x_\W})}^f\big|^2.  
\]
\end{theorem}  
A main tool in the proof is a spectral gap estimate.   
\begin{lemma} \label{thm: spectral gap}
Fix  $\alpha \in (1/2,1)$. There exists $\delta = \delta(n,\alpha)$ such that if $\W$ is a nearly spherical set of class $C^{2,\gamma}$ parametrized by $\phi$ with {$\| \phi\|_{C^{1,\alpha}} \leq \delta$}, 
\[
\tor(\W) - \tor(B) \geq \frac{1}{32n^2} \| \phi \|_{H^{1/2}(\partial B)}^2. 
\]
\end{lemma}

  {  If we assume that $\| \phi\|_{C^{2,\gamma}} \leq \delta$, then the above result is Theorem 3.3 in \cite{BDV15}. Since we only assume $f \in L^{\infty}$ and not necessarily H\"older continuous, we will need the weaker assumption that $\| \phi\|_{C^{1,\alpha}} \leq \delta$ for some $\alpha>1/2$. The proof with this weaker assumption is obtained using the ideas in the recent paper \cite{p24}, and we provide the details in the appendix. }

Theorem~\ref{cor: Fuglede} follows directly by combining 
Lemma~\ref{thm: spectral gap} and the following lemma.
\begin{lemma}\label{lem: beta on NSS}
{Fix $\alpha \in (1/2,1)$. There exist $\delta = \delta(n,\alpha)$ and $c=c(n,\alpha)$ such that if $f \in L^\infty(\R^n)$ with $\| f\|_{L^\infty}\leq 1$ and $\W$ is a nearly spherical set of class $C^{2,\gamma}$ parametrized by $\phi$ with $\| \phi\|_{C^{1,\alpha}} \leq \delta$,} then
	\[
	\| \phi \|_{H^{1/2}(\partial B)}^2 \geq c \int \big| u_\W^f - u_{B}^f\big|^2 .
	\]
\end{lemma}

Toward proving Lemma~\ref{lem: beta on NSS},  for $\phi: \partial B\to \R$ with {$\| \phi\|_{C^{1,\alpha}} \leq \delta$}, we consider the solution $h$ to 
\begin{equation} \label{e:h}
 \begin{cases}
  &\Delta h=0 \quad \text{ in } B \\
  &h=\phi \quad \text{ on } \partial B. 
 \end{cases}
\end{equation}
Recall that $\| h\|_{H^1(B_1)}=\| \phi\|_{H^{1/2}(\partial B_1)}$. 
From the maximum principle,  $\|h \|_{C^1(\overline{B})} \leq \| \phi \|_{C^1(\partial B)} 
\leq \| \phi\|_{C^{1,\alpha}(\partial B)}\leq \delta$.

We define a diffeomorphism  $\Psi: B\to \Omega$  by  $\Psi(x) = (1 + h(x) ) x$. If $\delta>0$ is chosen to be small, then $D\Psi(x)$ is nowhere vanishing and $\Psi$ is indeed a diffeomorphism from $B$ onto its image. Moreover, $\Psi$ maps $\partial B$ to $\partial \W$ and thus the image of $B$ under $\Psi$ is $\W$. Under these assumptions, we have the following:
\begin{lemma} \label{l:LNbound}
 Let $f \in L^{\infty}(\mathbb{R}^n)$ with $\| f\|_{L^\infty}\leq 1$. Then  
 \[
 \| f- f\circ \Psi \|_{H^{-1}(B)} \leq 2n\| \phi\|_{H^{1/2}(\partial B)}\,.
 \]
\end{lemma}

\begin{proof}
 Let $u \in C_0^1(B_1)$. From Lemma \ref{l:fsmooth} we may assume $f \in C^1(\mathbb{R}^n)$. 
 For $0 \leq t \leq \epsilon$, define
 \[
 \mathcal{F}(t):= \int_{B} f(x+th(x)x) \,u(x) \ dx. 
 \]
 From the mean value theorem, there exists $t_0 \in (0,\epsilon)$ such that 
 \[
 \int_{B}[f(x+h(x)x)-f(x)]\, u(x) \ dx =\mathcal{F}(1)-\mathcal{F}(0)=\mathcal{F}'(t_0)= \int_{B}\sum_{i=1}^n f_{x_i}(x+t_0h(x)x)x_i h(x)u(x). 
 \]
 Since $u h \in C_0^1(B)$, we may use integration by parts to obtain 
 \[
 \int_{B}[f(x+h(x)x)-f(x)]u(x) \ dx =- \int_{B} \sum_{i=1}^n f(x+t_0h(x)x)(u_{x_i}(x)h(x)x_i+ u(x)h_{x_i}(x)x_i + h(x)u(x)). 
 \]
 Using now that $x_i,f \in L^{\infty}$ and H\"older's inequality we have 
 \[
 \left| \int_{B}[f(x+h(x)x)-f(x)]u(x) \ dx \right| \leq 2n \| f \|_{L^{\infty}} \| u \|_{H_0^1(B_1)} \| h\|_{H^1(B_1)} = 2n \| f \|_{L^\infty}  \| u \|_{H_0^1(B_1)}\| \phi\|_{H^{1/2}(\partial B_1)}. 
 \]
\end{proof}
Using Lemma~\ref{l:LNbound}, we can now prove Lemma~\ref{lem: beta on NSS}.
\begin{proof}[Proof of Lemma \ref{lem: beta on NSS}]	
	We define $\hat{u}^f_\W : B\to \R$ by $\hat{u}^f_\W (x) = u^f_\W(\Psi(x))$, and we will prove that 
	\begin{align}
	\label{est 1}	\int |u^f_{\W}(x) - \hat{u}^f_\W (x)|^2 \,dx \leq C \| \phi \|_{H^{1/2}(\partial B)}^2,\\
	\label{est 2}	\int | \hat{u}^f_\W(x) - u^f_B(x)|^2 \, dx \leq C \| \phi \|_{H^{1/2} (\partial B)}^2
	\end{align}
	from which the lemma follows. We note that since $\phi \in C^{1,\alpha}(\partial B)$ and $\| f\|_{L^\infty(\R^n)} \leq 1$, by \cite[Theorem 8.33]{GT}, we have 
    \begin{equation} \label{e:c1est}
       \| u^f_{\W} \|_{C^1(\overline{\W})} , \, \| \hat{u}^f_{\W} \|_{C^1(\overline{B})}
       \leq C(n,\alpha). 
    \end{equation}

	To prove \eqref{est 1}, we observe that for any $x \in B\cap \W$, we have
	\[
	|u^f_\W(x) - \hat{u}^f_\W(x)| =|u^f_\W(x) - u^f_\W(\Psi(x))|\leq \| u^f_\W\|_{C^1(\W)} | x- \Psi(x)| \le  \| u^f_\W\|_{C^1(\W)} |h(x)|.
	\]
	So, 
	\[
	\int_{\W \cap B}|u^f_\W(x) - \hat{u}^f_\W(x)|^2 \,dx \leq C\| u^f_\W\|^2_{C^1(\W)}  \int_{\W \cap B} h^2(x) \leq  C\| u^f_\W\|_{C^1(\W)}  \| \phi \|_{L^2(\partial B)}^2\,.
	\]
	 Next, for $x \in \W\setminus B$ we have $|u^f_\W(x) - \hat{u}^f_\W(x)| = u^f_\W(x) \leq C\| u^f_\W\|_{C^1(\W)} d(x, \partial \W)$, so that 
	\begin{align*}
		\int_{\W \setminus B} |u^f_\W(x) - \hat{u}^f_\W(x)|^2 \,dx  & \leq C\| u^f_\W\|^2_{C^1(\W)}  \int_{\W \setminus B} d(x, \partial \W)^2 \, dx\\
		 &= C\| u^f_\W\|^2_{C^1(\W)}  \int_{S^{n-1} \cap \{ \phi>0\}} \int_1^{1+ \phi} (1+\phi -r)^2 \,r^{n-1} dr \, d\theta \\
		&=C\| u^f_\W\|^2_{C^1(\W)}  \int_{S^{n-1} \cap \{ \phi>0\}} \int_{0}^\phi s^2 \, ds \, d\theta  \leq \e \| \phi\|_{L^2(\partial B)}^2.
	\end{align*}
	An analogous argument shows the same bound for the integral over $B\setminus \W$. Summing up, we arrive at \eqref{est 1}.

	Next we prove \eqref{est 2}. We begin by noting that the pulled back function $\hat{u}^f_\W$ solves the equation
	\[
	\begin{cases}
		-\mathcal{L} \hat{u}^f_\W = \text{det} (d\Psi(x))f(\Psi(x))  & \text{ in } B\\
		\hat{u}^f_\W =0 & \text{ on } \partial B, 
	\end{cases}
	\]
	where $\mathcal{L}u = \text{div}(A \nabla u)$ for $A = m\mathcal{A}$, $\mathcal{A}(x) = d\Phi(\Psi (x) ) (d\Phi (\Psi(x))^*$ and $m = 1/\sqrt{\det \mathcal A}=\det(d \Psi(x))$ and $\Phi = \Psi^{-1}$ with the standard estimates 
	\begin{equation} \label{e:estimates}
	|A -\text{Id}| \leq C(h +|\nabla h|) , \qquad |\det(d\Psi(x))-1| \leq C(h+ |\nabla h|). 
	\end{equation}
	Let $v = u^f_B - \hat{u}^f_\W$; note that $v=0$ on $\partial B$ and that $-\Delta v = (f-mf(\Psi(x))) - (\mathcal{L}-\Delta)\hat{u}^f_\W. $ By the Poincar\'{e} inequality on the ball, it suffices to estimate $\int |\nabla v|^2 \leq C \| \phi \|_{H^{1/2}(\partial B)}^2.$ We have
	\begin{align*}
		 \int_B |\nabla v|^2\,dx& = \int_B \nabla v \cdot ( \nabla u^f_B - A\nabla \hat{u}^f_\W ) + \int_B \nabla v \cdot (A- \text{Id} ) \nabla \hat{u}^f_\W \\ 
   &=\int_B v \left( f(x)-\det(d \Psi(x))f(\Psi(x))\right) + \int_B \nabla v \cdot (A- \text{Id} ) \nabla \hat{u}^f_\W = (I) + (II). 
	\end{align*}
 Now 
 \[
 (I) = \int_{B} v(x) (f(x)-f(\Psi(x))) \ dx + \int_{B}v(x)f(\Psi(x))(1-\text{det} (d \Psi(x)))\ dx= (Ia)+ (Ib). 
 \]
 From Lemma \ref{l:LNbound} we have 
 \[
 |(Ia)|\leq n \| f \|_{L^{\infty}} \| v \|_{H_0^1(B)} \| \phi \|_{H^{1/2}(\partial B)}
 \leq \epsilon \int_{B} |\nabla v|^2 + \frac{C}{\epsilon} \| f \|^2_{L^{\infty}} \| \phi \|^2_{H^{1/2}(\partial B)}. 
 \]
 Using the estimates in \eqref{e:estimates} we also have 
 \[
 |(Ib)| \leq C \int_{B} |v(x)| |f(\Psi(x))| (|h| + |\nabla h|) \ dx 
 \leq \epsilon \int_{B} v^2 + \frac{C}{\epsilon} \| f \|^2_{L^{\infty}} \| \phi \|^2_{H^{1/2}(\partial B)}. 
 \]
 Again using \eqref{e:estimates} as well as estimates on $\hat{u}^f_{\Omega}$ we obtain 
 \[
 |(II)| \leq C \| \hat{u}^f_{\Omega} \|_{C^1(B)} \int_{B} |\nabla v| (|h|+ |\nabla h|) \ dx 
 \leq \epsilon \int_{B} |\nabla v|^2 + \frac{C}{\epsilon}  \| \hat{u}^f_{\Omega} \|^2_{C^1(B)} \| \phi \|^2_{H^{1/2}(\partial B)}. 
 \]
 Using \eqref{e:c1est} this proves \eqref{est 2} and concludes the proof. 
 \end{proof}

\section{Proofs of the Main Theorems}\label{sec: proofs of main theorems}
This section is dedicated to showing how the two main theorems, Theorem~\ref{thm: main} and Theorem~\ref{thm: eigenfunction estimates}, follow from Theorem~\ref{thm: summary for minimizers} and Theorem~\ref{cor: Fuglede}. For the time being, we assume Theorem~\ref{thm: summary for minimizers} holds; this theorem will be established in sections~\ref{sec: a priori and existence} and \ref{sec: FB2}. 

\subsection{Proof of Theorem~\ref{thm: main}}\label{ssec: proof of main}
We begin with Theorem~\ref{thm: main}, which, as explained in the introduction, follows from Theorem~\ref{thm: summary for minimizers} and Theorem~\ref{cor: Fuglede} using a classical selection principle argument.
\begin{proof}[Proof of Theorem~\ref{thm: main}]

 {\it Step 1: Reductions and setup.}
Together Lemmas~\ref{lem: KJ consequence}, \ref{l:fnonneg}, \ref{l:fsmooth}, and \ref{lem:fcptspt} show that to prove Theorem~\ref{thm: main}, it suffices to show that there is a positive constant $c_{n}>0$ depending on $n$ such that for any bounded open set $\W \subset \R^n$ with $|\W| = \omega_n$ and for any nonnegative function $f \in C^\infty_c(\R^n)$ with $\|f\|_{L^{\infty}(\R^n)}\leq 1$.
\begin{align}
\label{eqn: quant stability main 2}
\tor(\W ) - \tor(B)  & \geq c_{n}\int_{\R^n} \big|u_{\W}^f - u_{B(x_\W)}^f\big|^2 \,
\end{align}
Here $x_\W$ is the truncated  barycenter defined in \eqref{eqn: barycenter definition}. Since the integral on the right-hand side is bounded above by $2\om_n\|w_B\|_{L^\infty}^2 \leq \om_n/2n^2 \leq 1,$ the estimate \eqref{eqn: quant stability main 2} holds trivially when $\tor(\W ) - \tor(B)\geq \e$ for any $\e>0$ by choosing the constant $c_{n}$ to be small enough depending on $\e$. So, it suffices to show the theorem when  is  $\tor(\W ) - \tor(B)$ is small.

We use the short-hand notation $\beta_f(\W) = \| u_\W^f - u_{B(x_\W)}^f\|_{L^2(\R^n)}$. 	Suppose by way of contradiction that there are open bounded sets $\{\W_j\}$ in $\R^n$ with $|\W_j| = \omega_n$ and  smooth, nonnegative, compactly supported functions  $\{f_j\}$ with $\| f_j\|_{L^{\infty}}\leq 1$ such that $0< \tor(\W_j)-\tor(B_1)\to 0$ but 
	\begin{equation}\label{e: sp contra}
		\tor(\W_j ) - \tor(B_1) \leq \frac{1}{j}\, \beta_{f_j}^2(\W_j). 
	\end{equation}

 {\it Step 2: Replacement of $\W_j$.}
Fix $\delta = \delta(n, \alpha)$ according to Theorem~\ref{cor: Fuglede} and let $\tau \leq  \bar{\tau}_0(n,\alpha,\delta) = \bar{\tau}_0(n,\alpha)$ be chosen according to Theorem~\ref{thm: summary for minimizers}.  Let $a_j = \beta_{f_j}(\W_j)^2 \in (0, 1]$ and consider the functional
	\[
\mathcal{F}_j(U) =	\tor(U) + \tau \sqrt{a_j^2 + (a_j -\beta^2_{f_j}(U))^2}\,.
	\]
By Theorem~\ref{summary ex and reg theorem}, there is a minimizer  $U_j$ of the problem
    	\begin{equation}
		\inf \{ \mathcal{F}_j(U) \ : \ U \subset \R^n \text{ open bounded set}, \ |U| = \omega_n\},
	\end{equation} 
and $U_j -x_{U_j}$ is a nearly spherical set parametrized by a function $\phi_j \in C^{2,\alpha}( \partial B_1, \R)$ with $\|\phi_j\|_{C^{1,\alpha}}\leq \delta.$ So, by Theorem~\ref{thm: spectral gap} applied with $f = f_j(\cdot -x_{U_j})$, 
	\begin{equation}\label{eqn: sp fuglede}
		\tor(U_j ) - \tor(B_1) \geq c\, \beta_{f_j}^2(U_j). 
	\end{equation}

 {\it Step 3: Reaching a contradiction.}   
Taking $\W_j$ as a competitor for the minimality of $\mathcal{F}_j$ and applying \eqref{e: sp contra}, 
\begin{align*}
	\tor(U_j)  + \tau  \Big( \sqrt{a_j^2 + (a_j -\beta^2_{f_j}(U))^2} -a_j\Big)  \leq \tor(\W_j) \leq  \tor(B_1) + \frac{2}{j} a_j
\end{align*} 
The first inequality implies that $\tor(U_j) \leq \tor(\W_j)$, while the far left- and right-hand sides of the inequality and the Saint-Venant inequality guarantee that 
\[
\sqrt{1 + (1 -\beta^2_{f_j}(U_j)/a_j)^2} -1  \leq \frac{2}{j \,\tau } 
\]  
and thus $\beta_{f_j}(U_j)\geq 2 a_j= 2\beta_{f_j}(\W_j)^2$ for $j$ sufficiently large. Together these two observations and \eqref{e: sp contra} imply
	\begin{equation}
		\tor(U_j ) - \tor(B_1) \leq \frac{4}{j}\, \beta_{f_j}^2(U_j). 
	\end{equation}
	This contradicts \eqref{eqn: sp fuglede} when $j$ is large enough, completing the proof.
\end{proof}

\subsection{Proof of Theorem~\ref{thm: eigenfunction estimates}} Next we prove Theorem~\ref{thm: eigenfunction estimates} as an application of Theorem~\ref{thm: main}. Recall that by Lemma~\ref{lem: KJ consequence}, it suffices to replace the Faber-Krahn deficit $\lambda_1(\W)-\lambda_1(B)$ in Theorem~\ref{thm: eigenfunction estimates} with the Saint-Venant deficit $\tor(\W) -\tor(B)$, and that in section~\ref{ssec: proof of main} we proved that Theorem~\ref{thm: main} holds with $\tor(\W) -\tor(B)$ in place of $\lambda_1(\W)-\lambda_1(B)$.
Throughout this section we assume by translation that without loss of generality the truncated barycenter $x_\W$ of $\W$ is the origin.

Recall from the introduction that we denote $u_{\W,k}$ and $\lambda_{\W,k}$ as the $k$-th eigenfunctions and eigenvalues of $\Omega$, with the eigenfunctions extended globally to be defined on $\R^n$ and  normalized so that 
\begin{equation} \label{e:one}
 \| u_{\W,k} \|_{L^2(\W)} = 1 =  \| u_{B,k} \|_{L^2(\W)}. 
\end{equation}
If $k=1$, then we simply write $u_{\W}$ and $\lambda_{\W}$.

We will utilize the following qualitative control on the eigenvalues whose proof is accomplished with standard geometric measure theory and concentration compactness techniques. 
\begin{proposition} \label{p:standard}
 Let $\W\subset\R^n$ be an open bounded set with $|\W|=\om_n$. For any $k \in \mathbb{N}$ and $\epsilon >0$, there exists $\delta_k >0$ such that if 
 \[
 \tor(\W)-\tor(B) < \delta_k, 
 \]
 then 
 \[
  |\lambda_k(\W) -\lambda_k(B)| \leq \epsilon.  
 \]
\end{proposition}
The recent result of Bucur-Lamboley-Nahon-Prunier described in the introduction is a sharp quantitative version of the qualitative Proposition~\ref{p:standard}.

To illustrate the method, we first prove Theorem~\ref{thm: eigenfunction estimates} for the first eigenfunctions. 
\begin{theorem} \label{t:first}
 Let $\W\subset \mathbb{R}^n$ be an open bounded set with $|\W|=|B|$ and $x_\W = 0$. There is a constant $c>0$ depending on dimension such that 
 \[
 \tor(\W)-\tor(B) \geq c \int_{\mathbb{R}^n} |u_{\W}-u_{B}|^2. 
 \]
\end{theorem}
\begin{proof}
We use the notation $\delta(\W)=\tor(\W)-\tor(B)$. 
As in the proof of Theorem~\ref{thm: main} above, it suffices to consider the case when $\delta(\W)$ is small.
From Theorem~\ref{thm: main} we have 
\[
 \| u_{\W}^f - u_{B}^f \|^2_{L^2(\mathbb{R}^n)} \leq C \delta(\W)
\]
 for any $f \in {L^\infty}(\R^n)$ with $\| f\|_{{L^\infty}(\R^n)} \leq 1.$
The eigenfunction $\lambda_B u_B$ is {bounded} (indeed Lipschitz continuous), and we will consider $f=\lambda_B u_{B}$, so the inequality above becomes 
\begin{equation} \label{e:quant}
 \| u_{\W}^f - u_{B}^f \|^2_{L^2(\mathbb{R}^n)} \leq C \|\lambda_B u_B \|_{L^{\infty}(\mathbb{R}^n)} \delta(\W)\,.
\end{equation}
We now write 
\[
u_{\W}^f = \sum_{k=1}^{\infty} a_k u_{\W,k}
\quad \text{ and } \quad  
u_{B}\big|_{\W} = \sum_{k=1}^{\infty} b_k u_{\W,k}. 
\]
Then from \eqref{e:quant} we have 
\begin{equation} \label{e:separate}
  \|u_B \|^2_{L^2(B \setminus \W)} +  \| u_{\W}^f - u_{B}\big|_{\W} \|^2_{L^2(\mathbb{R}^n)} =   \| u_{\W}^f - u_{B} \|^2_{L^2(\mathbb{R}^n)} \leq C \delta(\W). 
\end{equation}
 In particular, 
\begin{equation} \label{e:epsab}
    \sum_{k=1}^{\infty} (a_k - b_k)^2 \leq C \delta(\W). 
\end{equation}
Using that $-\Delta u_{\W}^f = \lambda_B u_B \chi_{\W}$ we obtain that for each $k \in \mathbb{N}$,
\[
\lambda_B b_k = \lambda_{\W,k} a_k.  
\]
Then from \eqref{e:epsab} we have 
\[
\sum_{k=1}^{\infty} b_k^2 \left(1- \frac{\lambda_B}{\lambda_{\W,k}} \right)^2 \leq C \delta(\W). 
\]
The spectral gap for the eigenvalues on the ball gives $\lambda_{B}/\lambda_{B,2} \leq c_n <1$. Combining with Proposition~\ref{p:standard}, if $\delta(\W)$ is small enough, then 
$\lambda_{B}/\lambda_{\Omega,2} \leq \tilde{c}_n <1$, and so 
\[
  \frac{\lambda_{B}}{\lambda_{\Omega,k}} \leq \tilde{c}_n <1 \quad \text{ for all } k \in \mathbb{N}. 
\]
Then for a new constant $C_1$ we have 
\[
\sum_{k=2}^{\infty} b_k^2 \leq C_1 \delta(\W). 
\]
Recalling that $\|u_B\|_{L^2(\mathbb{R}^n)}=1$, we have from \eqref{e:separate} that
\[
\sum_{k=1}^{\infty} b_k^2 = \| u_B \big|_\W \|^2_{L^2(\W)} \geq 1 - C\delta(\W). 
\]
Then 
\[
b_1^2 \geq 1 - (C+C_1) \delta(\W),
\]
and so 
\[
\begin{aligned}
\| u_{\W} - u_B \|_{L^2(\mathbb{R}^n)}^2 
&=\| u_{\W} - u_B \big|_{\W} \|_{L^2(\W)}^2 + \|u_B \|^2_{L^2(B \setminus \W)} \\
&=(1-b_1)^2 + \sum_{k=2}^{\infty} b_k^2 + \|u_B \|^2_{L^2(B \setminus \W)} \leq C_2 \delta(\W). 
\end{aligned}
\]
\end{proof}

We now show how the above argument works for any simple eigenvalue $\lambda_{B,j}$. 
\begin{theorem} \label{t:second}
 Let $\W\subset \mathbb{R}^n$ be an open bounded set with $|\W|=|B|$ and $x_\W =0$ and assume that $\lambda_{B,j}$ is a simple eigenvalue for the ball. There is a constant $c$ depending on dimension and $j$ such that 
 \[
 \tor(\W)-\tor(B) \geq c \int_{\mathbb{R}^n} |u_{\W, j}-u_{B, j}|^2. 
 \]
\end{theorem}

\begin{proof}
Once again we let $\delta(\W)=\tor(\W)-\tor(B)$ and note that it suffices to prove the theorem when $\delta(\W)$ is small.
As before, we let $f= \lambda_{B,j} u_{B,j} \in L^\infty(\mathbb{R}^n)$, and by applying Theorem~\ref{thm: main}, we obtain 
\begin{equation}
    \label{a}
\|u_{B,j} \|^2_{L^2(B \setminus \W)} +  \| u_{\W}^f - u_{B,j} \big|_{\W}
\|^2_{L^2(\mathbb{R}^n)} =   \| u_{\W}^f - u_{B,j} \|^2_{L^2(\mathbb{R}^n)} \leq C (\tor(\W)-\tor(B)) 
\end{equation}
where the constant $C$ depends on $n$ and $j$ via the $L^\infty$ norm of $u_{B,j}$. We then write
\[
u_{\W}^f = \sum_{k=1}^{\infty} a_k u_{\W,k}
\quad \text{ and } \quad  
u_{B, j} \big|_{\W} = \sum_{k=1}^{\infty} b_k u_{\W,k}. 
\]
So, \eqref{a} implies
\begin{equation} \label{e:jtransfer}
\sum_{k=1}^{\infty} (a_k - b_k)^2 \leq C \delta(\W). 
\end{equation}
Using that $-\Delta u_{\W}^f = \lambda_{B,j} u_{B,j} \big|_{\W}$ we find that $\lambda_{B,j} b_k = \lambda_{\W,k} a_k$ for each $k \in \mathbb{N}$, so \eqref{e:jtransfer} becomes
\begin{equation}\label{b}
\sum_{k=1}^{\infty} b_k^2 \left(1 - \frac{\lambda_{B,j}}{\lambda_{\W,k}} \right)^2 \leq C \delta(\W). 
\end{equation}
Now, \eqref{a} also guarantees that $\sum_{k=1}^{\infty} b_k^2 = \|u_{B,j}\big|_{\W} \|_{L^2}^2\geq 1 - C \delta(\W)$.
Then using the spectral gap on the ball combined with Proposition~\ref{p:standard} we have that
\[
\left(1 - \frac{\lambda_{B,j}}{\lambda_{\W,k}} \right)^2 \geq C(n,j) \quad \text{ whenever } j \neq k,
\]
as long as $\delta(\W)$ is small depending only on $n$ and $k$.
Then from \eqref{b} we have $
b_j^2 \geq 1 - C_2 \delta(\W)$, and so in particular $(1-b_j) \leq C \delta(\W)$.
In conclusion, we obtain 
\[
\begin{aligned}
\| u_{\W,j} - u_{B,j} \|_{L^2(\mathbb{R}^n)}^2 
&=\| u_{\W,j} - u_{B,j} \big|_{\W} \|_{L^2(\W)}^2 + \|u_{B,j} \|^2_{L^2(B \setminus \W)} \\
&=(1-b_j)^2 + \sum_{k\neq j}^{\infty} b_k^2 + \|u_B \|^2_{L^2(B \setminus \W)} \leq C \delta(\W) 
\end{aligned}
\]
for a constant $C>0$ depending on $n$ and $j$.
\end{proof}
We now consider the situation in which $\lambda_j$ is not a simple eigenvalue on the ball. 
\begin{theorem} \label{t:third}
 Let $\W\subset \mathbb{R}^n$ be an open bounded set with $|\W|=|B|$ and assume that $\lambda_{B,j}$ is an  eigenvalue for the ball with multiplicity $N+1$ so that $\lambda_{B,j}=\lambda_{B,j+N}$. 
 Then there exists a constant $c$ depending on dimension and $j$, and coefficients $d_{kl}$ satisfying $\sum_{k=j}^{j+N} d_{kl}^2=1$ such that 
 \[
 \tor(\W)-\tor(B) \geq c \int_{\mathbb{R}^n} \Big(\sum_{l=j}^{N+j}d_{kl} u_{B,l}- u_{\W,l}\Big)^2. 
 \]
\end{theorem}
\begin{proof}
Once again set $\delta(\W)=\tor(\W)-\tor(B)$ and note that it suffices to prove the theorem when $\delta(\W)$ is small.

{\it Step 1.} We first show that an eigenfunction on $B$ is quantitatively close, in terms of $\delta(\W)$, to a linear combination of the eigenfunctions on $\W$. More precisely, we claim that for each for $k=0,\dots, N$, there are coefficients $d_{k0}, \dots, d_{kN}$ with $\sum_{l=0}^N d_{kl}^2 =1$ such that  
\begin{equation}
    \label{e}
\int |\phi_{k} - u_{B,j+k}|^2 \leq C \,\delta(\W ) \qquad \text{ where } \qquad \phi_k = \sum_{l=0}^N d_{kl}\,  u_{\W,l+j}
\end{equation}

By symmetry considerations, it suffices to prove the claim for $k=0.$
To this end, let $f = \lambda_{B,j} u_{B,j}$ 
and write
\[
u_{\W}^f = \sum_{k=1}^{\infty} a_k u_{\W,k}
\quad \text{ and } \quad  
u_{B} \big|_{\W} = \sum_{k=1}^{\infty} b_k u_{\W,k}. 
\]
Then, applying Theorem~\ref{thm: main} with this choice of $f$, we obtain
\begin{equation}
    \label{c}
\big\|u_{B,j} \big\|^2_{L^2(B \setminus \W)} +  \big\| u_{\W}^f - u_{B,j} \big|_{\W}
\big\|^2_{L^2(\mathbb{R}^n)} =  \big\| u_{\W}^f - u_{B,j} \big\|^2_{L^2(\mathbb{R}^n)} \leq C\, \delta(\W), 
\end{equation}
where $C$ depends on $n$ and $\|u_{B.j}\|_{L^\infty(\R^n)}$ and thus on $n$ and $j$. Repeating the argument of the last two proofs, using  \eqref{c} together with the fact that $\lambda_{B,j} b_k = \lambda_{\W,k} a_k$ for each $k \in \mathbb{N}$, we get
\begin{equation}
    \label{d}
\sum_{k=1}^{\infty} b_k^2 \left(1 - \frac{\lambda_{B,j}}{\lambda_{\W,k}} \right)^2 \leq C \delta(\W). 
\end{equation}
Since \eqref{c} implies that $\big\|u_{B,j}\big|_{\W} \big\|^2\geq 1 - C \delta(\W)$ we have 
$
\sum_{k=1}^{\infty} b_k^2 \geq 1 - C \delta(\W). 
$
Then using the spectral gap on the ball, Proposition~\ref{p:standard}, and \eqref{d}, we see that as long as $\delta(\W)$ is small, then 
\[
\sum_{l=j}^{j+N} b_l^2 \geq 1 - C_2 \delta(\W) \qquad \text{ and }\qquad \sum_{k=1}^{j-1} b_k^2 \ \  +\! \! \sum_{k=N+j+1}^{\infty} \! \! \! \! \! b_k^2 \leq C \delta(\W)
\] 
So, if we let $u = \sum_{l=j}^N b_l u_{\W,l}$, then
\[
\begin{aligned}
\big\| u - u_{B,j} \big\|_{L^2(\mathbb{R}^n)}^2 
&=\big\| u - u_{B,j} \big|_{\W} \big\|_{L^2(\W)}^2 + \big\|u_{B,j} \big\|^2_{L^2(B \setminus \W)} \\
&= \sum_{k=1}^{j-1} b_k^2 \ \  +\! \! \sum_{k=N+j+1}^{\infty} \! \! \! \! \! b_k^2 + \big\|u_{B,j} \big\|^2_{L^2(B \setminus \W)} \leq C\, \delta(\W). 
\end{aligned}
\]
This almost proves the claim, except we note that $\| u \|_{L^2} \leq 1$. But since $\| u_{B,j} \|_{L^2}\geq 1 - C \delta(\W)$, if we choose $\eta$ so that $\|(1+\eta)u\|_{L^2}=1$, then $\eta \leq C \delta(\W) $. Therefore $\| (1+\eta)u - u_{B,j}\|_{L^2(\mathbb{R}^n)}^2 \leq C \delta(\W)$. Then by the triangle inequality, we have proven the claim by choosing the coefficients $d_{jl} = d_l=b_l(1+\eta)$.\\

{\it Step 2:}
The $(N+1)\times (N+1)$ matrix $[d_{kl}]$ is almost orthonormal in the sense that the matrix product $[d_{kl}][d_{kl}]^T$ is equal to 1 along the diagonal, and all other entries have absolute value less than $3C\delta(\W)$. Letting $\phi_k$ be as in \eqref{e} for $k=0,\dots, N$,
then multiplying by the transpose matrix we obtain from \eqref{e} that
\[
\left\|\sum_{k=0}^N d_{lk} \phi_k - u_{\W,l+j} \right\|_{L^2(\R^n)}^2 \leq C \delta(\W) 
\]
for each $l=0, \dots, N$. Since  $\| \phi_k - u_{B,k+j}\|_{L^2(\R^n)}^2 \leq C\, \delta(\W)$ by \eqref{e}, and  $
\left|\sum_{k=1}^N d_{lk}^2\right| \geq  1- C\delta(\W),$  the result follows up to slightly rescaling the coefficients 
 $d_{lk}$ as in step 1. 
\end{proof}

\subsection{Sharpness of the Main Theorems: Examples}\label{ssec: examples}
Theorem~\ref{thm: main} is a (consequence of the Kohler Jobin inequality and the) resolvent estimate of the form
\begin{equation}
    \label{eqn: general resolvent est}
    \tor(\W ) - \tor (B) \geq c \left\| (-\Delta)_{\W}^{-1} -(-\Delta)_{B_1(x_\W)}^{-1}\right\|^2_{X\to Y}\
\end{equation}
with the function spaces $X = L^{\infty}(\R^n)$ and $Y = L^2(\R^n)$.
A natural question is whether this estimate is sharp.
Typically, the sharpness question refers to the exponent on the right-hand side of \eqref{eqn: general resolvent est}, i.e.
\begin{align*}
\text{Q1:} & \text{ Can the exponent $2$ in \eqref{eqn: general resolvent est} be replaced with a smaller one?}  \\
\intertext{It is also interesting to ask if \eqref{eqn: general resolvent est} is optimal with respect to the how we quantify the distance of the resolvent operator $(-\Delta_\W)^{-1}$ to  the family of resolvent operators on unit balls. That is,
}
\text{Q2:} & \text{ Can function space $Y = L^2(\R^n)$ in \eqref{eqn: general resolvent est} be replaced with a stronger norm/smaller function space $Y$?}\\
\text{Q3:} & \text{ Can function space $X= L^{\infty}(\R^n)$ in \eqref{eqn: general resolvent est} be replaced with a weaker norm/larger function space $X$?}
\end{align*}
In the three examples that follow, we address these questions of sharpness (and their counterparts for Theorem~\ref{thm: eigenfunction estimates}), starting with Q1:
\begin{example}[The exponent $2$ in \eqref{eqn: general resolvent est} is sharp] \label{Ex: 1}
{\rm Let $n=2$ and consider the family of ellipses $E_\e = \{(x,y) \in \R^2 : (1+\e)^2 x^2 + \frac{y}{(1+\e)^2} < 1 \}$ with $|E_\e | = |B|  =\pi$, which have (truncated) barycenter $x_\W$ at the origin.  Direct computation verifies  that $\tor(E_\e) - \tor(B) = c_1 \e^2$ for a universal constant $c_1$ and 
\[
w_{E_\e}(x,y) = C_\e \Big(1- (1+\e)^2 x^2 - \frac{1}{(1+\e)^2}y^2 \Big)
\]
where  $C_\e = \frac{1}{4} - \frac{1}{2}\e^2 + o(\e^2)$. On $B_{1/2} \subset B_1 \cap E_\e$ (for $\e$ small), a calculation shows
\[
|(w_{E_\e} - w_B)|(x,y) \geq \frac{\e}{2} |x^2 - y^2 | + o(\e),
\]
and thus $\int |(w_{E_\e} - w_B)|^2 \geq c_2\e^2$ for a universal constant $c_2.$ Thus, taking $f=1$ and letting $\e$ tend to $0$ shows that the exponent $2$ is \eqref{eqn: general resolvent est} is optimal. Clearly the analogous example can be used in any dimension.
}
\end{example}
Similarly, we may ask Q1 in the context of Theorem~\ref{thm: eigenfunction estimates}: can the exponent $2$ in \eqref{eqn: simple efn stability} be improved? The answer is again no. 
The next example is essentially similar to Example~\ref{Ex: 1}, though less hands-on since eigenfunctions do not have a clean closed form like the torsion functions of Example~\ref{Ex: 1}.
 
\begin{example}[The exponent $2$ in \eqref{eqn: simple efn stability} is sharp]\label{Ex: 1.5}
{\rm
Let $k=1$ and fix $\phi \in C^{\infty}(\partial B)$ be a nontrivial function with $\int_{\partial B} \phi = \int_{\partial B} \phi (x\cdot e_i)  = 0$ for $i=1,\dots, n$. Let $\{\W_t\}_{t \in (-\e,\e)}\}$ be a one-parameter family of open sets with smooth boundary such that $|\W_t| = \omega_n$ and $x_{\W_t}=0$ whose boundaries have normal velocity at $t=0$ equal to $\phi \nu_B$ on $\partial B$.
Then expansion of $\lambda_1(\W) - \lambda_1(B)$ in \cite[Theorem 1]{d02} togther with the spectral gap of \cite[Theorem 3.7]{AKN1} (see also \cite[Theorem 3.3]{BDV15}) show that $\lambda_1(\W) - \lambda_1(B) \approx t^2 \|\phi\|_{H^{1/2}}^2$. For $t$ small,  $B_{1/2} \Subset \Omega_t \cap B_1$ and, writing $u_\W = u_{\W,1}$,   we have $\sup_{x \in B_{1/2}} |u_{\Omega_t}(x) -u_{B}(x) - t \dot{u}(x)| \leq C t^2$ for a constant $C$ depending on $\phi$, where $\dot{u} \in H^1(B)$ solves 
\[
\begin{cases}
   -\Delta \dot{u} =  \lambda_1(B) \dot{u}& \text{ on } B\\
 \ \ \ \ \ \dot{u} = |\nabla u_B| \phi & \text{ on }\partial B\\
    \int_B u_B \dot{u} = 0.
\end{cases}
\]
see \cite[Section 5.7]{hp18}. Such a solution exists by the Fredholm alternative and it is not identically zero in $B_{1/2}$ by the maximum principle (as $\phi$ is not identically zero)
So, $\liminf_{t\to 0}\frac{1}{t^2} \int |u_{\W_t} - u_B|^2 \geq \int_{B_{1/2}} |\dot{u}|^2>0$, and thus $\int |u_{\W_t} - u_B|^2 \gtrsim \lambda_1(\W_t) - \lambda_1(B)$ for $t$ small.
}
    
\end{example}

Next we address Q2.  The analogue of the following example was originally given in our earlier paper \cite{AKN2} to demonstrate that the $L^2$ norm in \eqref{eqn: simple efn stability} of Theorem~\ref{thm: eigenfunction estimates} cannot be replaced by the $H^1$ norm. 
\begin{example}[$Y=L^2(\R^n)$ cannot be replaced by $Y = H^1(\R^n)$]
   \label{Ex: 2} {\rm
   Suppose by way of contradiction that the resolvent estimate \eqref{eqn: general resolvent est} holds with $H^1(\R^n)$ in place of $L^2(\R^n)$, so that in particular, for the family of ellipses $\{E_\e\}$ in Example~\ref{Ex: 1} above, 
    \[
    \tor(E_\e )- \tor(B) \geq c \int | \nabla w_{E_\e} - \nabla w_{B}|^2.
    \]
    Since $|\nabla w_B| \approx 1$ on $B\setminus E_\e$, this would imply $\tor(E_\e )- \tor(B) \geq c |B\setminus E_\e|$. On the other hand, direct computation shows that $\tor(E_\e )- \tor(B) \approx \e^2$ (as in Example~\ref{Ex: 1} above) and $|B\setminus E_\e| \approx \e$, a contradiction for $\e$ sufficiently small.
    }
\end{example}
\begin{remark}
    {\rm 
    Example~\ref{Ex: 2} gives a partial answer to Q2 by showing that the estimate \eqref{eqn: general resolvent est} {\it fails} with $Y = H^1(\R^n)$. However, the picture is still incomplete, as there are many norms that are stronger than $L^2(\R^n)$ but weaker than $H^1(\R^n)$. A very interesting open problem would be to understand if \eqref{eqn: general resolvent est} holds with $Y = H^s(\R^n)$ for some $s \in (0,1)$. 
    }
\end{remark}

The next example addresses Q3. 

\begin{example}[$X=L^{\infty}(\R^n)$ cannot be replaced by $X= L^p (\R^n)$ for $p<n/2$]\label{Ex: 3}
    {\rm
    Fix $p \in [1,\infty)$, and for $r\in (0,1)$ small, define the set
\[
\Omega:= B_{1-\rho}(0) \cup B_r(2e_1) \cup B_r(-2e_1),
\]
with $\rho$ chosen such that $|\Omega|=|B|$. The (truncated) barycenter of $\Omega$ is the origin and agrees with the barycenter of $B$. Let  
$$f:=M\chi_{B_r(2e_1)\cup B_r(-2e_1)}$$
with $M=2|B_r|^{-1/p}$ chosen so that $\|f \|_{L^p(\mathbb{R}^n)}=1$. Now 
\[
w_{\W}=
\begin{cases}
\ \  \frac{\rho^2-|x|^2}{2n} \quad &\text{ if } x \in B_{1-\rho} \\
\ \  \frac{r^2-|x\mp 2e_1|^2}{2n} \quad &\text{ if } x \in B_{\pm 2e_1} \\
\ \ 0 \quad &\text{ otherwise } 
\end{cases}
\qquad \text{ and } \qquad 
u_{\W}^f=
\begin{cases}
 \ \ 
 0 \quad &\text{ if } x \in B_{1-\rho} \\
 \ \ M\frac{r^2-|x\mp 2e_1|^2}{2n} \quad & \text{ if } x \in B_{\pm 2e_1} \\
  \ \ 0 \quad  &\text{ otherwise }. 
\end{cases}
\]
For small $r$ we have $\rho \approx r^n$, so that 
\[
\tor({\W})-\tor(B_1) \approx r^n. 
\]
Now $u_{B_1(x_{\W})}^f \equiv 0$ and so 
\[
 \|u_{\W}^f - u_{B_1}^f \|_{L^2(\mathbb{R}^n)}^2 \approx r^{n+4-2n/p}. 
\]
If $p< n/2$, then we cannot bound  $\|u_{\W}^f - u_{B_1}^f \|_{L^2(\mathbb{R}^n)}^2$ by a constant multiple of   $\tor(\W)-\tor(B_1)$. 

 }
\end{example}

\begin{remark}
    {\rm Example~\ref{Ex: 3} shows that the estimate \eqref{eqn: general resolvent est} fails with $X= L^p(\R^n)$ (even for the problem restricted to $B_R$) for $p<n/2$, and thus, for instance, rules out $X= L^2(\R^n)$ in dimension $4$ and higher. However,  as with Q2 and Example~\ref{Ex: 2} above, this provides only a partial  answer to Q3, since there are many other possible choices of the function space $Y$. In particular, linear stability arguments of section~\ref{sec: linear stability} go through if $f \in L^p$ for $p >n$.   
    }
\end{remark}

\section{Free Boundary Estimates, Part 1: Existence and basic regularity of minimizers}\label{sec: a priori and existence}

The remainder of the paper is dedicated to proving Theorem~\ref{thm: summary for minimizers}. 
As discussed in section~\ref{ssec: qual stab}, it is more convenient to study a relaxed version of the variational problem \eqref{eqn: main min prob} in which the volume constraint is replaced by a volume penalization term. 

The following  parameters will remain unchanged throughout the remainder of the paper.
\begin{equation}\label{eqn: parameter fix}
\begin{split}
&	\mbox{Let $\baryEps= \baryEps(n) >0$ be fixed according to Lemma~\ref{lem: lip bary}. } \\
&	\mbox{Let $\ccdel= \ccdel(n) \in (0, |\tor(B)|/4]$ be fixed according to Proposition~\ref{lemma: qual stability} applied with $\e = \baryEps/8$. } 
\end{split}
\end{equation}
Next, let $\eta_1(n)$ be the constant from Proposition~\ref{lemma: qual stability} and let
\begin{equation}
    \label{eqn: eta bar fix}
\bar{\eta}= \bar{\eta}(n) = \min\left\{ \eta_1(n),  \frac{1}{13|\tor(B)|}, \frac{ 2\bar{\delta}}{\om_n },  \frac{|\tor(B)|-2\bar{\delta}}{8\om_n}, 1\right\} \in(0,1].
\end{equation}
Let $\fvbar =\mathscr{V}_{\bar{\eta}} $ be the volume penalization term defined in \eqref{eqn: def vol penalization fixed param} with $\eta = \bar{\eta}$

We additionally fix a nonnegative function $f \in L^\infty(\R^n)$ with $\| f \|_{L^\infty(\R^n)}\leq 1$ and number $a\in (0,1]$, on which none of the estimates in this section will depend, and let 
\begin{equation}\label{eqn: SP nl}
\nl(\W) = \sqrt{a^2 +  \Big(a - \int_{\R^n} \big|u_\W^f - u_{B(x_\W)}^f\big|^2\Big)^2 }\,
\end{equation}
where $x_\W$ is the truncated barycenter of $\W$ defined in \eqref{eqn: barycenter definition}. For a parameter $\tau>0$,  define the functional $\Ep$ by
\begin{equation}\label{eqn: main functional}
\Ep(\W)   = \tor(\W) + \fvbar(|\W|)+ \err \nl(\W) \,.
\end{equation}
We consider the unconstrained minimization problem 
\begin{equation}
	\label{eqn: MAIN min prob}
	\inf \{{\Ep}(\W) : \W \subset \R^n \text{ open, bounded, } |\W| \leq 2\omega_n\}
\end{equation}
and its counterpart for sets contained in a large ball $B_R$.
The following statement summarizes the main result of the section.

\begin{theorem}\label{thm: existence 2} \label{prop: basic properties}
 Suppose  $f$ is compactly supported. There is a dimensional constant $\bar{\tau}>0$  such that for any $\tau \in (0,\bar{\tau}],$  a  minimizer $\W$ of the variational problem 
	exists, has $\text{diam}(\W) \leq 2$, and satisfies 
     \begin{equation}\label{eqn: low base energy main 2}
    \tor(\W) + \fvbar(|\W|) \leq \tor(B_1)+\ccdel.
\end{equation}
Moreover, there are positive dimensional constants $r_0$, $c_*$, $C_*$, $c$, and $K$ such that the following holds.
For all $x \in \partial \W$ and $r \in (0,r_0]$, the torsion function $w_\W$ of $\W$ satisfies 
    \begin{equation}\label{eqn: NDandLip}
      c_*r \leq \sup_{B_r(x)} w_\W \leq C_* r  \,.
    \end{equation}
Moreover, $\W$ is a set of finite perimeter in $\R^n$ with $\mathcal{H}^{n-1}(\partial \W \setminus \partial^*\W) =0$, satisfies the density estimates 
\begin{align}\label{eqn: volume density}
    c\, \om_nr^n & \leq |B_r(x) \cap \W| \leq (1-c)\om_n r^n\\
    \label{eqn: perim density}
    c\, r^{n-1} & \leq |B_r(x) \cap \pa \W| \leq \frac{1}{c}r^{n-1} 
\end{align}
for all $x \in \partial \W$ and $r \in (0,r_0],$ and is an NTA domain with NTA constant $K$.

\end{theorem}
 When $f$ is not compactly supported, analogous results hold for \eqref{eqn: MAIN min prob} for sets contained in a large ball $B_R$; see Theorem~\ref{theorem: existence} below.
The constant $ \bar{\tau}$ is defined in \eqref{eqn: tau fix}.

\subsection{Initial estimates for outward and inward minimizers}
Given $R>0$, an open bounded set $\W\subset B_R$ is called an inward (resp., outward) minimizer of $\Ep$ in $B_R$ if $|\W| \leq 2 \omega_n $
and $\Ep(\W) \leq \Ep(\W')$ for all $\W'\subset \W$ (resp., $\W\subset \W'\subset B_R$ with $|\W'|\leq 2 \om_n$).
We begin by proving some fundamental estimates that will be used to prove nondegeneracy and Lipschitz estimates for inward and outward minimizers, respectively. 

Since  $\nl(\cdot )$ is a $1$-Lipschitz function  of $\int_{\R^n} |u_\W^f - u_{B(x_\W)}^f|^2$, together Proposition~\ref{lemma: qual stability} and Lemma~\ref{lem: lip bary} imply there is a dimensional constant $\bar{C}_n>0$
 for any  $\W, \W'\subset\R^n$ open bounded sets with $|\W|\leq 2\omega_n$, 
\begin{equation}\label{eqn: nl Lip}
\begin{cases}
\tor(\W) +\vo(|\W|) \leq \tor(B_1) + \ccdel & \\
 |\W\Delta \W'|\leq \omega_n  \baryEps /{2}  &\end{cases} \quad
\implies \quad \ \ |\nl(\W) - \nl(\W')|  \leq \bar{C}_n\,\Big( |\W\Delta \W'| +\int |u_{\W}^f - u_{\W'}^f|\Big)\,.
	 	\end{equation}
Here the parameters are chosen as in \eqref{eqn: parameter fix}.
The following lemma, whose elementary proof is based on the maximum principle, will be key to controlling the right-hand side of \eqref{eqn: nl Lip}.

\begin{lemma}\label{lemma: key estimate}
	Let $\W, \W' \subset \R^n$ be open bounded sets with $\W' \subset \W$.
    Then 
	\begin{equation}\label{eqn: key estimate}
	\int |u_{\W}^f - u_{\W'}^f| \,dx \leq \tor(\W') - \tor(\W).
		\end{equation}
\end{lemma}
\begin{proof}
	By the maximum principle, $u_{\W}^f(x)\leq w_{\W}(x)$ for all $x \in \R^n$, and in particular for all $x \in \pa \W'$. Since $u_{\W'}^f , w_{\W'}$ vanish on $\R^n \setminus\W'$, this implies $u_{\W}^f - u_{\W'}^f \leq w_{\W} - w_{\W'}$ on $\R^n\setminus \W'$. In particular, this ordering holds on  $\partial \W'$. Since the left- and right-hand sides are harmonic in $\W'$,  by the maximum principle, $u_{\W}^f - u_{\W'}^f \leq w_{\W} - w_{\W'}$ in $\W'$ as well. Thus 
	\[
	0 \leq u_{\W}^f - u_{\W'}^f \leq w_{\W} - w_{\W'} \text{ on } \R^n.
	\] 
	Integrating both sides and recalling \eqref{e: torsion and dir energy} gives us 
	\[
	\int |u_{\W}^f - u_{\W'}^f| \leq \int w_{\W} - w_{\W'} = -\frac{1}{2} \left(\tor(\W) - \tor(\W')\right),
	\]
	completing the proof.
\end{proof}
 Lemma~\ref{lemma: key estimate} and the key Lipschitz estimate \eqref{eqn: nl Lip} together allow us to prove the following estimates comparing the torsional rigidity of  inward and outward minimizers satisfying the energy bound 
\begin{equation}\label{eqn: low base energy main}
    \tor(\W) + \fvbar(|\W|) \leq \tor(B_1)+\ccdel.
\end{equation}
to that of suitable competitors, provided the parameter $\tau$ in \eqref{eqn: main functional} is small enough.  Take $\bar{C}_n$ as in \eqref{eqn: nl Lip} and let 
 \begin{equation}\label{eqn: tau 0}
     \tau_0= \frac{\bar\eta}{2 \bar{C}_n} .
 \end{equation}
%
%
\begin{lemma}\label{lem: basic est outward min} For any $\tau \in (0, \tau_0],$ $R>0$, and open bounded set $\W \subset B_R$ satisfying the energy bound \eqref{eqn: low base energy main}, the following holds.
	\begin{enumerate}
		\item If $\W$ is an inward minimizer of $\Ep$ in $B_R$ then for any  competitor $\W'\subset \W$ with $|\W'\setminus\W| \leq \baryEps\om_n/2$,
		\begin{equation}
	\label{eqn: nondegen est 22}
 \, |\W \setminus \W'| \leq \frac{4}{\bar\eta} \left( \tor(\W') - \tor(\W)\right).
\end{equation}
		\item If $\W$ is an outward minimizer of $\Ep$ in $B_R$, then for any  competitor $\W'$ with $\W \subset \W' \subset B_R$ and   $|\W'\setminus\W| \leq \baryEps \om_n/2$,
	\begin{equation}
		\label{eqn: sym diff bounds tor}
\tor(\W) - \tor(\W') \leq \frac{4}{\bar\eta} \, |\W' \setminus \W|.
	\end{equation}
	\end{enumerate}
\end{lemma} 

\begin{proof}
We begin with the first estimate. By inward minimality,
	 \begin{equation}
	 	\label{eqn: nondegen est 1 v2}
	 	\begin{split}
	 			\tor(\W) - \tor(\W') &   \leq   \vo(|\W'|)-\vo(|\W|) + \tau (\nl(\W') - \nl(\W)) \\
	 	& \overset{\eqref{eqn: nl Lip}}{\leq} \vo(|\W'|)-\vo(|\W|) +  \bar{C}_n\tau \big( \,|\W \setminus \W'| + \int |u_\W^f - u_{\W'}^f|\big) \\
	 	& \overset{\eqref{eqn: key estimate}}{\leq}   \vo(|\W'|)-\vo(|\W|) +  \bar{C}_n\tau  |\W\setminus \W'| +  \bar{C}_n\tau(\tor(\W) - \tor(\W')).
	 	\end{split}
	 \end{equation}
	 Since $|\W| \leq  |\W'|$, we have 
	 \[
	 \vo(|\W'|)-\vo(|\W|) \leq - \bar\eta  \,(|\W|-|\W'|)\,,
	 \]
     so that \eqref{eqn: nondegen est 1 v2} becomes
     \[
     (1-\bar{C}_n\tau)(\tor(\W) -\tor(\W')) \leq (-\eta + \bar{C}_n\tau)|\W \setminus \W'|.
     \]
Since we have chosen $2\bar{C}_n \tau\leq \bar\eta \leq 1$, the coefficient on the left-hand side is at least $1/2$, while the coefficient on the right-hand side is at most $-\bar\eta/2$.
This yields $\frac{1}{2}(\tor(\W) - \tor(\W') )\leq -\frac{\bar\eta}{2} |\W \setminus \W' |,$ or equivalently, \eqref{eqn: nondegen est 22}.

Now to prove the second estimate, since $|\W'| \geq |\W|$, the outward minimality of $\W$ gives
\begin{equation}\label{eqn: basic outward competitor}	
\begin{split}
0 \leq 	\tor(\W) - \tor(\W')& \leq \vo(|\W'|) - \vo(|\W|) + \tau(\nl(\W') - \nl(\W))\leq \frac{1}{\bar\eta}|\W'\setminus \W| + \tau(\nl(\W') - \nl(\W)).
\end{split}
\end{equation}
Applying \eqref{eqn: nl Lip} followed by \eqref{eqn: key estimate}  to the right-hand side of \eqref{eqn: basic outward competitor} yields
\begin{align*}
	\tor(\W) - \tor(\W') \leq \Big( \frac{1}{{\bar\vpar}} + \bar{C}_n \tau\Big) |\W'\setminus \W| +  \bar{C}_n \tau \big( \tor(\W) - \tor(\W')\big). 
\end{align*}
Since we have taken $2\bar{C}_n \tau \leq \bar\eta\leq  1$, we can absorb the last term on the right to arrive at \eqref{eqn: sym diff bounds tor}.
\end{proof}

\subsection{Nondegeneracy for inward minimizers}
The next theorem shows that torsion function of an inward minimizer satisfying \eqref{eqn: low base energy main} grows at least linearly from the boundary. In this statement, $\tau_0>0$ is as in  \eqref{eqn: tau 0}.

\begin{theorem} \label{t:lb}  
There is a dimensional constant $c_* >0$ such that for any 
$\err\in ( 0,\err_0]$ and $R>0$, an inward minimizer $\W$ of $\Ep$ in $B_R$ satisfying the energy bound \eqref{eqn: low base energy main} has
\begin{equation}\label{def: DO}
	  \sup_{x \in B_r(y)}  w_\W(x)  \geq c_* r
\end{equation}
for all $r \in(0,1)$ and $y \in \overline{\Omega}$.
\end{theorem}
\begin{proof}
Let $\W$ be a inward minimizer $\W$ of $\Ep$ in $B_R$ satisfying the energy bound \eqref{eqn: low base energy main}. 
Following David and Toro \cite{DT15}, we prove the following main decay estimate: there is a constant  $c_*=c_*(n)>0$ such that for any $y \in \R^n$ and  $r \leq \min\{c_*^2, (\baryEps /2)^{1/n}\}$,
\begin{equation}\label{eqn: decay estimate}
\sup_{B_r(y)} w_\W \leq c_* r\qquad \implies \qquad \sup_{B_{r/2}(y)} w_\W \leq \frac{c_* r}{4}\,.
\end{equation}

Before proving \eqref{eqn: decay estimate}, let us see how iterating it will lead to the desired conclusion. 
Up to replacing $c_*$ with a smaller constant $c_*'$ in the final estimate \eqref{def: DO}, it suffices to prove the lower bound for $r \in (0, r_0)$ where $r_0 = \min\{c_*^2, (\baryEps /2)^{1/n}\}.$
Suppose the conclusion fails, so for some $y\in \overline{\Omega}$ and $r \in(0,r_0)$ we have $\sup_{B_{r}(y)} w_\W  \leq c_* r.$ Thus, by \eqref{eqn: decay estimate}, $\sup_{B_{r/2}(y)} w_\W< c_*r/4$, and in particular, 
\[
\sup_{B_{r/4}(x)} w_\W \leq c_*\left(\frac{r}{4} \right)\qquad \text{ for all } x \in B_{r/4}(y).
\]
Iteratively applying \eqref{eqn: decay estimate} at any such $x$ shows that 
\[
\sup_{{B_{r/4\cdot 2^k}(x)}} w_\W \leq\frac{c_* r}{4^{k+1}} \qquad \text{ for all } x \in B_{r/4}(y).
\]
So, each point in $B_{r/4}(y)$ is a Lebesgue point for $w_\W$ and $w_\W(x)= 0$ for every $x\in B_{r/4}(y)$. On the other hand $B_{r/4}(y)\cap \W$ is nonempty since $y \in \overline{\W}$. Since $w_\W >0$ in $\Omega$, we reach a contradiction and prove the theorem.

\smallskip

We now show \eqref{eqn: decay estimate}. In what follows,  $C_n>0$ denotes a dimensional constant whose value can change from line to line. 
Let $c_* \in (0,1]$ be a fixed number to be specified later.
	Fix $y \in \R^n$ and $r \in (0,\min\{c_*^2, (\baryEps /2)^{1/n}\})$ and assume $\sup_{{B_r(y)}}w_\W \leq c_* r$. Towards showing \eqref{eqn: decay estimate}, we first  claim that 
	\begin{equation}
\label{eqn: nondegen density estimate} 
\frac{| \W \cap B_{3r/4}(y)|}{r^n} \leq \frac{C_n}{\tau}\, c_*^2		
	\end{equation}
	Indeed, take $\W' = \W\setminus B_{3r/4}(y)$ as a competitor for the inward minimality of $\W$. 
Since $|\W \setminus \W| =|\W \cap B_{3r/4}(y)| \leq r^n \om_n  \leq \baryEps \om_n/2$, Lemma~\ref{lem: basic est outward min}(1) implies that 
\begin{equation}
	\label{eqn: nondegen est 2}
 |\W \cap B_{3r/4}(y)| \leq \frac{4}{\bar\eta}\big(\tor(\W') - \tor(\W)\big).
\end{equation}
We estimate the right-hand side of \eqref{eqn: nondegen est 2} as follows. Take a smooth cutoff function $\phi$ with $0 \leq \phi \leq 1$, $\phi =0$ in $B_{3r/4}(y)$, and $\phi =1$ in $\mathbb{R}^n \setminus B_{7r/8}(y)$ and $|\nabla \phi | \leq C_n/r$. Then we have $w_\W \phi \in H^1_0(\W')$ and 
\begin{equation}\label{eqn: 1}
\begin{split}
\int_{\W'} |\na (w_\W \phi) |^2 &= \int_{\W'} |\na w_\W|^2\phi^2 + \int_{B_{7r/8}(y)}  |\nabla\phi |^2 w_\W^2 + 2 \int_{B_{7r/8}(y)} (\nabla\phi \cdot \nabla w_\W )w_\W \phi \\
& \leq \int_{\W}  |\na w_\W|^2 + \frac{C_n}{r^2} \int_{B_{7r/8}(y)} w_\W^2  +  \int_{B_{7r/8}(y)} | \nabla w_\W|^2.
\end{split}
\end{equation}
In the inequality, we used  Young's inequality and the properties of $\phi$ and its gradient. Next, the Caccippoli inequality and the assumed $L^\infty(B_r(y))$ bound applied to the final term on the right-hand side imply that 
\begin{equation}
    \label{eqn: 2}
\int_{B_{7r/8}(y)} |\na w_{\W}|^2 \leq \int_{B_r(y)}  w_{\W} +  \frac{C_n}{r^2} \int_{B_r(y)} w_\W^2 \leq  c_* r^{n+1}  + C_n c_*^2  r^n \leq C_n c_*^2 r^n.
\end{equation}
In the final inequality we use $r \leq c_*^2 \leq c_*$. Combining  \eqref{eqn: 1} and \eqref{eqn: 2} and once more using that  $\sup_{{B_r(y)}}w_\W \leq c_* r$, we have
$\int_{\W'} |\na (w_\W \phi) |^2 \leq \int_{\W}  |\na w_\W|^2 + C_n  c_*^2 r^{n}.
$
A similar argument shows $\int_{\W'} w_\W \phi \geq \int_{\W} w_\W - C_n c_*^2 r^{n}$. Therefore, choosing $w_\W \phi$ as a test function for the torsional rigidity of $\W'$ shows that 
$\tor(\W') -\tor(\W)\leq C_n c_*^2 r^{n}.$
Combined with \eqref{eqn: nondegen est 2}, this shows \eqref{eqn: nondegen density estimate}.

Now, applying the local maximum principle for subsolutions, followed by the assumed $L^\infty(B_r(y))$ bound in \eqref{eqn: decay estimate} and then \eqref{eqn: nondegen density estimate}  implies
	\begin{align*}
		\label{eqn: loc max p}
	\sup_{B_{r/2}(y)} w_\W  & \leq C_n \big(r^{-n/2} \| w_\W\|_{L^2(B_{3r/4}(y))} +r^2\big)
 \leq C_n \Big( c_* r \left( \frac{|\W\cap B_{3r/4}(y)|}{r^n}\right)^{1/2}   + r^2\Big) \leq \frac{C_n}{\bar\eta} c_*^2 r,
	\end{align*}
and choosing $c_*$ small enough depending on $n$ and $\bar\eta$, thus only on $n$ so that $C_n c_*/\bar\eta \leq 1/4$ yields \eqref{eqn: decay estimate}. 
\end{proof}
Theorem~\ref{t:lb} implies that, up to decreasing the error $\bar{\delta}>0$ in the energy bound \eqref{eqn: low base energy main}, inward minimizers satisfying the assumptions of Theorem~\ref{t:lb} have diameter bounded independent of $R$. In the following statement, $\tau_0>0$ is defined in \eqref{eqn: tau 0}.
\begin{corollary}[Diameter bound for inward minimizers]\label{cor: diameter bound}
There  is a  dimensional constant  $\delta_{D} \in (0, \ccdel]$
 such that for any $\tau \in (0,\tau_0]$ and 
$R>0$, an inward minimizer $\W$  of $\Ep$ in $B_R$ satisfying the energy bound
\begin{equation}\label{eqn: low energy for diam bound}
    \tor(\W) + \vo(|\W|) \leq \tor(B) + \delta_D\,
 \end{equation}   
has  $\text{diam}(\W) \leq 2.$
\end{corollary}
\begin{proof} Let $\delta_D \le \ccdel$ be a fixed constant to be specified later in the proof, fix $\tau \in(0,\tau_0]$ $R>0$, and let $\W$ be an inward minimizer in $B_R$ satisfying \eqref{eqn: low energy for diam bound}.  We claim there are dimensional constants $r_0\in (0,1/2]$ and $c_0>0$ such that, for any $x \in \overline{\W}$, and $0< r \leq r_0$,
\begin{equation}
	\label{eqn: lower density}
c_0 r^{n+2} \leq  |B_{r}(x)\cap \W|\,.
\end{equation}
Indeed, for any $r \in (0,1)$ and $x \in \overline{\W}$, apply Theorem~\ref{t:lb} followed by the local maximum principle for subsolutions to find
	\begin{align*}
		\frac{r}{2}\, c_* \leq \sup_{B_{r/2}(x)} w_\W \leq C_n r^{-n/2} \| w_\W\|_{L^2(B_{3/4}(x))} +C_nr^2 \,.
	\end{align*}
    Since $ \sup_{\R^n} w_\W\leq C_n$, 
 \eqref{eqn: lower density} follows by choosing $r_0$ small enough depending on $n$ and $c_*$, thus only on $n$,  to absorb the final term into the left-hand side.

Now, let $\e = c_0 r_0^{n+2}/\om_n$. By  Proposition~\ref{lemma: qual stability}, there exists $\delta_D$ sufficiently small depending on $\e$ and thus on $n$ such that \eqref{eqn: low energy for diam bound} guarantees that $|\W \Delta B_1(x_1)| \leq \frac{1}{2}c_0 r_0^{2+n}$ for some $x_1 \in \R^n$. In particular, for any $x \in \R^n$ with  $|x-x_1| \geq 2,$ since  $B_{r_0}(x)$  and $B_1(x_1)$ are disjoint, we deduce from \eqref{eqn: lower density} that $x \not\in \overline{\W}$. That is, $\W \subset B_2(x_1)$. 
\end{proof}

\medskip 

\subsection{Lipschitz estimate for outward minimizers} The next theorem shows that, for an outward minimizer of $\Ep$ in $B_R$ satisfying \eqref{eqn: low base energy main}, the torsion function is Lipschitz continuous, with uniform bounds on the Lipschitz constant. Again, $\tau_0>0$ is defined in \eqref{eqn: tau 0}.
\begin{theorem}\label{thm: Lip} Fix  $R>0$.
	There is a constant $C_*=C_*(n,  R)$ such that for any $\tau \in (0,\tau_0]$ and any outward minimizer $\W$ of $\Ep$ in $B_R$ satisfying the energy bound \eqref{eqn: low base energy main}, we have 
	\begin{equation}\label{def: UP}
	 { w_\W(x)} \leq C_*d(x, \partial \W).
\end{equation}
for all $x \in \W$ with $d(x, \partial \W) \in (0, 1)$.
In particular, $w_\W$ is globally Lipschitz with $\sup \frac{|w_\W(x)- w_\W(y)|}{|x-y|} \leq \tilde{C}^*$ for a constant $\tilde{C}_*=\tilde{C}_*(R, n)$.

If $\text{dist}(\W , \partial B_R) \geq 2,$ the constants $C_*$ and $\tilde{C}_*$ depend only on $n$.
	\end{theorem}
\begin{remark}\label{remark: Lip poisson equation solution}
    {\rm
    It follows from Theorem~\ref{thm: Lip} and the maximum principle that $u_\W^f$ is also Lipschitz with the same bound on the Lipschitz constant.
    }
\end{remark}
Before proving Theorem~\ref{thm: Lip}, it is useful to first show that the torsion function $w_\W$ of an outward minimizer as in Theorem~\ref{thm: Lip} is  (H\"{o}lder) continuous on $\R^n$.
Let $\bar{\e}$ be as in \eqref{eqn: parameter fix}. Choose $r_0$ small enough so that $|\W| + \omega_n r_0^n \leq \om_n \bar{\e} /2$. 
Let $\W' = \W \cup (B_r(x) \cap B_R)$, which satisfies $|\W' \setminus \W| \leq \omega_n \e/2$, and let $u$ be the harmonic replacement of $w_\W$ in $B_r(x) \cap B_R,$ that is, let  $u$ be the unique harmonic function in $B_r(x) \cap B_R$ with $u -w_\W \in H^1_0(B_r(x) \cap B_R)$, and extend $u$ so  $u=w_\W$ on $\R^n \setminus (B_r(x) \cap B_R)$. We claim that 
\begin{equation}
    \label{eqn: main claim}
\int_{B_R\cap B_r(x)} |\nabla u - \nabla w_\W|^2 \leq C |B_r(x) \setminus \W| \leq  C r^n
\end{equation}
for a constant $C$ depending on $n$.
From \eqref{eqn: main claim}, a classical argument using Campanato's criterion shows that $w_\W \in C^{0,\alpha}(B_R)$ for each $\alpha<1$.
To show \eqref{eqn: main claim}, choose $u$ as a competitor for the torsional rigidity of $\W$. Applying Lemma~\ref{lem: basic est outward min}(2), this implies 
\begin{equation}\label{eqn: 3}
\frac{1}{2}\int_{B_r(x) \cap B_R} (|\nabla w_\W|^2 - |\nabla u|^2) - \int_{B_r(x) \cap B_R}  (w_\W - u) \leq \tor(\W) - \tor (\W') \leq \frac{4}{\bar\eta} \, |\W'\setminus\W | \leq \frac{4\om_n}{\bar\eta} r^n.
\end{equation}
Since $u$ is harmonic, and integration by parts shows that the first term on the left-hand side is equal to $\frac{1}{2} \int_{B_r(x) \cap B_R}  |\nabla w_\W - \na u|^2$. Applying the maximum principle to $u$ in $B_r(x) \cap B_R$, shows that the second term on the left-hand side of \eqref{eqn: 3} is bounded by $C_n r^n$. These two observations applied to 
\eqref{eqn: 3} establish \eqref{eqn: main claim}.

Next, following \cite{AC81},  we first prove the following lemma as the main step toward proving Theorem~\ref{thm: Lip}. 
\begin{lemma}\label{lem: integral est for lip} 
 There are  constants $M_0= M_0(n)>0$ and $r_0=r_0(n)>0$  such that the following holds for any  $\tau\in (0, \tau_0]$, $R>0$,  and outward minimizer  $\W$ of $\Ep$ in $B_R$ satisfying \eqref{eqn: low base energy main}. For   $x_0 \in \R^n$ and $r \in (0,r_0]$ with $\text{dist}(x, \partial B_R) \geq  2r,$ if 
    \begin{equation}
      \label{eqn: hp}
    \fint_{\partial B_r(x_0)} w_\W \,d \mathcal{H}^{n-1} \geq M_0 r,
    \end{equation}
    then $w_\W >0 $ in $B_r(x_0)$.
\end{lemma}
\begin{proof} 
Let $\bar{\e}$ be as in \eqref{eqn: parameter fix}. Choose $r_0$ small enough so that $|\W| + \omega_n r_0^n \leq \om_n \bar{\e} /2$, so that $\W' := \W \cup B_r(x_0)$ satisfies $|\W' \setminus \W| \leq \omega_n \e/2$. Note also that $\W'\subset B_R.$ Let $M_0$ be a fixed dimensional constant to be specified later in the proof and assume that $x_0 \in \R^n$ and $0<r<r_0$ are such that \eqref{eqn: hp} holds.
Let $v$ be the ``torsion replacement''  of $w_\W$ defined by
	\begin{equation}
	    \label{eqn: lip lem replacement}
	\begin{cases}
		-\Delta v =1 & \text{ in }B_r(x_0)\\
		v=w_\W & \text{ in } \R^n \setminus B_r(x_0), 
	\end{cases}
    \end{equation}
Taking $v$ as a competitor for the torsional rigidity of $\W'$ and applying Lemma~\ref{lem: basic est outward min}(2), we have 
\begin{align}\label{eqn: lip lem a}
\int	\frac{1}{2}( |\na w_\W|^2 - | \na v|^2 ) - \int (w_\W - v) &
\leq \tor(\W) - \tor(\W')\leq  \frac{4}{\bar\eta} |\W' \setminus \W| =\frac{4}{\bar\eta} |B_r(x_0) \setminus\W|.
\end{align}
On the other hand, since  $\int \na (w_\W - v) \cdot \na v = \int (w_\W -v)$  by \eqref{eqn: lip lem replacement}, the left-hand side of \eqref{eqn: lip lem a} is equl to
\begin{align*}
	\int	\frac{1}{2}( |\na w_\W|^2- | \na v|^2 ) - \int (w_\W - v)  & =\int  \frac{1}{2} \nabla(w_\W - v) \cdot \nabla(w_\W + v) - \int(w_\W - v) \\
	& = \int \frac{1}{2} |\na (w_\W - v)|^2 + \int \na (w_\W- v) \cdot \na v - \int (w_\W -v) =  \int \frac{1}{2} |\na (w_\W - v)|^2 \,.
\end{align*}
Combining this with \eqref{eqn: lip lem a} gives
\begin{align}\label{eqn: lip lem b}
	\int  |\nabla( w_\W - v)|^2 \leq \frac{4}{\bar\eta}|B_r(x_0) \setminus\W|.
\end{align}
Now, as in \cite[Lemma 5.10]{DPMMCapacity}, apply the Hardy-type inequality\footnote{ This inequality can be proven using integration by parts and H\"{o}lder's inequality.} $\int_{B_r(x_0)}\frac{g^2}{(r- |x-x_0|)^2 } \leq 4 \int_{B_r(x_0)} |\nabla g|^2$ for $g \in H^1_0(B_r(x_0))$ to the function $g = w_\W -v$. This gives 
\begin{align}\label{eqn: hardy app}
	C_n \int_{B_r(x_0)} |\nabla (w_\W - v)|^2 & \geq 	\int_{B_r(x_0)}\frac{(w_\W(x) - v(x)) ^2}{(r- |x-x_0|^2) } \geq \int_{B_r(x_0)\setminus \W}\frac{v(x)^2}{(r- |x-x_0|)^2}\,.
\end{align}
Since $v$ is superharmonic on $B_r(x_0)$, it lies above the harmonic function $\int_{\pa B_r(x_0)}K(x,y) w_\W(y) \, d\mathcal{H}^{n-1}$ on $B_r(x_0)$ with the same boundary data. Here $K(x,y) = \frac{r^2- |x-x_0|^2}{r |x-y|^n}$ is the Poisson kernel on $B_r(x_0)$. Combining this observation with the assumption \eqref{eqn: hp} and noting that $\inf_{y \in \pa B_r(x_0)}K(x,y)\geq \frac{c_n (r-|x- x_0|)}{r^n}$, we have
\begin{align}\label{eqn: poisson formula app}
{v(x)} &  \geq \int_{\pa B_r(x_0)}K(x,y) w_\W(y) \, d\mathcal{H}^{n-1} \geq
	 c_n M_0 (r-|x-x_0|)\,
\end{align}
for any $x \in 
 B_r(x_0).$
Applying \eqref{eqn: poisson formula app} to the integrand on the right-hand side of \eqref{eqn: hardy app} we get 
\begin{equation}\label{eqn: upper denisity in terms of harm replacement}
C_n \int_{B_r(x_0)} |\nabla (w_\W - v)(x)|^2 
\geq   M_0 |B_r(x_0) \setminus \Omega|.
\end{equation}
Putting this together with \eqref{eqn: lip lem b}, we see that $|B_r(x_0)\setminus \W | \leq \frac{C_n}{M_0\bar\eta}|B_r(x_0)\setminus \W |$. Taking $M_0>0$ large enough depending on $n$ and $\bar\eta$, and thus only on $n$, so that $\frac{C_n}{M_0\bar\eta} <1$, we deduce that $|B_r(x_0)\setminus \W |=0$, i.e. $B_r(x_0)\subset \W$. By the strong maximum principle, the lemma holds.
\end{proof}

We now prove Theorem~\ref{thm: Lip}. 
\begin{proof}[Proof of Theorem~\ref{thm: Lip}]
Fix $x \in \partial \W$ and $r\in(0,1)$ and let 
$$S(x,r) = \sup\{ w_\W(y)\,  : \, y \in \partial B_r(x),  \, d(y,\partial \W) = r\}.$$
We claim that $S(x,r) \leq C_*r$, which directly implies \eqref{def: UP}. 
Let $r_0>0$ is the dimensional constant from Lemma~\ref{lem: integral est for lip}. Since $\|w_\W\|_{L^\infty(\R^n)}$ is bounded by a dimensional constant, we trivially have $S(x,r) \leq C_0 r$ for $r\geq \min\{1/8, r_0\}$ and for a dimensional constant $C_0$. We assume henceforth that $r \leq \min\{1/8, r_0\}$.

{\it Case 1:} First suppose $\text{dist}(x, \partial B_R) > 2r.$ 
In this case, we show that $S(x,r) \leq C_1 r$ where $C_1$ depends only on $n$. 
Choose $y \in \partial B_r(x)$ achieving the supremum in the definition of $S(x,r)$. 
By Harnack inequality applied to $w_\W$ on $B_r(y) \subset \W$, 
\begin{equation}\label{eqn: lip intermed}
S(x,r) \leq C_2 \inf_{B_{r/2}(y)} w_\W + C_2r^2
\end{equation}
for a dimensional constant $C_2$.
If $S(x,r) \leq \tfrac{C_2r}{2}$,  the claim holds (with $C_1 =C_2/2$. If $S(x,r) \geq \tfrac{C_2r}{2}$,  we can absorb the final term in \eqref{eqn: lip intermed} into the left-hand side to find that $S(x,r) \leq 2C_2 \inf_{B_{r/2}(y)} w_\W$. 
This estimate, together with the easily verified geometric fact $\mathcal{H}^{n-1}(\partial B_r(x) \cap B_{r/2}(y) ) \geq c_n r^{n-1}$, implies
\[
S(x,r) \leq 2C_2 \fint_{\partial B_r(x)\cap B_{r/2}(y) } w_\W \, d\mathcal{H}^{n-1} \leq C_3\fint_{\partial B_r(x)} w_\W \, d\mathcal{H}^{n-1}
\]
where $C_3$ is another dimensional constant.
Since $x \in \partial \W$, we have $w_\W (x)=0$ and so by (the contrapositive of) Lemma~\ref{lem: integral est for lip}, the left-hand side is bounded above by $C_3 M_0 r$.  This shows the claim in this case. Thus, $S(x,r) \leq C_1r$ where $C_1 = \max\{ \tfrac{C_2}{2}, C_3 \, M_0\}$ is a constant depending on $n$.
\\

{\it Case 2: } Next, suppose  $\text{dist}(x, \partial B_R) \leq 2r.$ By the maximum principle, $w_\W(y) \leq w_{B_R}(y) \leq C(n,R) r$ for any $y \in B_r(x)$, so $S(x,r) \leq C(n,R)r$ in this case.

Thus \eqref{def: UP} holds with $C_* = \max \{ C_0, C_1, C(n,R)\}$. In the case when $\text{dist}(\W , \partial B_R) \geq 2$, case 2 above never occurs do \eqref{def: UP} holds for the dimensional constant  $C_* = \max \{ C_0, C_1\}$.
\end{proof}

\medskip

\subsection{The lower semicontinuity and replacement lemmas}
As discussed in the introduction, direct approaches do not suffice to  establish the existence of minimizers of our main variational problem \eqref{eqn: MAIN min prob}. 
The main challenge is that the nonlinear term $\nl(\cdot)$ in the energy $\Ep$ has  poor (lower semi)continuity properties, and an arbitrary minimizing sequence for $\Ep$ (in $B_R$, say) will not be compact in a sufficiently strong topology to make the energy lower semicontinuous. 

Lemma~\ref{lem: LSC lemma} below shows that under rather strong convergence assumptions, $\Ep$ is lower semicontinuous. Here, the nonlinear term alone may {\it not} be lower semicontinuous, but by choosing the parameter $\tau$ in $\Ep$ to be sufficiently small and estimating the failure of lower semi-continuity in a quantitative way in \eqref{eqn: existence quant non LSC nl}, we can obtain lower semi-continuity of the full energy $\Ep$.

An arbitrary minimizing sequence will not subsequentially converge in the sense needed to apply  Lemma~\ref{lem: LSC lemma}. 
We replace an arbitrary minimizing sequence with a more regular one comprised of outward minimizers using Lemma~\ref{lemma: replace outer} below. Combining this with the Lipschitz bound for outward minimizers shown in Theorem~\ref{thm: Lip}, we can prove existence of minimizers for the problem in $B_R$; see Theorem~\ref{theorem: existence} below.
To then establish the existence of minimizers in all of $\R^n$ in Theorem~\ref{thm: existence 2}, the diameter bound of Corollary~\ref{cor: diameter bound} plays a key role.

 \begin{lemma}[Lower semicontunity lemma]
 \label{lem: LSC lemma}
   Fix  $\tau \in (0, \tau_0]$, and $R>0$. Let $\{U_j\}$ be a sequence of open sets in $B_R$ with $|U_j| \leq 2 \om_n$ satisfying \eqref{eqn: low base energy main}, and let  $\W \subset B_R$ be an open set. 
     Assume $\{ w_{U_j}\}$ and $w_\W$ are continuous on $\R^n$ and that $w_{U_j} \to  w_\W$ and uniformly. Then $|\W| \leq 2\om_n$ and 
     \begin{equation}\label{eqn: lower semicontinuity}
         \Ep(\W) \leq \liminf_j \Ep(U_j).
     \end{equation}
 \end{lemma}
 \begin{proof}
By continuity of the torsion functions, we may write $\W = \{w_\W >0\}$ and $U_j = \{ w_{U_j} >0\}$, and from the uniform convergence of the torsion functions, $ \liminf_j 1_{U_j\cap \W} \to  1_\W$ pointwise. So, $|\W| \leq \liminf_j |\W\cap U_j|$ by Fatou's lemma. Thus, setting 
\[
\gamma = \liminf_j |U_j|,
\]
we find that $ \liminf_j  \, |\W\Delta U_j |   = \gamma - |\W|\geq 0$. In particular, $|\W| \leq \gamma \leq    2 \om_n$.

Note that $w_{U_j} \to w_\W$ in $L^1(\R^n)$ and $L^2(\R^n)$. Since  $\tor(E) = -\frac{1}{2}\int_E w_E$ by \eqref{e: torsion and dir energy}, this guarantees that
 \begin{equation}
     \label{eqn: existence torsion continuity} \lim_{j}\tor(U_j) = \tor(\W).
 \end{equation}
Moreover, thanks to a theorem of \v{S}ver\'{a}k (see \cite[Theorem 3.2.5]{hp18}) we have  $ u_{U_j}^f \to u_{\W}^f$ in $L^2(B_R)$. 

To make use of this latter fact, we aim to apply \eqref{eqn: nl Lip}.
Since $t\mapsto \fv(t)$ is nondecreasing and $|\W| \leq \gamma$, we know $\tor(\W) + \fv(|\W|) \leq \liminf_j (\tor(U_j) + \fv(|U_j|)$. Thus
$\W$ satisfies the energy bound \eqref{eqn: low base energy main} as well.
Applying Lemma~\ref{lemma: qual stability} to $U_j$ and to $\W$ shows that 
$$
\Big(1-\frac{\baryEps}{2}\Big)\, \om_n \leq |U_j|, \, |\W| \leq \Big(1+\frac{\baryEps}{2}\Big)\, \omega_n,
$$
so in particular, $\gamma - |\W| \leq \baryEps.$ Since $ \liminf_j  \, |\W\Delta U_j |   = \gamma - |\W|$ as observed above, this means
  $\liminf_j |\W\Delta U_j |  \leq \baryEps\om_n.$ 
  So, by \eqref{eqn: nl Lip} and $ u_{U_j}^f \to u_{\W}^f$ in $L^2(B_R)$,
 	\begin{align} \label{eqn: existence quant non LSC nl}
 	\liminf_j \nl(U_j) - \nl(\W)  \geq   -\bar{C}_n\liminf_j|\W\Delta U_j| = -\bar{C}_n(\gamma - |\W|) \,
 	\end{align}
where $\bar{C}_n$ is the dimensional constant from \eqref{eqn: nl Lip}.
Next, from the definition of $\fv$ we have
\begin{align}
 \label{eqn: existence V term quant  LSC}
&  \liminf_j\, \vo(|U_j|) - \vo(|\W|)  \geq {\bar\eta}\, (\gamma - |\W|)\,.
 \end{align}

 Combining \eqref{eqn: existence torsion continuity}, \eqref{eqn: existence V term quant  LSC} and \eqref{eqn: existence quant non LSC nl}, we find that 
 	\begin{align*}
 		\liminf_j \Ep(U_j)- \Ep(\W)
 		& \geq (\bar\eta- \bar{C}_n\tau)(\gamma - |\W|) \geq 0,
 	\end{align*}
  where the final inequality follows from \eqref{eqn: tau 0} and the fact that we have chosen $\tau \leq \tau_0= {\bar\eta}/2\bar{C}_n$.
This completes the proof.     
 \end{proof}

The following lemma, which is similar to \cite[Lemma 6.1]{AKN1}, allows us to replace any minimizing sequence by a minimizing sequence of outward minimizers, all of whose torsion functions satisfy a uniform Lipschitz estimate.
Recall that the dimensional constant $\bar{\delta}$ was fixed in \eqref{eqn: parameter fix}.
\begin{lemma}\label{lemma: replace outer}
 For any $\tau \in (0, \min\{ \tau_0, \ccdel/2\}]$  and $R> 0$, the following holds. 
	 Let $\Omega\subset B_R$ be an open set  satisfying $|\Omega|\leq 2 \omega_n$ and
	\begin{equation}
		\label{eqn: small base energy existence section}
	\tor(\W) + \fv(|\W|) \leq \tor(B_1) + \bar{\delta}/2,
		\end{equation}
	there is an open set $U$ with $\W\subset U \subset B_R$ and $\Ep(U) \leq \Ep(\W)$ such that $U$ is an outward minimizer of $\Ep$ in $B_R$.
\end{lemma}
\begin{proof}
Given an open bounded set $E\subset B_R$ with $|E|\leq 2\omega_n$, define the quantity
	$$m_E = \inf \big\{\Ep(V) \ : \ V \text{ open}, E\subset V\subset B_R, \ |V|\leq 2 \omega_n \big\}.
	$$
Note that $m_E \in [\tor(B_1), \Ep(E)]$.  If $\Ep(\W)= m_\W$, then $\W=U$ is an outward minimizer satisfying the conclusions of the lemma. 

Otherwise, define a (possibly finite) sequence of nested open sets $\W=:U_0 \subset U_1 \subset \dots \subset U_{k-1} \subset U_{k}\subset \dots$ inductively as follows. 
For $k\geq 0$,  if $m_{U_k} = \Ep(U_k)$, terminate the sequence and take $U=U_k$ to conclude the proof of the lemma. If instead $m_{U_k} < \Ep(U_k)$,  choose an open set $U_{k+1}$ with $U_k \subset U_{k+1} \subset B(x_0,R)$ satisfying $|U_{k+1}|\leq 2\omega_n$ and $\Ep(U_{k+1} ) \leq (\Ep(U_k)+m_{U_k})/2.$ If the resulting sequence is finite, the proof is complete.

 Otherwise, let $U =\bigcup_{k=1}^\infty U_k$. Then $\W\subset U\subset B(x_0,R)$, and  $|U| = \lim |U_k|$ by monotone convergence so that $|U|\leq 2\om_n$ and $|U\Delta U_k|\to 0$. Since $m_U \geq m_{U_k}$  for each $k$ by containment, $m_U \geq \bar{m}:=\lim_{k\to \infty} \Ep(U_k)$, and $\bar{m} \leq \Ep(\W)$ since $\Ep(U_k)$ is decreasing.
We claim that $\Ep(U) =\bar{m}$,  which implies $U$ is an outer minimizer satisfying $\Ep(U)  \leq \Ep(\W)$, thus completing the proof. 
 
Since $\tau_0\leq\bar{\delta}/2$, by \eqref{eqn: small base energy existence section} we have $\Ep(\W) \leq \tor(B_1)+\bar{\delta}$, and thus
 $\tor(U_k) + \fv(|U_k|) \leq \Ep(U_k) \leq  \tor(B_1) + \bar{\delta}$
for all $k$. Since additionally  $|U_k\Delta U|= |U\setminus U_k| \leq \omega_n  \baryEps /{2}$ for $k$ large enough, by \eqref{eqn: nl Lip} we have
$$
\ |\nl(U_k) - \nl(U)|  \leq \bar{C}_n\,\Big( |U\setminus U_k| +\int |u_{U}^f - u_{U_k}^f|\Big)\,.
$$
 Then by Lemma~\ref{lemma: key estimate}, 
 $$\ |\nl(U_k) - \nl(U)| \leq \bar{C}_n\big(|U_k\Delta U| + \tor(U_k) - \tor(U)\big)
 $$
and in particular $\nl(U) \leq \liminf_k \nl(U_k) + \bar{C}_n\big(\liminf_k \tor(U_k) - \tor(U)\big)$.

Taking $w_{U_k}\in H^1_0(U)$ as competitors for $\tor(U)$ shows that $\liminf_k \tor(U_k) \geq \tor(U)$. Thus, 
\begin{align*}
\liminf_k \Ep(U_k) &-	\Ep(U) = \liminf_k \tor(U_k) -\tor(U) + \tau \big( \liminf_k \nl(U_k)  -\nl(U)\big) \\
	& \geq (1- \bar{C}_n\tau)\big(\liminf_k \tor(U_k) - \tor(U)\big)\geq \frac{1}{2}\big(\liminf_k \tor(U_k) - \tor(U)\big) \geq 0
\end{align*}
 where the penultimate inequality follows since $\bar \eta \leq 1$ and thus in \eqref{eqn: tau 0} we take  $\tau_0 \leq \tfrac{\bar\eta}{2\bar{C}_n} \leq \tfrac{1}{2\bar{C}_n}$.
 \end{proof}

 \subsection{Existence and basic properties}  Let
 \begin{equation}
     \label{eqn: tau fix}
     \bar{\tau} = \min\{\tau_0, \delta_D, \bar{\delta}/16\}
 \end{equation}
with $\tau_0$ as in \eqref{eqn: tau 0}, $\ccdel$  as in \eqref{eqn: parameter fix}, and $\delta_D$ from Corollary~\ref{cor: diameter bound}. Next we establish the existence of minimizers of $\Ep$ in a ball $B_R.$

\begin{theorem}\label{theorem: existence} For any $\tau \in (0,\bar{\tau}]$ and  $R>0$,  a  minimizer $\W$ of the variational problem 
	\begin{equation}
		\label{eqn: min problem existence section}
	\inf\{ \Ep(E) : E \subset B_R \text{ open}, \  |E|\leq 2\omega_n\}
	\end{equation}
	admits a minimizer. Any minimizer $\W$  of \eqref{eqn: min problem existence section} has $\text{diam}(\W) \leq 2$ and satisfies \eqref{eqn: low base energy main}.
    \end{theorem}

 \begin{proof}
 Let $\alpha \in \R$ denote the infimum in \eqref{eqn: min problem existence section} and a minimizing sequence $\{\W_j\}$ for \eqref{eqn: min problem existence section}. 
 . By Lemma~\ref{prop: base summary}, $\alpha \geq \tor(B)$. On the other hand,  Note that $\nl(\W) \leq \sqrt{2}+ \int_{\R^n} |u_\W^f - u_{B(x_\W)}^f|^2 \leq 4$ by \eqref{eqn: bound for tor fn}, so taking $B_1$ as a competitor in \eqref{eqn: min problem existence section} shows $\alpha \leq \tor(B) +4\bar{\tau}.$
Then, since we chose $4\bar{\tau} \leq \bar{\delta}/4$ in \eqref{eqn: tau fix}, 
  we see that $\W_j$ satisfies the energy bound \eqref{eqn: small base energy existence section} assumed in the statement of Lemma~\ref{lemma: replace outer} as soon as $j$ is large enough that $\Ep(\W_j) \leq \alpha + \bar{\delta}/4$.

 Apply Lemma~\ref{lemma: replace outer} to obtain a new minimizing sequence $\{U_j\}$ for \eqref{eqn: min problem existence section} where $U_j$ is an outward minimizer of $\Ep$ in $B_R$ and $U_j$  satisfies the energy bound \eqref{eqn: low base energy main} (since $\Ep(U_j)\leq \Ep(\W_j)\leq \tor(B)+\ccdel/2$).
 By Theorem~\ref{thm: Lip}, the torsion functions $w_{U_j}$ are equi-Lipschitz in $B_R$ and thus subsequentially converge uniformly to a Lipschitz function $w$. Let 
 \[
 \W = \{ w>0\} \subset B_R.
 \]
 The set $\W$ is open since $w$ is continuous.
By uniform convergence, for any compact subset $K\subset \W$, we have $K \subset U_j$ for large $j$, so standard elliptic estimates show  $w_j \to w$ in $C^2_{loc}(\W)$ and thus $w=w_\W$. 
By the lower semicontinuity lemma, Lemma~\ref{lem: LSC lemma}, 
\[
\liminf_{j}\Ep(U_j) \geq \Ep(\W)
\]
and so  $\W$ is a minimizer of $\Ep$ in $B_R$. 

Finally, $\W$ is, in particular, an inward minimizer in $B_R$, and $\Ep(\W) = \alpha \leq \tor(B) +  \delta_D$ thanks to  \eqref{eqn: tau fix}. Thus, $\text{diam}(\W)\leq 2$ by  Corollary~\ref{cor: diameter bound}. This completes the proof.
    \end{proof}
Now, focusing on functions $f$ that are compactly supported, we prove existence of minimizers on $\R^n$ by taking a sequence of minimizers from Theorem~\ref{theorem: existence} on balls with radii tending to infinity. 
    \begin{proof}[Proof of Theorem~\ref{thm: existence 2}]
For each $j \in \mathbb{N}$, let $U_j$ be a minimizer of \eqref{eqn: min problem existence section} with $R =j.$ Theorem~\ref{theorem: existence} guarantees that such a minimizer exists, has  $\text{diam}(U_j)\leq 2$, and satisfies \eqref{eqn: low base energy main}. Choose $x_j$ so that $U_j \subset B_2(x_j).$

{\it Case 1:} Suppose $\limsup_j |x_j| <\infty$. Then there is  $R_0>0$ such that $U_j \subset B_{R_0}$ for all $j \in \mathbb{N}.$ Thus for $j$ large, $\text{dist}(U_j, \partial B_{R_j}) \geq 2$, and so by Theorem~\ref{thm: Lip}, the torsion functions $w_{U_j}$ are equi-Lipschitz. By Arzel\`{a}-Ascoli, up to a subsequence, the torsion functions $w_{U_j}$ converge uniformly to a Lipschitz function $w$ supported in $B_{R_0}$. Let 
\[
\W  =\{w>0\} \subset B_{R_0}
\]
 As in the proof of Theorem~\ref{theorem: existence}, the uniform convergence together with elliptic estimates show that $w= w_\W.$ By Lemma~\ref{lem: LSC lemma},  $|\W| \leq 2 \om_n$ and 
 \[
 \Ep(\W) \leq \liminf_j \Ep(U_j).
 \]
 This shows that $\W$ is a minimizer of \eqref{eqn: MAIN min prob}, since for any open, bounded competitor $U$ with $|U|\leq 2\om_n,$ there is some $j_0$ large enough so that $U \subset B_{R_j}$ for all $j \geq j_0$, and thus 
 \[
 \Ep(U_j) \leq \Ep(U)\,.
 \]
 Taking the liminf over $j$ completes the proof of existence.

 {\it Case 2:} Suppose $\limsup_j |x_j| = +\infty$, so  in particular $U_j \subset B_2(x_j) \subset \R^n \setminus B_\rho$  for $j$ large. The functional $\Ep$ is translation invariant outside of $B_\rho$, in the sense that if $\W$ and $\W + z$ are both contained in $\R^n \setminus B_\rho$, then $\Ep(\W) = \Ep(\W + z).$ Therefore, the sets $\W_j  = U_j + z_j$ where $z_j = x_j + (\rho + 2)e_1$ have $\W_j \subset B_2((\rho+2) e_1))\subset \R^n \setminus B_\rho,$ and thus $\Ep(\W_j ) = \Ep(U_j)$. So, each $\W_j$ is another minimizer of \eqref{eqn: min problem existence section} with $R =j$. Applying the argument of Case 1 to $\{ \W_j\}$ completes the proof of existence.

 In summary, we have shown the existence of a minimizer $\W$ of \eqref{eqn: MAIN min prob} which, by taking $B$ as a competitor, satisfies the energy bounds \eqref{eqn: low base energy main} and \eqref{eqn: low energy for diam bound} thanks to the choice of $\tau$ in \eqref{eqn: tau fix}. Since $\W$ is an inward and outward minimizer in any $B_R$, by Theorems~\ref{t:lb}, Corollary~\ref{cor: diameter bound}, and Theorem~\ref{thm: Lip}, we see that $\W$ has $\text{diam}(\W)\leq 2$ and the torsion function enjoys the nondegeneracy and Lipschitz bounds \eqref{eqn: NDandLip}.

 Finally, the measure theoretic estimates of Theorem~\ref{thm: existence 2} follow from  now-classical arguments for minimizers of Bernoulli-type energies that satisfy nondegeneracy (Theorem~\ref{t:lb}) and Lipschitz estimates (Theorem~\ref{thm: Lip}). In particular, the proof is a nearly verbatim repetition of the proofs of Lemma 7.1, Corollary 7.3, and Corollary 7.6 in \cite{AKN1}, and thus we omit the details here.
    \end{proof}

 \section{Free Boundary Estimates, Part 2: Stationarity Conditions and Free Boundary Regularity}
\label{sec: FB2}
Throughout this section, let  $\tau \in (0,\bar{\tau}],$ where $\bar{\tau}$ is the dimensional constant from Theorem~\ref{thm: existence 2}.  Let $f \in C^\infty_c(\R^n)$ be a nonnegative function with $\|f \|_{L^{\infty}}\leq 1$.
Let $\W$ be a minimizer of the variational problem \eqref{eqn: MAIN min prob}, whose existence was shown in Theorem~\ref{thm: existence 2}. Recall from Theorem~\ref{thm: existence 2} that $\text{diam}(\W) \leq 2$, so in particular the truncated barycenter $x_\W$ agrees with the classical barycenter $\frac{1}{|\W|}\int_\W x\,dx.$ Moreover, $\W$ is a set of finite perimeter and an NTA domain with NTA constant $K$ depending only on $n$,  and there exist dimensional constants $c_*,C_*$ and $r_0$ depending on $n$ such that if $x \in \partial \W$ and  $r\in (0,r_0]$, then 
\[
 cr \leq \sup_{B_r(x)} w_{\W} \leq Cr\,.
\]
{These facts allow us to employ the inhomogeneous boundary Harnack principle \cite[Theorem 2.2]{aks23} which we will utilize throughout the remainder of the paper. We emphasize that the inhomogeneous boundary Harnack principle will not hold for arbitrary NTA domains; however, the growth condition for the torsion function is a sufficient condition for the inhomogeneous boundary Harnack principle to hold, see Lemma 7.7 in \cite{AKN2}. For convenience, we restate the result here in the context of our setting. }
\begin{theorem} \label{t:bhpf} Fix  $x \in \partial \W$ and  $0<r\leq r_0$.
Assume that $|\Delta v|\leq 1$ in $\W \cap B_r(x)$ and $v=0$ on $\partial \W \cap B_r(x)$. Then there exists a constant $C$ and $0<\gamma<1$ depending on $n$ such that 
    \[
    \left\|\frac{v}{w_{\W}} \right\|_{C^{0,\gamma}(\overline{B_{r/2}(x)\cap \W)}}
    \leq C \sup_{B_{r/2}(x) \cap \W} |v|. 
    \]
\end{theorem}

{\begin{remark} \label{r:bhpf}
\rm{
    By applying the above result to $u_{\W}^f/w_{\W}$, it is clear we may replace  the torsion function $w_{\W}$ with $u_{\W}^f$ in the above theorem (provided we divide the constant $C$ by $\sup_{B_{r/2}\cap \W} u$) since $0\leq f\leq 1$. Also, Theorem 2.2 in \cite{aks23} assumes $v \geq 0$ and only gives the bound on the quotient. However, since the Laplace operator is linear, standard techniques, see for instance \cite{ac85}, give the H\"older continuity of the quotient and do not require the numerator to have a sign. 
    }
\end{remark}}

In view of Remark~\ref{r:bhpf} and Remark~\ref{remark: Lip poisson equation solution}, there are 
dimensional constants $C, c>0$ such that
for $x \in \partial \W$ and  $r\in (0,r_0]$,
\begin{equation}
    \label{eqn: lin growth uf}
     cr \leq \sup_{B_r(x)} u_{\W}^f \leq Cr\,.
\end{equation}
\begin{remark}
    {\rm For functions $f$ that are not constantly equal to $1$ outside a large ball, the estimates shown in this section will also hold away from $\partial B_R$ for a minimizer $\W_R$ of the variational problem \eqref{eqn: min problem existence section}.
    }
\end{remark}

\subsection{Euler-Lagrange equation}
Our first task is to compute the first variation of the energy functional $\Ep$, which we carry out in several steps.

For a vector field $T \in C^1_c(\mathbb{R}^n; \R^n)$, let $\Phi_t(x) = \Phi(t,x): (-\e, \e) \times \R^n \to \R^n$ is the flow generated by $T$, i.e. the solution to $\frac{d}{dt} \Phi_t(x) = T(\Phi_t(x))$
with $\Phi_0(x) = x.$ Let $\Omega_t = \Phi_t(\Omega)$.

Since $f$ is smooth, in particular $f \in H^1(\W)$. Thus  by \cite[Theorem 5.3.1]{hp18}, 
$t \mapsto u_{\W_t}^f$ is differentiable in $L^2$ and the derivative  $\dot{u}_{\W}^f$ at $t=0$ satisfies
\begin{equation}
    \label{eqn: equation for dot u}
\begin{cases}
 -\Delta \dot{u}_{\W}^f =0 \quad \text{ in } \W \\
 \dot{u}_{\W}^f + T \cdot \nabla u_{\W}^f \in H_0^1(\W),
\end{cases}
\end{equation}
and   $t \mapsto \int ( u_{\Omega_t}^f)^2$ is differentiable at $t=0$ with
\begin{equation}\label{eqn: deriv of L2 norm}
\frac{d}{dt}\Big|_{t=0}\, \int \left(u_{\Omega_t}^f\right)^2 =2 \int  u_{\Omega}^f\,  \dot{u}_{\W}^f.
\end{equation}
To express \eqref{eqn: deriv of L2 norm} as a boundary integral with a clearer $T$-dependence,  consider the auxiliary equation 
\begin{equation} \label{e:aux}
 \begin{cases}
   -\Delta p = u_{\W}^f \quad &\text{ in } \W, \\
   p=0 \quad &\text{ on } \partial \W. 
 \end{cases}
\end{equation}
If $\W$ were smooth, then integrating by parts twice on the right-hand side of \eqref{eqn: deriv of L2 norm} using \eqref{eqn: equation for dot u} would show that $\int u_{\Omega}^f\,  \dot{u}_{\W}^f =  \int_{\partial \W} |\na p| \,| \nabla u_{\W}^f| \, T\cdot\nu_\W \,d \mathcal{H}^{n-1}$. In Lemma~\ref{lem: square variation} below, we show that this formal computation can be made rigorous, interpreting the terms $|\nabla p|$ and $T\cdot \nabla u_\W^f$ as non-tangential limits. As a first step in this direction, we prove the following lemma.
\begin{lemma} \label{l:NTlimits}
 For $\mathcal{H}^{n-1}\llcorner{\partial \W}$ a.e. point, the nontangential maximal function $(\na u_\W^f)^*$ is in $L^\infty(\partial \W, d\mathcal{H}^{n-1})$ and the non-tangential limit of $\nabla u_{\W}^f$ exists and is equal to $-|\na u_\W^f| \, \nu_\W$ with  $|\na u_\W^f| \in L^\infty(\partial\W, \mathcal{H}^{n-1})$.
\end{lemma}

\begin{proof}
 We recall from Theorem~\ref{thm: Lip} and Remark~\ref{remark: Lip poisson equation solution} that $\| \nabla u_{\W}^f \|_{L^{\infty}(\mathbb{R}^n)} \leq C$. From this it follows that $(\na u_\W^f)^*\in L^\infty(\partial \W, d\mathcal{H}^{n-1})$ at any point on $\partial \W$ for which the nontangential limit of $\nabla u_{\W}^f$ exists, its modulus is bounded by the same constant. 
 
 Let $x \in \partial \Omega$, and let $v$ be the harmonic replacement of $u_{\W}^f$ in $B_r(x)\cap \Omega$ and extended to be zero on $B_r(x) \setminus \W$. From Theorem \ref{t:bhpf}  the ratio $v/u_{\W}^f \in C^{0,\alpha}(\overline{B_{r/2}(x)\cap \Omega})$. Then $v$ has the same linear growth \eqref{eqn: lin growth uf} from $\partial \Omega$. Consequently, we may utilize the results in \cite{AC81} for nonnegative harmonic functions which vanish on the boundary and have linear growth from the boundary. Specifically, by  \cite[Remark 4.2 and Theorem 4.3]{AC81},
 \[
 c\rho^{n-1} \leq \int_{B_{\rho}(x_0)} d \lambda \leq C \rho^{n-1},
 \]
 where $\lambda := -\Delta v$
 is a positive Radon measure with support in $\partial \W \cap B_{r}(x)$. If $B_{\rho}(x_0)\cap \partial \W \subset B_{r/2}(x) \cap \partial \W$, then again from the H\"older continuity of the ratio $v/u_{\W}^f$, the distributional Laplacian 
of $ u_{\W}^f$ satisfies 
 \begin{equation}\label{eqn: estimates for laplacian}
  c\rho^{n-1}\leq -\Delta u_{\W}^f |_{\partial \W} \leq C\rho^{n-1}.
  \end{equation}
 So, by the perimeter density estimates \eqref{eqn: perim density}, the measures $-\Delta u_\W^f\llcorner{\partial \Omega}$ and $\mathcal{H}^{n-1}\llcorner{\partial \W}$ are mutually absolutely continuous, and by the Radon-Nikodym theorem, 
 \begin{equation}\label{eqn: denisty of laplacian}
 -\Delta u_{\Omega}^f \llcorner{\partial \Omega} = g \mathcal{H}^{n-1}\llcorner{\partial \W},
  \end{equation}
for a function  $g$ defined $\mathcal{H}^{n-1}$ a.e on $\partial \W$.

Recall $\mathcal{H}^{n-1}(\partial \W \setminus \partial^*\W) =0$ from Proposition~\ref{prop: basic properties}, so for $\mathcal{H}^{n-1}$-a.e. $x$ in $\partial \W$, $g$ is defined and $x \in \partial^*\W$. Let  $x_0$ be such a point.
By translation and rotation we may assume that $x_0=0$ the origin and  the inward unit normal of $\partial \W$ at $x_0$ is $e_n$, i.e. $\nu_\W(x_0) =-e_n$. Consider the rescaling 
 \[
 u_r(x):= \frac{u_{\Omega}^f(rx)}{r}. 
 \]
From the linear growth estimates, we have that $\| u_r \|_{C^{0,1}(B_R)} \leq C$. Thus, for a subsequence $r_j \to 0$ we obtain a limiting function $u_{r_j} \to u_{\infty}$ in $C^{0,\beta}$ on uniformly on compact subsets of $\mathbb{R}^n$ with $\Delta u_{\infty}=0$ on $\{u_{\infty}>0\}$. Since $x_0 \in \partial^* \W$, we have that $\{u_{\infty}>0\}=\{x_n>0\}$. From the Liouville theorem, $u_\infty(x) = a x_n^+$ for a positive constant $a$. 
By integrating against test functions along the sequence, we find that $ -\Delta u_\infty= g(x_0) \mathcal{H}^{n-1}|_{\{x_n=0\}}$. This means $a = g(x_0)$, and in particular,  
this limit $u_\infty$ is independent of the sequence $r_j$ chosen. Thus, the convergence is uniform in $C^1$ on $\{1/2<|x|<2\}\cap \{x \cdot e_n > \delta\}$ for any small $\delta >0$. Thus, $\nabla u_{\Omega}^f$ has a non-tangential limit equal to $g(x_0) e_n$ at $x_0$. This completes the proof with $g= |\na u_\W^f|.$
\end{proof}

Next we use Lemma~\ref{l:NTlimits} to establish non-tangential limits of the function $\dot{u}_\W^f$ from \eqref{eqn: equation for dot u}. 
\begin{lemma} \label{l:NTlimits2} For $\mathcal{H}^{n-1}\llcorner{\partial \W}$ a.e. point, nontangential limits of $\dot{u}_{\W}^f$ exist and are equal to  $-\nabla u_\W^f\cdot T =|\nabla u_\W^f|\nu_\W \cdot T \in L^\infty(\partial \W, \mathcal{H}^{n-1})$,   and the non-tangential maximal function $(\dot{u}_{\W}^f)^*$ of $\dot{u}_{\W}^f$ is in $L^\infty(\partial \W, \mathcal{H}^{n-1})$.
\end{lemma}

\begin{proof}
Let $v=\dot{u}_{\W}^f + T \cdot \nabla u_{\W}^f \in H_0^1(\Omega)$ and extend $v$ by zero so that $ v \in H^1(\mathbb{R}^n)$.
 From Theorem 1 in Chapter 4.8 in \cite{eg15}, 
$v$ has a $2$-quasicontinuous representative $v^*$ such that for $\mathcal{H}^{n-1}$ a.e. $x \in \partial \W$, 
\[
 \lim_{r \to 0}  \fint_{B_r(x)} |v-v^*(x)|^2=0. 
\]
Since $v \in H_0^1(\W)$, for $x \in \partial^* \W$, we have $\lim_{r\to 0} \frac{|B_r(x)\setminus \W|}{|B_r(x)|} = 1/2$ and therefore 
\[
0\geq \lim_{r \to 0}  \fint_{B_r(x)\setminus\W} |v-v^*(x)|^2
= \lim_{r \to 0} \fint_{B_r(x)\setminus \W} |v^*(x)|^2
\geq c |v^*(x)|^2\,.
\]
Then $v^*=0$ for $\mathcal{H}^{n-1}$ a.e. $x \in \partial^* \W$.

Since $x\in \partial^*\W$ and thanks to \eqref{eqn: cone}, for each $\e>0$ we can choose $r>0$ small enough so that $\mathcal{C} \cap  B_r(x) \subset \W$ for the cone $\mathcal{C} = \{ y : y \cdot \nu_\W(x)<-\e|y|\}$  touching $x$.
By integrating on such an interior cone $\mathcal{C} \subset \W$ we have 
\[
0 = \lim_{r \to 0}\fint_{B_r(x)} |v-v^*(x)|^2 \geq c \fint_{\mathcal{C}\cap B_r(x)} |v-v^*(x)|^2=c\fint_{\mathcal{C}} |v|^2 . 
\]
Now fix $\epsilon >0$, and choose $r$ small enough so that 
\[
\fint_{\mathcal{C}\cap B_r(x)} |v|^2 < \epsilon, 
\]
as well as 
\[
|  T \cdot \nabla u_{\W}^f (y) -  T\cdot \nabla u_{\W}^f(x) |  < \epsilon \quad \text{ for any } y \in \mathcal{C} \cap B_{r}(x),
\]
which we may do thanks to Lemma \ref{l:NTlimits}. Then 
\[
\begin{aligned}
&\fint_{\mathcal{C}\cap B_{r}(x)} |\dot{u}_{\W}^f(y) +T\cdot \nabla u_{\W}^f(x)|^2 \ dy \\
&\leq 
2\fint_{\mathcal{C}\cap B_r(x)} |\dot{u}_{\W}^f(y)+T\cdot \nabla u_{\W}^f(y)|^2 \ dy
+2 \fint_{\mathcal{C}\cap B_r(x)} |T\cdot \nabla u_{\W}^f(x)-T\cdot\nabla  u_{\W}^f(y)|^2 \ dy\\
&= \fint_{\mathcal{C}\cap B_r(x)} |v(y)|^2 \ dy + \fint_{\mathcal{C}\cap B_r(x)} |T\cdot \na u_{\W}^f(y)-T\cdot \na u_{\W}^f(x)|^2 \ dy\\
&\leq 2 \epsilon.
\end{aligned}
\]
Now let $\mathcal{C}'$ be an interior cone to $\mathcal{C}$. From the interior $L^2-L^{\infty}$ estimate for harmonic functions, we conclude 
\begin{equation}
    \label{eqn uniform bound}
\| \dot{u}_{\W}^f + T \cdot\na u_{\W}^f(x) \|^2_{L^{\infty}(\mathcal{C'}\cap (B_r(x)\setminus B_{r/2}(x))} \leq C(\mathcal{C},\mathcal{C}')\fint_{\mathcal{C}\cap B_r(x)} |\dot{u}_{\W}^f(y) +T\cdot \na u_{\W}^f(x)|^2 \ dy \leq 2 \epsilon. 
\end{equation}
It is not hard to show that using \eqref{eqn: cone} that for any $\beta>0$, $\Gamma_\beta(x) \cap B_r(x)$ is contained in $\mathcal{C}'\cap B_r(x)$  for some interior cone $\mathcal{C}'$ touching $x$, so \eqref{eqn uniform bound} shows that  $\dot{u}_{\W}^f$ has a nontangential limit  at $x$ equal to $- T\cdot \nabla_\W^f(x)$, which by Lemma~\ref{l:NTlimits} is equal to $| \nabla u_\W^f(x)| \nu_\W(x)\cdot T(x)$. Note that \eqref{eqn uniform bound} together with the Lipschitz bound $\| \na u_\W^f\|_{L^\infty}$ guarantee that $(\dot{u}_{\W}^f)^* \in L^\infty(\partial \W, \mathcal{H}^{n-1}).$ This completes the proof.
\end{proof}

Since $u_{\W}^f, w_{\W}$ and $p$ from \eqref{e:aux} are zero on $\partial \W$, then we may apply the boundary Harnack principle and gain regularity for the quotients of the derivatives as well. 

\begin{lemma} \label{l:gradhigher}
 The quotients $|\nabla p|/|\nabla w_{\W}|$ and $|\nabla p|/|\nabla u_{\W}^f|$ are H\"older continuous up to the boundary. Consequently, both $p$ and $\nabla p$ have nontangential limits $\mathcal{H}^{n-1}\llcorner\partial \W$ a.e., and also $|\nabla p|$ is bounded on $\overline{\W}$. 
\end{lemma}

\begin{proof}
 By translation and rotation, we may assume that $x \in \partial^* \W$ is the origin with inward unit normal $e_n$. Since both $u_{\W}^f$ and $p$ are zero on $\partial \W$, then 
 \[
 \frac{p(te_n)}{u_{\W}^f(t e_n)} = \frac{p(te_n)-p(0)}{u_{\W}^f(t e_n) - u_{\W}^f(0)}. 
 \]
 Since the limit of the left-hand side exists by the boundary Harnack principle Theorem \ref{t:bhpf}, and the denominator of the right hand side exists. Then the numerator of the right hand side exists. Moreover, $|p/u_{\W}^f|=|\nabla p |/ |\nabla u_{\W}^f|$ on $\partial \W$. From Theorem \ref{t:bhpf} it follows that $|\nabla p |/ |\nabla u_{\W}^f|$ is H\"older continuous up to and along $\partial \W$. The same argument holds true with $w_\W$ replacing $u_{\W}^f$. 
\end{proof}

The preceding lemmas allow us to compute the first variations of the terms in the energy functional $\Ep$ at $\W$, starting with the following.
\begin{lemma}\label{lem: square variation} Let $p$ be as in \eqref{e:aux}. 
    The derivative in \eqref{eqn: deriv of L2 norm} is equal to
    \begin{equation}
        \label{eqn: first var f}
\frac{d}{dt}\Big|_{t=0}\, \int \big( u_{\Omega_t}^f\big)^2 = 2 \int_{\partial \W} |\na p|\, |\nabla u_\W^f| \, T\cdot \nu_\W \,d \mathcal{H}^{n-1}\,.
    \end{equation}
\end{lemma}

\begin{proof}[Proof of Lemma~\ref{lem: square variation}]
In view of \eqref{eqn: deriv of L2 norm}, we need to show 
\begin{equation}\label{eqn: intermediate deriv step}
    \int_\W  u_{\Omega}^f \, \dot{u}_{\W}^f = \int_{\partial \W} |\na p| \, |\na u_\W^f| \, T\cdot \nu_\W \, d\mathcal{H}^{n-1}
\end{equation}
where $\dot{u}_{\W}^f$ is defined in \eqref{eqn: equation for dot u}. As mentioned above, this formally follows by integrating by parts twice, but to carry out this integration by parts rigorously, we first  regularize $\dot{u}_\W^f$.
 Let $v_{\epsilon}$ be a smooth approximation, so that 
 \begin{equation}\label{eqn: approx converge}
 \|v_{\epsilon} - |\nabla u_{\W}^f|\, \nu_\W\cdot T) \|_{L^q(\partial \W)} \to 0, 
  \end{equation}
 for any $1\leq q < \infty$, with $\| v_{\epsilon} \|_{L^{\infty}(\partial \W)} \leq C$. By Lemma~\ref{lem: poisson kernel}, for each $\e>0$, the function 
 $$\dot{u}_{\epsilon}(x) := \int_{\partial \W} K(x,y) v_\e(y) \, d\mathcal{H}^{n-1}(y)$$
 is harmonic in $\W$ with $ \dot{u}_{\epsilon} = v_{\epsilon}$ on $\partial \W$ in the sense of nontangential limits. 
Similarly, it follows from Lemma~\ref{l:NTlimits2} and Lemma~\ref{lem: poisson kernel} that the solution $\dot{u}_\W^f$ of \eqref{eqn: equation for dot u} has the representation 
\begin{equation}\label{eqn: PK rep u dot}
\dot{u}_\W^f = \int K(x,y) (|\na u_\W^f|\, \nu_\W \cdot T)(y) \, d \mathcal{H}^{n-1}(y)\,.
\end{equation}
 From \eqref{eqn: PK rep u dot} and the dominated convergence theorem, we see that $\dot{u}^\e \to \dot{u}_\W^f$ pointwise as $\e \to 0$. 
 It then follows, again by  the dominated convergence theorem, that 
 \begin{equation}
     \label{eqn: approx converge 2}
 \lim_{\epsilon \to 0}\int_{\W} \dot{u}_{\epsilon}\,  u_{\W}^f 
 = \int_{\W} \dot{u}_{\W}^f\,  u_{\W}^f. 
  \end{equation}
 Now recalling $p$ from \eqref{e:aux}, we let  
  $Z = p \nabla \dot{u}_{\epsilon}-\dot{u}_{\epsilon}\nabla p$, then we note that $Z$ has nontangential limits $\mathcal{H}^{n-1}\llcorner \partial \W$ a.e. from Lemmas \ref{l:NTlimits2} and \ref{l:gradhigher}, with bounded nontangential maximal function, and $\text{div}(Z) \in L^2(\W)$. Then 
  we applying the divergence theorem of Lemma~\ref{l:fancy div thm} to $Z$ we obtain 
 \[
 \int_{\W} \dot{u}_{\epsilon}\,  u_{\Omega}^f 
 = \int_{\W} \dot{u}_{\epsilon} (-\Delta p)
 = \int_{\W} \dot{u}_{\epsilon} (-\Delta p) + p \Delta \dot{u}_{\epsilon} 
 = \int_{\W} \text{div} Z  
 = \int_{\partial \W} Z \cdot \nu_\W \, d\mathcal{H}^{n-1}
 = \int_{\partial \W} v_{\epsilon} \,|\nabla p| \,d\mathcal{H}^{n-1}\,.
 \]
 Then \eqref{eqn: intermediate deriv step} follows from letting $\epsilon \to 0$ and applying \eqref{eqn: approx converge} and \eqref{eqn: approx converge 2}.
\end{proof}

A similar proof gives the following result. 
\begin{lemma} \label{l:diff torsion}
    Then $t \mapsto \tor(\W_t)$ is differentiable at $t=0$ and 
    \begin{equation}
        \label{eqn: first var torsion}
\frac{d}{dt}\Big|_{t=0}\, \tor(\W_t) = -\frac{1}{2} \int_{\partial \W} |\na w_\W|^2 \, T\cdot \nu_\W \, d\mathcal{H}^{n-1}.
    \end{equation}
Here  $|\na w_\W|^2$ is understood in the sense of nontangential limits.
\end{lemma}

\begin{proof}
    From \cite[Theorem 5.3.1]{hp18}, $t \mapsto w_{\W_t}$ is differentiable in $L^2$ at $t=0$ with derivative $\dot{w}_\W$ satisfying 
    \[
     \begin{cases}
      &\Delta \dot{w}_{\W}=0 \quad \text{ in } \W, \\
      &\dot{w}_{\W} + T \cdot \nabla w_{\W} \in H_0^1(\W)
     \end{cases}
    \]
    and $\frac{d}{dt}|_{t=0} \int w_{\W_t} =\int \dot{w}_{\W_t} $
    Utilizing Lemmas \ref{l:NTlimits} and \ref{l:NTlimits2} with $f \equiv 1$, we have that the boundary values are defined in the sense of nontangential limits. Using the same ideas as in Lemma \ref{lem: square variation}, the following equalities are rigorously justified:
    \[
    \int_{\W} \dot{w}_{\W} 
    = \int_{\W} \dot{w}_{\W} (- \Delta w_{\W})
    = \int_{\W} \dot{w}_{\W} (- \Delta w_{\W}) + w_{\W} \Delta \dot{w}_{\W}
    = \int_{\partial^* \W} (-T \cdot \nabla w_{\W})(- \nabla w_{\W} \cdot \nu_{\W})
    = \int_{\partial^* \W} |\nabla w_{\W}|^2 \, T \cdot \nu_{\W}. 
    \]
    Recalling that $\text{tor}(\W_t)= -\frac{1}{2}\int_{\W_t} w_t$, we obtain the result. 
\end{proof}

Next, as a step toward computing the first variation of the term $\nl(\cdot)$ in $\Ep(\cdot)$ at $\W$, we differentiate the truncated barycenter along a variation. Recall that since $\text{diam}(\W) \leq 2 \leq 100$, the truncated barycenter $x_\W$ agrees with the classical barycenter, and the same is true for $x_{\W_t}$ for $t$ sufficiently small.
\begin{lemma}\label{lemma: Barycenter Derivative}
	 The function $t \mapsto x_{\Omega_{t}}$ defines a $C^1$ curve and 
\[
\dot{x}_\W :=\frac{d}{d t}\Big|_{t=0} x_{\Omega_t}=\int_{\partial^*\Omega} {A}_{\Omega} \, T \cdot \nu_{\Omega}\, d \mathcal{H}^{n-1}
\]
where 
\begin{equation}
    \label{eqn: A omega def}
{A}_{\Omega}(x):=\frac{1}{|\Omega|}[x- x_{\Omega}], \qquad x \in \R^n.
\end{equation}
\end{lemma}
\begin{proof}
	By the classical Jacobian computation and Hadamard's
variational formula respectively, we know that
$\frac{d}{d t}|_{t=0}|\Omega_t|  =\int_{\partial^* \Omega} T \cdot \nu_{\Omega} \, d \mathcal{H}^{n-1}$  and
$\left.\frac{d}{d t}\right|_{t=0} \int_{\Omega_t} x  \,dx =\int_{\partial^* \Omega} x \ T \cdot \nu_{\Omega} d \mathcal{H}^{n-1}$.
Therefore,
\begin{align*}
\left. \frac{d}{d t}\right|_{t=0} \frac{1}{\left|\Omega_{t}\right|} \int_{\Omega_{t}} x\, d x  =-\frac{1}{|\W|^2}\int_{\partial^* \Omega} &T \cdot \nu_{\Omega} \,d\mathcal{H}^{n-1} \cdot \int_{\Omega} x\,  d x+\frac{1}{|\Omega|} \int_{\partial^* \Omega} x\  T \cdot \nu_{\Omega} \, d \mathcal{H}^{n-1} \\
& =  \int_{\partial^* \Omega}  \frac{1}{|\W|}\left[x -x_\Omega \right]   T \cdot \nu_{\Omega}  \, d \mathcal{H}^{n-1}  =\int_{\partial^* \Omega} A_{\Omega}\,  T \cdot \nu_{\Omega} \, d \mathcal{H}^{n-1} .
\end{align*}
\end{proof}

Next we compute the first variation of the main sub-term of $\nl(\W)$.
\begin{lemma} \label{l:distterm}
 The map 
 \[
 t \mapsto \beta^2(\Omega_t):=\int \big(u_{\W_t}^f - u_{B(x_{\W_t})}^f \big)^2
 \]
 is differentiable at $t=0$ with  
 \[
 \frac{d}{dt} \beta^2(\W_t) \Bigr|_{t=0}=2\int_{\partial^* \W} \big( a_1(x) |\nabla w_{\W}(x)|^2 - V\cdot A_\W(x)\big) T\cdot \nu_\W \ d \mathcal{H}^{n-1}(x),
 \]
  where $A_\W$ is the function defined in \eqref{eqn: A omega def}, $V$ is a vector depending on $\W$ and $f$ with $|V|\leq C$ and $a_1: \partial^*\W \to \R$ is a function depending on $\W$ and $f$ with
 \[
 \|a_1 \|_{C^{0,\alpha}(\partial^* \W)} \leq C,
 \]
for a constant $C>0$ depending only on $n$ and  $\| f\|_{L^{\infty}} \leq 1$. 
\end{lemma}

\begin{proof}
 Let $\dot{x}_\W = \int_{\partial^*\Omega} {A}_{\Omega} \ T \cdot \nu_{\Omega}\, d \mathcal{H}^{n-1}$ be the initial velocity of the barycenter computed in  Lemma \ref{lemma: Barycenter Derivative}. Then $t \mapsto u^f_{B(x_{\W_t})}$ is differentiable at $t=0$ with derivative $\dot u^f_{B(x_\W)}$ satisfying
     \[
     \begin{cases}
      -\Delta \dot{u}^f_{B(x_\W)}=0  &  \text{ in } B(x_\W), \\
     \ \ \ \ \  \dot{u}_{B(x_\W)}^f = |\nabla u_{B(x_\W)}^f|\,  \dot{x}_\W \cdot \nu_{B(x_\W)} & \text{ on }\partial B(x_\W),
     \end{cases}
    \] 
 and we have 
 \[
 \frac{d}{dt}\Bigr|_{t=0} \int \left(u_{\W_t}^f - u_{B(x_{\W_t})}^f \right)^2
 = 2\int \left(u_{\W}^f - u_{B(x_{\W})}^f \right)\left(\dot{u}_{\W}^f - \dot{u}_{B(x_{\W})}^f \right).  
 \]
Using the ideas in Lemma \ref{lem: square variation}, we have 
 \begin{equation}
     \label{eqn: first var term 1}
 \int \left(u_{\W}^f - u_{B(x_{\W})}^f \right)\dot{u}_{\W}^f
 = \int_{\W}\left(u_{\W}^f - u_{B(x_{\W})}^f \right)\dot{u}_{\W}^f
 = \int_{\partial^* \W} |\nabla p_1| \, |\nabla  u_{\W}^f| \, T \cdot \nu_\W \, d\mathcal{H}^{n-1} \,.
 \end{equation}
 where $p_1$ solves 
 \begin{equation}
     \label{eqn: aux PDE}
 \begin{cases}
      -\Delta p_1 = u_{\W}^f - u_{B(x_{\W})}^f & \text{ in } \W, \\
  \ \ \ \ \  p_1 = 0 &  \text{ on } \partial \W. 
 \end{cases}
 \end{equation}
 As usual, $|\nabla p_1|$ and $|\na u_\W^f|$ on the right-hand side of \eqref{eqn: first var term 1} are understood in the sense of non-tangential limits, with the former existing by proof of Lemma~\ref{l:gradhigher}.
 Next, letting $p_2$ solve 
 \[
 \begin{cases}
     -\Delta p_2 = u_{\W}^f - u_{B(x_{\W})}^f &  \text{ in } B({x_{\W}}), \\
  \ \ \ \ \  p_2 = 0 &  \text{ on } \partial B(x_{\W}), 
 \end{cases}
 \]
 we have that 
 \[
 \begin{aligned}
 \int \left(u_{\W}^f - u_{B(x_{\W})}^f \right)\dot{u}_{B(x_{\W})}^f 
 &= \int_{B(x_{\W})} \left(-\Delta p_2 \right)\dot{u}_{B(x_{\W})}^f \\
 &= \int_{\partial B(x_{\W})} |\nabla p_2| \, | \nabla u_{B(x_{\W})}^f| \, \dot{x}_\W \cdot \nu_{B(x_\W)} \, d\mathcal{H}^{n-1}
 =  \dot{x}_\W \cdot \int_{\partial B(x_{\W})} |\nabla p_2|\,  | \nabla u_{B(x_{\W})}^f| \, \nu_{B(x_\W)} d\mathcal{H}^{n-1}.
 \end{aligned}
 \]
So, by Lemma~\ref{lemma: Barycenter Derivative}, 
 \[
 \int \left(u_{\W}^f - u_{B(x_{\W})}^f \right)\dot{u}_{B(x_{\W})}^f 
 = \int_{\partial^* \W} V \cdot A_{\W} \, T \cdot \nu_{\W}\,  d \mathcal{H}^{n-1},
 \]
 where 
 \[
 V=\int_{\partial B(x_{\W})} |\nabla p_2| \, | \nabla u_{B(x_{\W})}^f| \, \nu_{B(x_\W)} \, d \mathcal{H}^{n-1}
 \]
 is a constant vector field that is independent of $T$ and has modulus bounded by a universal constant since  $\| f \|_{L^{\infty}} \leq 1$. We have thus shown
 \begin{equation}
     \label{eqn: nl term intermediate deriv comp}
  \frac{d}{dt}\Bigr|_{t=0} \int \left(u_{\W_t}^f - u_{B(x_{\W_t})}^f \right)^2 = 2\int_{\partial^* \W} \left(  |\nabla p_1| \, |\nabla  u_{\W}^f| - V\cdot A_\W \right)\, T \cdot \nu_\W \, d\mathcal{H}^{n-1} \,.
  \end{equation}
 We now employ the boundary Harnack principle Theorem \ref{t:bhpf} on $\partial \W$ to $|\na p_1|/|\na w_\W|$ and $|\na u^f_\W| /|\na w_\W|$ just as in Lemma~\ref{l:gradhigher} to find that $|\nabla p_1(x)|\, |\na u_\W^f(x)|=a_1(x)|\nabla w_{\W}(x)|^2$ where $a_1(x)$ H\"older continuous and $ \|a_1(x) \|_{C^{0,\alpha}(\partial^* \W)} \leq C$ for a constant $C$ depending only on $n$ and  $\| f\|_{L^{\infty}} \leq 1$.
\end{proof}

Putting the preceding lemmas together, we arrive at the following distributional form of the Euler-Lagrange equation:
\begin{theorem}  \label{t:vectoreulerlagrange}
  If $\Omega_t$ is measure preserving up to first order, i.e. if $\int_{\partial \W }T\cdot \nu_\W \, d\mathcal{H}^{n-1} = 0$, then 
  \begin{equation}
      \label{eqn: meas pres EL}
 -\frac{1}{2} \int_{\partial^* \W} |\nabla w_{\Omega}|^2 T \cdot \nu_{\W}
 +\tau \int_{\partial^* \W} C_{\Omega}(a_1(x)|\nabla w_{\W}|^2- V\cdot A_\Omega) \, T\cdot \nu_\W \, d \mathcal{H}^{n-1}(x)=0 
   \end{equation}
 for a constant $C_{\W} \in \R$ depending on $\W$ and on the parameter $a$ in the definition of $\nl(\, \cdot\, )$ with $|C_{\W}|\leq 2.$
Here $a_1$ and $V$ are, respectively, the H\"{o}lder continuous function and fixed vector from Lemma~\ref{l:distterm} and $A_\W$ is from Lemma~\ref{lemma: Barycenter Derivative}.

\end{theorem}

\begin{proof}
 From Lemma \ref{l:diff torsion} we have the derivative of the torsional rigidity term. To take the derivative of $\nl(\W)$ deinfed in \eqref{eqn: SP nl}, we utilize Lemma \ref{l:distterm} to obtain 
 \[
 \begin{aligned}
 \frac{d}{dt}\Bigr|_{t=0} \nl(\W_t)&= \frac{\beta^2_f(\W)-a}{\sqrt{(a^2+(a-\beta^2_f(\W))^2)}} \ \frac{d}{dt}\Bigr|_{t=0} \beta^2_f(\Omega_t) \\
 &=2 \frac{\beta^2_f(\W)-a}{\sqrt{(a^2+(a-\beta^2_f(\W))^2)}}
 \ \int_{\partial^* \W}  (a_1(x) |\nabla w_{\W}(x)|^2 - V \cdot A_{\W} )\, T \cdot \nu_{\W} \ d \mathcal{H}^{n-1}(x).
 \end{aligned}
 \]
  Letting $C_{\W} =2 \frac{\beta^2_f(\W)-a}{\sqrt{(a^2+(a-\beta^2_f(\W))^2)}}$, the proof  of \eqref{eqn: meas pres EL} follows. 
\end{proof}
We can use Theorem~\ref{t:vectoreulerlagrange} to get the following pointwise form of the Euler-Lagrange equation. In the following statement, $\bar{\tau}$ is the dimensional constant from Theorem~\ref{thm: existence 2}.
\begin{corollary} \label{c:pointeulerlagrange}
{There is a dimensional constant $\bar\tau_0\in (0, \bar{\tau}]$ such that, provided $\tau\in (0, \bar\tau_0]$,}  the torsion function $w_\W$ of $\W$ satisfies 
    \begin{equation} \label{e:rhsholder}
    |\nabla w_{\W}|^2=a_2(x),
    \end{equation}
    for $\mathcal{H}^{n-1}$ a.e. point of $\partial^* \W$,
    where $\|a_2\|_{C^{0,\alpha}(\partial \W)} \leq C$  with the constant $C$ depending on $n$ and $\| f\|_{L^{\infty}} \leq 1$, and with $a_2(x)$ satisfying $0< C_1^{-1} \leq a_2(x) \leq C_1$ with $C_1$ depending on $n$ through the Lipschitz and nondegeneracy estimates from Theorem~\ref{thm: existence 2}.   
\end{corollary}

\begin{proof}
 Fix $x_0,y_0 \in \partial^* \W$. Let $T_1$ be in the direction of the outward normal at $x_0$ and $T_2$ be in the direction of the inward unit normal at $y_0$. We take a standard approximation of Dirac by letting  $0\leq \psi \in C_0^{\infty}(B_1)$ with
 $\int \psi =1$ and defining $\psi_{\epsilon}(x):=\epsilon^{-n} \psi(x/\epsilon)$. We now define $\Omega_t$ as the image of $\Omega$ under the map
 \begin{equation}
     \label{eqn: variation VF}
 x \mapsto x + tT_1 \psi_{\epsilon}(x-x_0)+tT_2\psi_{\epsilon}(x-y_0). 
 \end{equation}
 We recall in the proof of Lemma \ref{lemma: Barycenter Derivative} that 
 \[
 \frac{d}{d t}\Big|_{t=0}|\Omega_t|  =\int_{\partial^* \Omega} T_1 \psi_{\epsilon}(x-x_0)+T_2\psi_{\epsilon}(x-y_0) \cdot \nu_{\Omega} \, d \mathcal{H}^{n-1}.
 \]
 We choose the lengths of $T_1,T_2$ depending on $\epsilon$ so that the above expression is zero and thus the vector field \eqref{eqn: variation VF} is measure preserving to first order. 
  Thus by Theorem \ref{t:vectoreulerlagrange} we obtain that 
  \[
  \begin{aligned}
 &-\frac{1}{2} \int_{\partial^* \W} |\nabla w_{\Omega}|^2 (\psi_{\epsilon}(x-x_0)T_1 +\psi_{\epsilon}(x-y_0)T_2)\cdot \nu_{\W} \ d \mathcal{H}^{n-1}(x) \\
 &+\tau \int_{\partial^* \W}  C_{\Omega} \, a_1(x)|\nabla w_{\W}|^2 (\psi_{\epsilon}(x-x_0)T_1 +\psi_{\epsilon}(x-y_0)T_2)\cdot \nu_{\W} \ d \mathcal{H}^{n-1}(x)\\
 &- \tau\int_{\partial^* \W}  C_{\Omega} V\cdot A_{\W} (\psi_{\epsilon}(x-x_0)T_1 +\psi_{\epsilon}(x-y_0)T_2)\cdot \nu_{\W} \ d \mathcal{H}^{n-1}(x) =0. 
 \end{aligned}
 \]
 Since $x_0,y_0 \in \partial^* \W$, then as $\epsilon \to 0$, we have that $T_1,T_2$ approach the outward and inward unit normals respectively at $x_0$ and $y_0$. Since $x_0,y_0$ are any two points on $\partial^* \W$, then we obtain for $\mathcal{H}^{n-1}$ a.e. point of $\partial^* \W$ that  
 \[
 \frac{1}{2}|\nabla w_{\W}|^2  -\tau C_\W a_1(x) |\nabla w_{\W}|^2+\tau C_\W \, V \cdot A_{\W}(x)  =L, 
 \]
 for some constant $L$. By choosing $\tau$ to be sufficiently small, then $L$ is bounded above and below by the Lipschitz and nondegeneracy estimates from Section \ref{sec: a priori and existence}. Factoring out $|\nabla w_{\W}|^2$ and moving everything to the right hand side we obtain \eqref{e:rhsholder}
\end{proof}

\subsection{Boundary regularity} \label{s:boundaryregularity}

If every free boundary point is a so-called ``flat" free boundary point, then  Corollary~\ref{c:pointeulerlagrange} can be used to prove that $\partial \W \in C^{1,\gamma}$. The next lemma will show that every free boundary point is ``flat". 

\begin{lemma} \label{l:flat}
 Fix $\epsilon>0$. There exists a constant $\bar{\tau}_1\in (0, \bar{\tau}_0]$ depending on $n$ and $\e$ such that if $\tau \leq \bar{\tau}_1$, then, up to a translation by $-x_\W$, the minimizer $\W$ satisfies 
 \[
 B_{1-\epsilon} \subset \W \subset B_{1+\epsilon}. 
 \]
\end{lemma}

\begin{proof}
 By choosing $\tau$ small enough, we have from \eqref{eqn: qual stab} that after translation 
 \[
  |B_1 \Delta \W| \leq \epsilon_1. 
 \]
 We will now improve the $L^1$ estimate above to an $L^{\infty}$ estimate. Assume $y \in \partial \W \setminus B_1$, and let $r=\text{dist}(y,B_1)$. From the density estimates \eqref{eqn: volume density} we have
 \[
 \epsilon_1 \geq |\W \cap B_r(y)| \geq c \omega_n r^n.
 \]
 Thus, we choose $\epsilon_1$ so that $\epsilon \leq (\epsilon_1 /(c\omega_n))^{1/n}$. If $y \in B_1$, then from the density estimates we have 
 \[
 \epsilon_1 \geq |B_1 \setminus \W|\geq (1-c) \omega_n r^n, 
 \]
 and the conclusion follows as before. 
\end{proof}

\begin{theorem} \label{t:c1alpha}
 Fix $\epsilon>0$ and let $\gamma$ be the H\"older constant from Theorem \ref{t:bhpf}. There exists a constant $\bar{\tau}_2 \in (0, \bar{\tau}_0]$ depending on $n$ and $\epsilon$ and a universal constant $C>0$ such that for $\tau \in (0,\bar{\tau}_2]$, after a translation by $-x_\W$,   $\partial \W$ can be parametrized by a $C^{1,\gamma}$ function $\phi$ over $\partial B$ with 
 \[
 \|\phi \|_{C^{1,\gamma}(\partial B)} \leq C \epsilon. 
 \]
\end{theorem}

\begin{proof}
    Let $x \in \partial \W$, and let $u$ be the harmonic replacement of $w_{\W}$ in $B_r(x)\cap \W$. From Theorem \ref{t:bhpf}, the quotient $w_{\W}/u \in C^{0,\gamma}$ up to the boundary. Since $w_{\W}$ and $u$ both vanish on $\partial \W \cap B_{r}(x)$, we have on the reduced boundary $\partial^* \W$
    \[
    \frac{w_{\W}}{u} = \frac{|\nabla w_{\W}|}{|\nabla u|}.
    \]
    Then $|\nabla w_{\W}|/|\nabla u| \in C^{0,\gamma}(\overline{B_{r/2}(x) \cap \W})$. Then  from Corollary \ref{c:pointeulerlagrange} we conclude
    \[
    |\nabla u(x)|=b(x) \quad \text{ on } \partial^* \W \cap B_{r/2}(x),
    \]
    with $b(x) \in C^{0,\gamma}$ and $b(x)>c$. 
     From Theorem \ref{t:bhpf}, the harmonic replacement $u$ also inherits the Lipschitz growth estimates from $w$ near $\partial \W$. Consequently, $u$ satisfies conditions $(1-3)$ in Definition 5.1 of \cite{AC81}. Since $|\nabla u| \in C^{0,\gamma}$, and since $B_{1-\epsilon} \subset \Omega \subset B_{1+\epsilon}$, (so that $\W$ is ``flat" at all points) it follows from Theorem 8.1 in \cite{AC81} that locally near $x$, that $\W$ can be parametrized by a $C^{1,\gamma}$ function $\phi$ over $\partial B$ with 
\[
 \| \phi \|_{C^{1,\gamma}(\partial B \cap B_{r/4}(x_0))} \leq C \epsilon.
\]
Using a covering argument, we then obtain that $\W$ is parametrized by $\phi: \partial B \to \mathbb{R}^n$ and 
\[
\| \phi \|_{C^{1,\gamma}(\partial B)} \leq C \epsilon.
\]
\end{proof}

We now utilize the $C^{1,\gamma}$ boundary regularity to show that $\W$ satisfies the volume constraint. 

\begin{lemma}[Volume constraint lemma]\label{lem: volume constraint lemma}\label{p:volume} Let $\bar{\tau}_2$ be fixed according to Theorem~\ref{t:c1alpha} with $\e = 1/2$ and fix $\tau \leq \bar{\tau}_1$. Then $|\W| = \omega_n$.
\end{lemma}

\begin{proof} Recall the parameters $\baryEps, \bar{\delta}$, and $\bar{\eta}$ fixed in \eqref{eqn: parameter fix} and \eqref{eqn: eta bar fix}.  Let $r>0$ be chosen so that $\W':= r\W$ has $|\W'|=\om_n$.
Thanks to \eqref{eqn: low base energy main} and Lemma~\ref{lemma: qual stability}, we know $||\W|-\om_n| \leq \om_n \baryEps/2$.
Moreover,  Theorem~\ref{t:c1alpha} in particular implies that $\W$ is star-shaped and thus we have either $\W' \subset\W$ (in the case that $|\W| \geq \omega_n $) or $\W' \supset \W$ (if $|\W| \leq \omega_n$). Hence, we can apply Lemma~\ref{lemma: key estimate}.  
 So, applying \eqref{eqn: nl Lip} followed by Lemma~\ref{lemma: key estimate} yields
    \begin{equation}\label{eqn: rescale nl compare}
 |\nl(\W) - \nl(\W')|  \leq \bar{C}_n\,\Big( |\W\Delta \W'| +\int |u_{\W}^f - u_{\W'}^f|\Big)\leq  \bar{C}_n
        \begin{cases}
           |\W| - \om_n+  \tor (\W') - \tor(\W) & \text{ if } |\W| \geq \om_n\\
      \om_n-   |\W| + \tor (\W) - \tor(\W') & \text{ if } |\W| \leq \om_n.
        \end{cases}
    \end{equation}
    Here $\bar{C}_n$ is the dimensional constant  from \eqref{eqn: nl Lip}.

We take $\W'$ as a competitor in the minimization problem \eqref{eqn: MAIN min prob}, and use \eqref{eqn: rescale nl compare} together with the scaling of the torsional rigidity and volume to show that $\Ep(\W') \leq \Ep(\W),$ with strict inequality unless $r=1$.

{\it Case 1:} Assume $|\W| \geq \omega_n$, so that $r\leq 1$.
Since $\W$ is  minimizer, we use \eqref{eqn: rescale nl compare} to find 
\begin{align*}
0 \leq \Ep(\W') - \Ep(\W) & = \tor (\W') - \tor(\W) -\frac{1}{\eta}\left(|\Omega|-\omega_n\right)+\tau\left(\nl\left(\Omega^{\prime}\right)-\nl(\Omega)\right) \\
& \leq \tor (\W') - \tor(\W)-\frac{1}{\eta}\left(|\Omega|-\omega_n\right)+\bar{C}_n\tau  \left(  |\W| -\om_n+  \tor (\W') - \tor(\W)\right) \\
& =\left(1+\bar{C}_n\tau  \right) (\tor (\W') - \tor(\W) )-\left(\frac{1}{\eta}-\bar{C}_n\tau \right)\left(|\Omega|-\omega_n\right) \\
 & =\left(1+ \bar{C}_n \tau \right)\left(r^{n+2}-1\right) \operatorname{tor}(\Omega)+\left(\frac{1}{\eta}- \bar{C}_n \tau\right)\left(\frac{r^n-1}{r^n}\right) \\
& =\left(1+ \bar{C}_n \tau\right)\left(1-r^{n+2}\right)|\operatorname{tor}(\Omega)|-\left(\frac{1}{\eta}- \bar{C}_n\tau\right)\left(\frac{1-r^n}{r^n}\right)\,.
\end{align*}
Since $\fv(\W) \geq0$ in this case, Lemma~\ref{prop: base summary} tells us that $\tor(\W) \geq \tor(B)$, i.e. $|\tor(\W)|\leq |\tor(B)|.$ Since we have chosen $\tau$ such that $\bar{C}_n \tau \leq \frac{1}{2} \leq \frac{1}{2\eta}$, the chain of inequalities above that if $r\neq 1$ (i.e. if $|\W|>\om_n$) then
\begin{equation}\label{eqn: scaling contra case 1}
    \frac{1}{4|\tor(B)|\, \eta}  \leq  \frac{r^n(1-r^{n+2})}{1-r^n}
\end{equation}
Keeping in mind that $r<1$, for the right-hand side we have
\begin{align*}
& \frac{r^n\left(1-r^{n+2}\right)}{1-r^n}=\frac{r^n(1-r)\left(r^{n+1}+\ldots+r+1\right)}{(1-r)\left(r^{n-1}+\ldots+r+1\right)}=\frac{r^n\left(r^{n+1}+\ldots r+1\right)}{\left(r^{n-1}+\ldots+r+1\right)} \\
&=r^n\left(1+\frac{r^{n+1}+r^n}{r^{n-1}+\ldots+1}\right) \leq 3 r^n \leq 3 .
\end{align*}
{Since we have chosen $\eta \leq 1/13|\tor(B)|$,} we reach a 
contradiction to \eqref{eqn: scaling contra case 1} and thus have $|\W| = \omega_n$.

{\it Case 2: } 
Now assume $|\W|\leq \w_n$, so that $r \geq 1.$ Since $\W$ is  minimizer, we use \eqref{eqn: rescale nl compare} to find 
\begin{align*}
 0 \leq \Ep(\W') -\Ep(\W) &= \tor(\W') - \tor(\W)-\eta\left(|\Omega|-\omega_n\right)+\tau(\nl(\W')-\nl(\Omega))\\
& \leq\tor(\Omega')-{\tor}(\Omega)+\eta\left(\omega_n-|\Omega|\right)+ \bar{C}_n \tau \left(\left(\omega_n-|\Omega|\right)+{\tor}(\Omega)-\tor({\Omega'})\right) \\
& =\left(1-\bar{C}_n \tau\right)(\operatorname{tor}({\Omega}')-\operatorname{tor}(\Omega))+\left(\eta+ \bar{C}_n \tau\right)\left\{\omega_n-|\Omega|\right\}\\
& ={\left(1- \bar{C}_n \tau \right)\left(r^{n+2}-1\right)} \operatorname{tor}(\Omega)+\left(\eta+\bar{C}_n\tau\right)\left(\frac{r^n-1}{r^n}\right) \omega_n\,.
\end{align*}
Now, by Lemma~\ref{prop: base summary}, $|\W| \geq \om_n/2$ and thus, {since $\W$ satisfies \eqref{eqn: low base energy main},}
$\tor(\W) - \eta \om_n/2 \leq \tor(\W) + \fv(|\W|) \leq \tor(B) + \bar{\delta}$, and so {since we have chosen 
$\eta$ small enough so that $\eta \om_n/2\leq \bar{\delta}$}, we have $|\tor(\W)| \geq |\tor(B)|-2\bar{\delta}.$ Recall that we have also chosen $\tau$ small enough so that $\bar{C}_n\tau \leq \eta$ and $1-\bar{C}_n\tau \geq 1/2$. So, if $r> 1,$ i.e. if $|\W| <\om_n$, then the inequalities above yield
\begin{equation}
1 \leq  \left(\frac{r^n({r^{n+2}-1})}{r^n-1}\right)   \leq \frac{4\eta \om_n}{|\tor(B)|-2\bar{\delta}}\,.
\end{equation}
{Since we have chosen $\eta \leq \frac{|\tor(B)|-2\bar{\delta}}{8\om_n}$} we reach a contradiction and conclude that $|\W| = \omega_n$.
\end{proof}

Since we only assume quantitative bounds on the  $L^{\infty}$ norm of the right hand side $f$, we will not be able to obtain $C^{2,\alpha}$  regularity of the boundary with estimates. 
However, we will now upgrade the boundary regularity to $C^{1,\alpha}$ for any $0<\alpha<1$, with estimates.
\begin{theorem} \label{t:1/2alpha}
 {Fix $\epsilon>0$ and let $0<\alpha<1$. There exists a constant $\bar{\tau}_2 \in (0, \bar{\tau}_0]$ depending on $n$ and $\epsilon$ and a universal constant $C>0$ such that for $\tau \in (0,\bar{\tau}_2]$, after a translation by $-x_\W$,   $\partial \W$ can be parametrized by a $C^{1,\alpha}$ function $\phi$ over $\partial B$ with 
 \begin{equation} \label{e:holderhigher}
 \|\phi \|_{C^{1,\alpha}(\partial B)} \leq C \epsilon. 
 \end{equation}}
\end{theorem}

\begin{proof}
 We utilize the recent technique (see \cite{dp23,ttv24,z24,aks25}) of considering the quotient $v=u_{\W}^f/w_{\W}$ which is a solution to a highly degenerate elliptic equation. Specifically, pointwise inside $\W$ the quotient $v$ satisfies 
 \[
 \text{div}\left(w_{\Omega}^2 \nabla v\right)= -w_{\Omega}f+u_{\Omega}^f.
\]
Since both $w_{\Omega}$ and $u_{\Omega}^f$ have Lipschitz growth away from the boundary, 
\[
|w_{\Omega}^2(x) \nabla v(x)| \leq C\text{dist}(x,\partial \Omega)
\]
which gives a zero Neumann condition on $\partial \W$. 
Notice that $-w_{\Omega}f+u_{\Omega}^f=\text{dist}(x,\partial \W) g$ with $g \in L^{\infty}$. 
Using locally that $\partial \Omega \in C^{1,\gamma}$, we flatten the boundary to obtain under the transformation the equation 
\[
 \begin{cases}
 &\int_{B^+_r(x)} x_n^2 \langle A(x) \nabla \tilde{v}, \nabla \phi \rangle = -
 \int_{B^+_r(x)} x_n g_1(x) \phi(x), \\
 & \lim_{x_n \to 0} x_n^2 |\nabla \tilde{v}(x',x_n)|=0.
 \end{cases}
\]
 The matrix $A$ is symmetric, uniformly elliptic and 
\[
 \| A(x)-I \|_{C^{0,\gamma}(B_r)} \leq \epsilon. 
\]
Taking the primitive from $x_n=0$ and factoring out an $x_n^2$, we have that $\partial_{x_n} (x_n^2 g_2(x',x_n))=x_n g_1(x',x_n)$
with $g_2 \in L^{\infty}$. Letting $G=(0,\ldots,0,x_n^2 g_2(x))$ we have 
\[
 \begin{cases}
 &\int_{B^+_r(x)} x_n^2 \langle A(x) \nabla \tilde{v}, \nabla \phi \rangle = 
 \int_{B^+_r(x)} x_n^2 g_2(x) \phi_n(x), \\
 & \lim_{x_n \to 0} x_n^2 |\nabla \tilde{v}(x',x_n)|=0.
 \end{cases}
\]
We may now apply a lower-order Schauder estimate. For instance, Theorem 1.3 (i) in \cite{stv21} applies. Since $A$ is uniformly $\gamma$-H\"older continuous, the estimates from Theorem 1.3(i) in \cite{stv21} are uniform. Then $\tilde{v}\in C^{0,\alpha}$ for any $\alpha<1$, and so $u_{\W}^f/w_{\W} = v \in C^{0,\alpha}$. Then the argument in Theorem \ref{t:c1alpha} as before implies $\partial \W \in C^{1,\alpha}$ for any $\alpha<1$. Using interpolation of H\"older spaces, we obtain \eqref{e:holderhigher}. 
\end{proof}

{To apply Lemma \ref{thm: spectral gap} we will need $\phi \in C^{2,\gamma}$, but we will not need estimates on the $\|\phi\|_{C^{2,\gamma}}$. This is accomplished by assuming $f$ is smooth. Indeed, all we will need is for $f$ to be H\"older continuous.} 
\begin{theorem} \label{t:c2alpha}
Fix $\epsilon>0$. There exist $\bar{\tau}_3 \in (0, \bar{\tau}_0]$ depending on $n$ and $\epsilon$ and universal constants $\alpha,C>0$ depending such that for $\tau \in (0,\bar{\tau}_2]$, after a translation by $-x_\W$,   $\partial \W$ can be parametrized by a $C^{2,\alpha}$ function $\phi$ over $\partial B$.  
\end{theorem}

\begin{proof}
  From Theorem \ref{t:1/2alpha} we have  $\partial \W \in C^{1,\alpha}$ and $\W$ is parametrized by $\phi: \partial B \to \mathbb{\R}^n$ with $\| \phi \|_{C^{1,\alpha}(\partial B)}\leq C\epsilon$. We now utilize that $f \in C^{0,\alpha}$ and not just bounded. From the higher order boundary Harnack principle in \cite{ds15}, since the right hand side $f$ is H\"older continuous, we have that the ratio $u_{\W}^f/w_{\W} \in C^{1,\alpha}$. Then from the Euler-Lagrange equation we have that 
\begin{equation} \label{e:higher}
|\nabla w_{\W}|^2 = g \in C^{1,\alpha} \quad \text{ on } \partial \W. 
\end{equation}
From the Hodograph transform method of Kinderlehrer-Nirenber-Spruck we have that $\partial \W \in C^{2,\alpha}$ with estimates depending on the $C^{1,\alpha}$ estimates of $\partial \W$. This theorem is stated as Theorem 8.4 in \cite{AC81} in the case with zero right hand side. Alternatively, by taking a harmonic replacement $u$ of $w$ in a local neighborhood of $\partial \W$, one can apply the boundary Harnack principle in \cite{ds15} to the quotient $w/u$ and conclude that since $w$ satisfies \eqref{e:higher}, then $u$ satisfies the same estimate in \eqref{e:higher}. Then one can apply Theorem 8.4 in \cite{AC81} directly to $u$ and conclude that $\partial \W \in C^{2,\alpha}$. 
\end{proof}

\begin{remark}
    \rm{The proof of Theorem 8.4 in \cite{AC81} which relies on the Hodograph transform will give that $\| \phi \|_{C^{2,\alpha}(\partial B)}\leq C\epsilon$; however, this will depend on the H\"older norm of $f$, so we do not provide the details. }
\end{remark}

\begin{remark}
    \rm{We cannot obtain regularity for $\partial \W$ beyond $C^{2,\alpha}$ even if we assume $f$ is smooth. The reason can be seen in the proof of Lemma~\ref{l:distterm} where we compute the first variation of $\beta^2(\Omega_t)$. There, the solution $p_1$ to the auxiliary PDE \eqref{eqn: aux PDE} is at most $C^{2,\alpha}$, and thus its corresponding contribution to the Euler-Lagrange equation is at most $C^{1,\alpha}$.}  
\end{remark}

\appendix

\section{Explanation of the proof of Lemma \ref{thm: spectral gap}}
 We provide here the explanation of the proof of Lemma \ref{thm: spectral gap}. The estimates and computations are nearly all contained in \cite{p24} where the author considers the more difficult situation of the first eigenvalue of a domain. The other difference is that we assume $\phi \in C^{2,\alpha}(\partial B)$ with $\| \phi \|_{C^{1,\alpha}}(\partial B) \leq \delta$ with  $\alpha \in (1/2,1)$, while in \cite{p24} the author only assumes that $\phi \in C^{1,\alpha}$ rather than $\phi \in C^{2,\alpha}$. We make the stronger assumption in order to avoid using a lower-order second variation term to prove continuity (in domain variations) of the second derivative as in \cite{p24}. Following the notation in \cite{p24} we extend $\phi$ as a function on $\partial B$ to all of $\mathbb{R}^n$ by letting $\phi(x):=\theta(x)\phi(x/|x|)$ with $\theta \equiv 1$ near $\partial B$ and such that $\phi$ is compactly supported and which is constant (close to $\partial B$) in directions normal to the ball. We define $\xi_{\phi}:=\phi(x)x$ and define 
\[
 \Omega_t := (\text{Id}+t\xi_{\phi})(B)\,.
\]
We let $e(t):=\tor(\Omega_t)$.  
For a $C^3$ domain $\Omega$ we have from Theorem 2.1, Remark 2.2, and equation $(23)$ all in \cite{dl19} that 
\begin{equation} \label{e:2ndvar}
\begin{aligned}
e''(t)&=\int_{\Omega_t} |\nabla v_t|^2 + \int_{\partial \Omega_t} (\partial_{\nu_t} w_{\Omega_t} + \frac{1}{2}H_t (\partial_{\nu_t} w_{\Omega_t})^2) (\xi_{\phi} \cdot \nu_t)^2\\
&\quad -\frac{1}{2}\int_{\partial \Omega_t}(\partial_{\nu_t}{w_{\Omega_t}})^2[\textbf{B}_t((\xi_{\phi})_{\tau_t},(\xi_{\phi})_{\tau_t, }) -2 \nabla_{\tau_t}(\xi_{\phi} \cdot \nu_t)(\xi_{\phi})_{\tau_t}]. 
\end{aligned}
\end{equation}
In the above identity, $\nu_t$ is the outward unit normal, $w_{\W_t}$ is the torsion function, $\textbf{B}_t$ is the second fundamental form, $H_t$ is the mean curvature, $\nabla_\tau$ is the tangential gradient, and $(\xi_{\phi})_{\tau_t}$ are the tangential components. Finally, $v_t$ is the solution to 
\[
\begin{cases}
 \Delta v_t =0 \quad \text{ in } \Omega_t, \\
 v= (\partial_{\nu_t} w_{\Omega_t}) \xi_\phi \quad \text{ on } \partial \Omega_t. 
\end{cases}
\]
An alternative second variation formula for domains satisfiying a uniform outer ball condition, thus in particular for $C^{2,\alpha}$ domains, is proven in \cite[Proposition 11]{Laurain}, which necessarily agrees on $C^3$ domains with \eqref{e:2ndvar}. 
Since \eqref{e:2ndvar} and the expression for the second variation in \cite[Proposition 11]{Laurain} are both preserved under $C^{2,\alpha}$ approximations, we have that \eqref{e:2ndvar} also holds for $C^{2,\alpha}$ domains. In \cite{p24}, the author utilizes a lower-order 2nd variation formula and is therefore able to bypass assuming $\phi \in C^{2,\alpha}$. Since we may obtain our main result through approximations, we may assume $f$ is H\"older so that $\phi \in C^{2,\alpha}$, see Theorem \ref{t:c2alpha}. Therefore, we do not provide here the lower-order analogous 2nd approximation formula of the torsion. As a drawback, we make the additional assumption in Lemma \ref{thm: spectral gap} that $\phi \in C^{2,\gamma}$.

Lemma \ref{thm: spectral gap} is proven in the following steps. First, in the main step, we claim 
\begin{equation} \label{e:torcont}
 |e''(t)-e''(0)|\leq \omega(\| \phi \|_{C^{1,\alpha}(\partial B)}) \| \phi \|^2_{H^{1/2}(\partial B)}. 
\end{equation}
 The identity \eqref{e:2ndvar} contains five terms. The last three terms are (up to constants) identical to the last three terms in equation $(48)$ of \cite{p24} which is the second variation formula for the first eigenvalue of a domain. Consequently, those three terms are handled exactly as in \cite{p24}. The first term in \eqref{e:2ndvar} is different from the first term in $(48)$ of \cite{p24}. However, the torsion term is actually simpler. Indeed, in \cite{p24} the author handles the first term by decomposing into two pieces one of which is harmonic and the author labels as $\textbf{H}_{\xi,\theta}$. Although our harmonic function $v_t$ has different boundary data, the boundary data has the same regularity as $\textbf{H}_{\xi,\theta}$, and thus bounding the term 
\[
\int_{\Omega_t} |\nabla v_t|^2 - \int_{B}|\nabla v_0|^2,
\]
can be handled as in \cite{p24} using the estimates in Lemmas 3.7 and 3.8 in \cite{p24}. 
The final remaining term (which does not appear in $(48)$ in \cite{p24}) necessitates bounding 
\[
\int_{\partial \W_t} (\partial_{\nu_t} w_{\W_t})(\xi_{\phi}\cdot \nu_t)^2 - \int_{\partial B} \partial_{\nu} w_{B}.
\]
This term is actually a lower-order term and the easiest of all to handle. Using the pullback notation as in \cite{p24}, we have 
\[
\begin{aligned}
&\int_{\partial \W_t} (\partial_{\nu_t} w_{\W_t})(\xi_{\phi}\cdot \nu_t)^2 - \int_{\partial B} \partial_{\nu} w_{B} \\
&=\int_{\partial B} \langle A_{\xi} \nabla \hat{w}_t,\nu \rangle (\hat{\xi}_{\phi} \cdot \hat{\nu}_t) \tilde{J}_{\xi}- \langle \nabla w_{B}, \nu \rangle,
\end{aligned}
\]
where $\text{div}(A_\xi \nabla \hat{v}_t)=-J_{\xi}$ and $A_\xi, J_{\xi}$ are defined as in $(27)$ in \cite{p24}, and with $\tilde{J}_{\xi}$ defined as in $(22)$ in \cite{p24}. Following the proof of Lemma 3.6 in \cite{p24}, we obtain $\|\hat{w}_t - w_{B}\|\leq C \| \xi \|_{C^{1,\alpha}(\mathbb{R}^n, \mathbb{R}^n)}$. The estimates in Lemma 3.7 in \cite{p24} give the bounds for all the other terms, so that by using the triangle inequality we obtain 
\[
\left|\int_{\partial \W_t} (\partial_{\nu_t} w_{\W_t})(\xi_{\phi}\cdot \nu_t)^2 - \int_{\partial B} \partial_{\nu} w_{B} \right| \leq C \| \xi \|_{C^{1,\alpha}(\mathbb{R}^n, \mathbb{R}^n)}. 
\]
Then we are able to conclude \eqref{e:torcont}.

In the second step, we use \eqref{e:torcont} and the spectral gap of \cite[Theorem 3.3]{BDV15} to conclude the proof. By Taylor's theorem and \eqref{e:torcont},
\begin{align*}
 \tor(\W)-\tor(B) & = e'(0) + \frac{1}{2} e''(0) + \int_0^1 (1-s) (e''(s) - e''(0))\,ds\\
 &= e'(0) + \frac{1}{2} e''(0) + \omega(\| \phi \|_{C^{1,\alpha}(\partial B)}) \| \phi \|^2_{H^{1/2}(\partial B)}. 
\end{align*}
A Taylor expansion using the fact that $|\Omega|=|B|$ shows that 
\[
\int_{\partial B}  \phi   = - \frac{n-1}{2} \int_{\partial_B }\phi^2 + \| \phi \|_{L^2}^2\, o(\| \phi\|_{C^0}).
\]
Then, using the first variation formula for the torsional rigidity (see, e.g. \cite[Lemma 2.7]{dl19}) followed by the fact that $|\na w_B| = 1/n$ on $\partial B$, we find 
\begin{equation}
    \label{eqn: eprime}
e'(0) = -\frac{1}{2}  \int_{\partial B}  |\nabla w_\W|^2 \phi
= -\frac{1}{2n^2}  \int_{\partial B}  \phi = \frac{n-1}{4n^2} \int_{\partial B} \phi^2  + \| \phi \|_{L^2}^2\, o(\e).
\end{equation}
Next, writing $e''(0))$ using the expression \eqref{e:2ndvar} above at $t=0$ and again using the fact that $\partial_\nu w_B= -1/n$ on $\partial B,$ we see that $v_0 = -H(\phi)/n$ for $v_0$ as defined above, where $H(\phi)$ is the harmonic extension of $\phi$ in $B$. So, since the mean curvature is $H_0 = (n-1),$ combining \eqref{eqn: eprime} and \eqref{e:2ndvar} yields
\[
2e' + e''(0) = 
\frac{1}{n^2}\left( \int_{B_1} |\na H(\phi)|^2\ ,dx - \int_{\partial B_1} \phi^2 \, d\mathcal{H}^{n-1}\right) + \|\phi\|_{L^2(\partial B)} o(\| \phi\|_{C^1(\partial B})
\]
The remainder of the proof is identical to \cite[Theorem 3.3]{BDV15}.

\bibliographystyle{amsplain}
\bibliography{quantitativeACFref}

\end{document}